\documentclass[11pt]{amsart}
\usepackage[a4paper,left=2.5cm,right=2.5cm,top=2.5cm,bottom=2.5cm]{geometry}
\usepackage{amsfonts,amssymb,amsmath,latexsym,amsthm}
\usepackage{tikz-cd}
\usepackage{color}
\usepackage{hyperref}
\usepackage{url}
\newtheorem{definition}{Definition}[section]
\newtheorem{theorem}{Theorem}[section]
\newtheorem{proposition}{Proposition}[section]
\newtheorem{remark}{Remark}[section]
\newtheorem{corollary}{Corollary}[section]
\newtheorem{lemma}{Lemma}[section]
\numberwithin{equation}{section}

\title[Construction of blow-up solutions]{Construction of type II blow-up solutions for the energy-critical wave equation with an inverse-square potential in dimension 3}
\author{Dinghan Wang}
\address{School of Mathematical Sciences\\ University of Science and Technology of China\\ Hefei 230026\\ Anhui\\ China}
\email{wangdinghan@mail.ustc.edu.cn}
\date{\today}

\begin{document}
	
	\begin{abstract}
		In this paper, we construct finite-time type-II blow-up solutions for the focusing energy-critical wave equation with an inverse-square potential
		$$\partial_t^2 u-\Delta u+\frac{\alpha}{|x|^2}u = u^5,$$
		with discussions of the influence of the potential being made. The key ingredients include an approxiamte construction of solutions and spectral analysis of the linearized operator, especially the distorted Hankel transform. The construction is based on the framework established by Krieger-Schlag-Tataru in 2009.
	\end{abstract}
	\maketitle
	\tableofcontents
	
	\section{Introduction}
	
	We consider the energy-critical wave equation with an inverse-square potential in spacial dimension $3$:
	\begin{equation}\label{eq:NLW}
		\partial_t^2 u - \Delta u + \frac{\alpha}{r^2} = u^5, \tag{NLWp}
	\end{equation}
	where $\frac{\alpha}{r^2}$ is the inverse-square potential and $\alpha>-\frac{1}{4}$ denotes the potential strength. Here we restrict the range of $\alpha$ to be $(-\frac{1}{4},+\infty)$ in order to ensure that $\mathcal{L}_\alpha:=-\Delta+\frac{\alpha}{r^2}$ is a positive operator, which enables us to use wave technics to study (\ref{eq:NLW}).  Here we interpret $\mathcal{L}_\alpha$ as the Friedrichs extension of the quadratic form defined on $C_0^\infty(\mathbb{R}^3\backslash\{0\})$ via
	\begin{equation*}
		u\mapsto \int_{\mathbb{R}^3}\left(|\nabla u|^2+\frac{\alpha}{r^2}u^2\right)dx,\quad u\in C_0^\infty(\mathbb{R}^3\backslash\{0\}).
	\end{equation*}
	We define the $\alpha$-Sobolev space $\dot{H}_\alpha^1(\mathbb{R}^3)$, that is the Sobolve spce associated with $\mathcal{L}_\alpha$, to have the squared norm
	\begin{equation*}
		\|u\|_{\dot{H}_\alpha^1(\mathbb{R}^3)}^2:=\left<\mathcal{L}_\alpha u,u\right>=\int_{\mathbb{R}^3}\left(|\nabla u|^2+\frac{\alpha}{r^2}u^2\right)dx,\quad u\in C_0^\infty(\mathbb{R}^3\backslash\{0\}).
	\end{equation*}
	We also denote by $\dot{H}_\alpha^1(\mathbb{R}^3)\times L^2(\mathbb{R}^3)$ the energy space and say $(u,\partial_t u)\in C(I,\dot{H}_\alpha^1(\mathbb{R}^3)\times L^2(\mathbb{R}^3))$ is a solution of (\ref{eq:NLW}) if the following Duhamel formula holds:
	\begin{equation*}
		u(t)=\cos(t\sqrt{\mathcal{L}_\alpha})u(0)+\frac{\sin(t\sqrt{\mathcal{L}_\alpha})}{\sqrt{\mathcal{L}_\alpha}}\partial_t u(0)+\int_{0}^{t}\frac{\sin((t-s)\sqrt{\mathcal{L}_\alpha})}{\sqrt{\mathcal{L}_\alpha}}u^5(s) ds,
	\end{equation*}
	where $\sqrt{\mathcal{L}_\alpha}$ is defined via $L^2$-functional calculus. If $\alpha>-\frac{1}{4}+\frac{1}{25}$, then (\ref{eq:NLW}) is locally well-posed, see \cite{MMZ20}.
	
	The solution of (\ref{eq:NLW}) admits a conserved energy
	\begin{equation*}
		\mathcal{E}[u](t):=\int_{\mathbb{R}^3}\left(\frac{1}{2}(\partial_t u)^2 + \frac{1}{2}|\nabla u|^2 + \frac{\alpha}{r^2} u^2 - \frac{u^6}{6}\right)dx.
	\end{equation*}
	As one may check, if $u$ is a solution to (\ref{eq:NLW}), then so is $u_\lambda$, where
	\begin{equation*}
		u_\lambda(t,x)=\lambda^{\frac{1}{2}}u(\lambda t,\lambda x),\quad \lambda>0,
	\end{equation*}
	and we have
	\begin{equation*}
		\mathcal{E}[u]=\mathcal{E}[u_\lambda].
	\end{equation*}
	It is the property of the scale-invarience of energy that we call (\ref{eq:NLW}) the energy-critical equation.
	
	\par As is shown in \cite{MZZ13}, (\ref{eq:NLW}) has a stationary solution, called the ground state soliton,
	\begin{equation*}
		W_\alpha(r)=(3\beta^2)^{\frac{1}{4}}\left(\frac{r^{\beta-1}}{1+r^{2\beta}}\right)^{\frac{1}{2}}, \quad r=|x|,
	\end{equation*}
	where $\beta:=\sqrt{1+4\alpha}>0$, and it is the unique (up to translations and scaling) positive radial solution of the nonlinear elliptic equation
	\begin{equation*}
		-\Delta u + \frac{\alpha}{r^2} u = u^5.
	\end{equation*}

	\par (\ref{eq:NLW}) can be viewed as a natural generalization of energy-critical wave equation, which corresponds to $\alpha=0$. One of the key features of $\mathcal{L}_\alpha$ is its scale-invarience: it is homogeneous of degree $-2$, due to which $\mathcal{L}_\alpha$ often appears as a scaling limit of non-homogeneous operators (for example, in strong resolvent sence). As a result, the Laplacian $\Delta$ and the potential $\frac{\alpha}{r^2}$ can be thought of equally strong at any length scale, which makes it impossible to use simple perturbative arguments. 
	
	\par The inverse-square potential arises frequently in many contents of mathematics and physics, where the scaling behavior may not hold everywhere, but at least in a certain region of the space, for example near the origin, or near the infinity, or both. For example, the Dirac equation with an Coulomb potential can be recast in a form of Klein-Gordon equation with an inverse-square potential, plus some other terms. It also appears in the linearized perturbations of classical space-time metrices to the Einstein equation in general relativity, such as Schwarzschild solution or Reissner–Nordström solution. Furthermore, $\mathcal{L}_\alpha$ is also related to the geometry of manifolds with conic sigularities. (\ref{eq:NLW}) is often to be thought as a model problem for more complicated settings of actual physics and geometric interest.
	
	\par In the defocusing case, that is to replace $u^5$ in the RHS of (\ref{eq:NLW}) by $-u^5$, \cite{MMZ20} shows the global existence and scattering. But for the focusing case we are considering now, finite-time blow-up may occur. According to the boundedness of the solution in energy space, there are two types of blow-up solutions. Type I solutions are those when approaching the blow-up time (say $0$), 
	\begin{equation*}
		\limsup\limits_{t\to 0^+}\|(u,\partial_t u)\|_{\dot{H}_\alpha^1(\mathbb{R}^3)\times L^2(\mathbb{R}^3)}=\infty.
	\end{equation*}
	Type II solutions are those whose energy norm remain bounded:
	\begin{equation*}
		\limsup\limits_{t\to 0^+}\|(u,\partial_t u)\|_{\dot{H}_\alpha^1(\mathbb{R}^3)\times L^2(\mathbb{R}^3)}<\infty.
	\end{equation*}
	In this paper, we consider the type II blow-up solutions.
	
	\par The soliton resolution conjecture is a well-known conjecture for general dispersive equations. In the radial context of energy-critical wave equations, it states that any radial type II finite-time blow-up solution of (\ref{eq:NLW}) decomposes into a finite sum of decoupled solitons and a radiation (we assume the blow-up time to be $0$):
	\begin{equation}\label{eq:sol_res}.
		u(t,r)=\sum_{j=1}^{N}\iota_j\lambda_j^{\frac{1}{2}}(t) W_\alpha(\lambda_j(t)r) + v(t,r) + o_{\dot{H}_\alpha^1\times L^2}(1),\quad t\to 0^+,\quad \iota_j\in\{\pm 1\}.
	\end{equation}
	For energy-critical wave equation, that is $\alpha=0$, the soliton resolution conjecture has been proved in the radial case through a series of works by Duyckaerts, Kenig, Merle, Jendrej and Lawrie, see \cite{DKM23}\cite{JL23} and references therein. For general $\alpha$, this was also proved by Li, Miao and Zhao in \cite{LMZ22} with a technical restriction on $\alpha$. The aim of this paper is to \textit{construct} a solution of the type predicted by the soliton resolution conjecture. We shall confine ourselves to the \textit{one bubble} case, that is $N=1$ in (\ref{eq:sol_res}).

	\par There are lots of works concerning the constrction of blow-up solutions for the focusing energy-critical wave equation. In \cite{KST09a}, Krieger, Schlag and Tataru constructed finite-time blow-up solutions in dimension $3$ with polynomial blow-up rate $\lambda(t)=t^{-1-\nu}$ for any $\nu>\frac{1}{2}$. Later, the bound of $\nu$ is relaxed to $\nu>0$ by Krieger and Schlag in \cite{KS14}, which contains all possible powers for a polynomial blow-up rate because of the dynamical condition $\lim_{t\to 0^+}t\lambda(t)=\infty$. Recently, the above result is generalized to dimension $4$ and $5$ by Samuelian in \cite{Sa24}, but it seems harder to further apply the same method to $6$ or higher dimensions. There are also results with prescirbed non-polynomial blow-up rate, for example, Donninger, Huang, Krieger and Schlag \cite{DHKS14} shows that the oscillatory blow-up rate $\lambda(t)=t^{-1-\nu(t)}$, where $\nu(t)=\nu+\varepsilon_0\frac{\sin(\log(t))}{\log(t)}$ is admissble in dimension $3$ for $|\varepsilon_0|\ll 1$ and $\nu>3$. Based on a different approach, Hillairet and Raphaël \cite{HR12} obtained $C^\infty$ blow-up solutions in dimension $4$ with blow-up rate $\lambda(t)=t^{-1}e^{\sqrt{-\log(t)}(1+o(1))}$. In \cite{Jen17}, Jendrej obtained blow-up solutions in dimension $5$ with $\lambda(t)=C t^{-4}+o(t^{-4})$ and $\lambda(t)=t^{-1-\nu}$ for $\nu>8$ with nondegenerate asymptotic profile.

	The following is the main result of this paper, which is in the same spirit of \cite{KST09a}.
	
	\begin{theorem}[main theorem]\label{thm:main}
		Let $\nu>0$ and $-\frac{1}{4}+\frac{1}{16}<\alpha<\frac{15}{4}$ with $\nu>\frac{1}{2\sqrt{1+4\alpha}}$. For any $\delta>0$, there exists an energy solution of (\ref{eq:NLW}), which blows up precisely at $r=t=0$ and which has the following property. In the light cone $|x|=r\leq t$ and for small times $t$, the solution has the form, with $\lambda(t)=t^{-1-\nu}$,
		\begin{equation*}
			u(t,x)=\lambda^{1/2}(t)W_\alpha(\lambda(t)x)+\eta(t,x),
		\end{equation*}
		where the local energy of $\eta(t,\cdot)$ tends to $0$ as $t\to0$, i.e.
		\begin{equation*}
			\mathcal{E}_{\mathrm{loc}}[\eta](t):=\int_{|x|<t}\left((\partial_t \eta)^2+|\nabla \eta|^2+\frac{\alpha}{r^2}u^2+\eta^6\right)dx\to 0,\quad t\to 0^+,
		\end{equation*}
		and outside the light cone, $u(t,x)$ satisfies
		\begin{equation*}
			\int_{|x|\geq t}\left(+|\partial_tu(t,x)|^2+|\nabla u(t,x)|^2+\frac{\alpha}{r^2}u^2+|u(t,x)|^6\right)dx<\delta
		\end{equation*}
		for all sufficiently small $t>0$. In particular, the energy of these blow-up solutions can be chosen arbitrary close to $E(W_\alpha,0)$, that is, the energy of the stationary solution. 
	\end{theorem}

	\begin{remark}[on the lower bound of $\alpha$]
		We remark that in this paper, we have to restrict $\alpha$ to be in $(-\frac{1}{4}+\frac{1}{16},\frac{15}{4})$. The additional $\frac{1}{16}$ in the lower bound is due to the equivalence of Sobolev spaces $H^s_\alpha(\mathbb{R}^3)=H^s(\mathbb{R}^3)$ for suitable $s=s(\beta)$. This lower bound may not be optimal, and by analyzing the nonlinearity more carefully, one may extend this result to $\alpha>-\frac{1}{4}+\frac{1}{25}$. However, it is not clear whether one can remove this $\frac{1}{25}$ in the lower bound, since the local well-posedness of (\ref{eq:NLW}) has only been proved in this range (see \cite{MMZ20}). Moreover, in proving the equivalence of $\alpha$-Sobolev space $H^{2s}_\alpha(\mathbb{R}^3)$ in physics space and the weighted $L^2$ space $L^{2,s}_{d\rho}$ in frequency sapce, we also need $\alpha>-\frac{1}{4}+\frac{1}{25}$ to use the Hardy's inequality associated with $\mathcal{L}_\alpha$. Nonetheless, this lower bound is not so important and less related to our propose. The only significant thing is that our result admits some \textbf{negative $\alpha$}.
	\end{remark}

	\begin{remark}[on the upper bound of $\alpha$]\label{rem:ubound}
		The upper bound $\alpha<\frac{15}{4}$ is the main assumption in this paper. We cannot construct the desired sequence of approximate solutions to (\ref{eq:NLW}) when $\alpha\geq\frac{15}{4}$ because of the fast growing mode of the linearized ellptic equation, and we have to impose this upper bound to make sure the first correction to be in $\dot{H}_\alpha^1(\mathbb{R}^3)$. Moreover, this upper bound also guarantees the boundedness of transference operator $[\mathcal{K},\xi\partial_\xi]$, see Section \ref{sec:trans}. 
		\par Actually, it is even not clear whether this kind of blow-up solution exists when $\alpha\geq\frac{15}{4}$. There is a rough analogy between $3$-dim (\ref{eq:NLW}) and $d$-dim energy-critical wave equation:
		\begin{equation*}
			4\alpha=(d-3)(d-1).
		\end{equation*}
		This identity follows from comparaing the linearzed operator and its spectral measure of this two equations. As a result, $\alpha\geq\frac{15}{4}$ corresponds exactly to $d\geq 6$, which may suggest that there might not exist finite-time type II blow-up solution for energy critical wave equation in dimension $d\geq 6$.
	\end{remark}

	\begin{remark}
		The blow-up solution we constructed is only of finite regularity $H^{1+\frac{\beta\nu}{2}-}\cap H^{1+\frac{\beta}{2}}\cap H^{\frac{3}{2}}$. This non-smoothness comes from the sigularity of the type $(1-r/t)^{\frac{1+\beta\nu}{2}}(\log(1-r/t))^n$ of our approxiamte solution across the light-cone, as well as restrictions from nonlinear estimate.
	\end{remark}

	\par Rather than using modulation theory, our construction is based on the framework established in \cite{KST09a}. Their method is proven to be robust and can be applied to many different settings. For example, based on the same framework, there are constructive results for wave maps \cite{KST08}, for Yang-Mills equation \cite{KST09b}, for energy-critical Schr\"odinger equation \cite{OP14}\cite{Sch23} and Schr\"odinger maps \cite{Per14}, and very recently for $4$ dimensional energy critical Zakharov system \cite{KS24a}\cite{KS24b}. See also \cite{DK13} for ``infinite-time blow-up" for energy-critical wave and \cite{BMG21} for blow-up construction for hyperbolic mean curvature flow.

	\par Now we succinctly explain the steps of the proof.
	\begin{itemize}
		\item[Step 1.] We construct a sequence of approximate solutions $u_k=u_{k-1}+v_k$ with $u_0(t,r)=\lambda^{\frac{1}{2}}(t)W_\alpha(\lambda(t)r)$ by adding corrections $v_k$. As $k$ increases, the fully nonlinear error of $u_{k}$, $e_k:=u_k^5-(\partial_t^2-\Delta +\frac{\alpha}{r^2})u_k$ has stronger time decay as $t\to 0^+$. This is called the renormaliztion step, and we will follow the iteration procedure introduced in \cite{KST09a}: we iterate between solving an elliptic equation and a wave equation with good scaling property using right ansatz.
		\item[Step 2.] We perturb around the renormalized profile $u_{k}$ for $k$ large enough. To use the Fourier method to solve this perturbed equation, we eliminate the time-dependent potential $5W_\alpha^4(\lambda(t)r)$ by changing to $\tau,R$ variables, which results in a ``dilated type" nonlinear wave equation, that is $\partial_\tau$ being replaced by $\partial_\tau+\frac{\lambda_{\tau}}{\lambda}(R\partial_R-1)$.
		\item[Step 3.] By analyzing the spectrum of the linearized opertor $\mathcal{L}-\partial_R^2+\frac{\alpha}{R^2}-5W_\alpha^4(R)$, including precise asymptotics of its spectral measure, Weyl-Titchmarsh solutions, etc., we establish the distorted Hankel transform which digonalizes $\mathcal{L}$.
		\item[Step 4.] After applying the distorted Hankel transform established above, we solve the final equation on the Fourier side. With a crucial observation in \cite{KS14} that this in-homogeneous problem can be solved explicitly (at least for non-negative frequencies), we can solve this equation in suitable weighted $L^2$ space. However, there are some \textit{linear} error terms involving the so-called transference operator $\mathcal{K}$ also appearing in the final equation, which is not clear whether they can be iterated away when solving the equation. Here we take the same trick as in \cite{KST09a}, that is by putting a huge weight on time in our function space, one can get the smallness of these linear error trivially. This is exactly the reason to do a renormalization procedure before doing perturbation. This is done in Section \ref{sec:trans} and Section \ref{sec:final}.
		\item[Step 5.] For nonlinear terms, we bound them in the $\alpha$-Sobolev spaces associated with $\mathcal{L}_\alpha$. We will mainly depend on the equivalence of Sobolev spaces between the one we are going to use and the usual one established in \cite{KMVZZ18}. Here comes the major restriction on the lower bound of $\alpha$, that is $\alpha>-\frac{1}{4}+\frac{1}{16}$.
		\item[Step 6.] Finally, we use the local well-posedness and finite speed of propagation to get a real solution in the whole space (the approximate solutions are only constructed inside the light-cone).
	\end{itemize}

	\par Here we summarize some major differences and novelties in our paper compared to \cite{KST09a}.
	\begin{itemize}
		\item[(i)] In the renormalization procedure, we can not keep the iteriated corrections and errors in finitely many $IS^m(R^k(\log(R))^\ell,\mathcal{Q}_\beta)$ spaces going to be introduced. Instead, when $0<\beta<1$ and $3<\beta<4$, the indices of these spaces, that is $k,m$ and $\ell$ would increases, but in a controllable manner (strictly less than $2$ order for $m$ at each step).
		\item[(ii)] Also in the renormalization procedure, we encounter a kind of power series of the form
		\begin{equation*}
			f(R)=R^\delta\sum_{i_1,\cdots,i_N=0}^{\infty}a_{i_1,\cdots,i_N}R^{i_1p_1+\cdots+i_Np_N},\quad R < R_0,
		\end{equation*}
		where $N\geq 1$ is a finite number and $p_1,\cdots,p_N$'s are finitely many ``base powers". This kind of power series shares many properties with ususal power series, but the point is that we have to carefully show that there are indeed only finitely many base powers.
		\item[(iii)] The approximate construction scheme fails for $\alpha\geq\frac{15}{4}$, due to a fast growing mode of the linearzied elliptic equation, which won't happen in the context of enegry-critical wave equation in dimension $d<6$.
		\item[(iv)] The spectrum of the linearized operator $\mathcal{L}=-\partial_R^2+\frac{\alpha}{R^2}-5W_\alpha^4(R)$ various with $\alpha$. In particular, $0$ will be become an eigenvalue instead of only being a resonance when $\beta>2$. Compared to the negative eigenfunction $\phi_d(R)$ which has exponential decay as $R\to\infty$, the eigenfunction $\phi_0(R)$ only has limited decay rate as $R\to \infty$, which caused more difficulties.
		\item[(v)] The perturbation $5W_\alpha^4(R)$ has a strong singularity at $R=0$, making various estimates on the boundedness of the transference operator $\mathcal{K}$ become extremely delicate. For example, the Schwartz kernel of $\mathcal{K}_{cc}$ doesn't have arbitrary off-diagonal decay anymore. For another example, small $\beta$ prevents us from proving the equivalence of higer order Sobolve spaces defined via $\mathcal{L}_\alpha$ and $\mathcal{L}=\mathcal{L}_\alpha-5W_\alpha^4(R)$. This turns out to be the most difficult part of this paper.
		\item[(vi)] For nonlinear estimates, we work on the $\alpha$-Sobolev spaces, that is the Sobolev spaces defined via $\mathcal{L}_\alpha$ instead of $-\Delta$. We have to depend on some recent results on this kind of Sobolev space, including Hardy's inequality associated with $\mathcal{L}_\alpha$, Littlewood-Paley theory and Strichartz estimates, etc.
	\end{itemize}

	The paper is organized as follows. In Section \ref{sec:appro}, we construct our approxiamate solution. The perturbed equation around the renormalization profile is derived in Section \ref{sec:perturbed}. In Section \ref{sec:spec}, we collect some spectral properties of the linearized operator, and deduce the precise asymptotics of the spectral measure. In Section \ref{sec:trans}, we introduce the transference operator and study its $L^p$ boundedness. In Section \ref{sec:final}, we solve the perturbed eqaution in the Fourier side while bounding the parametrix in some $t$-weighted space. In Section \ref{sec:nonlinear}, we establish the nonlinear estimate in $\alpha$-Sobolev spaces. Finally, we recall the well-posedness result and construct the real solution with desired bulk term, thus giving a proof of our main theorem in Section \ref{sec:proof}.

	For convenience of the reader, we collect some notations frequently used in the paper below.
	\begin{itemize}
		\item $\alpha>-\frac{1}{4}$ the potential strength, $\beta=\sqrt{1+4\alpha}>0$ 
		\item $W_\alpha$ the ground state soliton
		\item $\lambda(t)=t^{-1-\nu}$ the concentration rate with $\nu>0$
		\item $\mathcal{Q}_\beta$ function algebra, also $\mathcal{Q}_\beta'$
		\item $S^m(R^k(\log(R))^\ell)$, $IS^m(R^k(\log(R))^\ell,\mathcal{Q}_\beta)$, $IS^m(R^k(\log(R))^\ell,\mathcal{Q}_\beta')$ function spaces
		\item $\sigma(\mathcal{L})$, $\sigma_{ac}(\mathcal{L})$, $\sigma_{pp}(\mathcal{L})$, $\sigma_{ess}(\mathcal{L})$ the spectrum, absolutely continuous spectrum, point spectrum, essential spectrum of the operator $\mathcal{L}$
		\item $\mathcal{F}f$ or $\widehat{f}$ the distorted Hankel transfrom of $f$
		\item $L^{2,s}_\rho$, $L^{\infty,N}L^{2,s}_\rho$ weighted $L^2$ spaces
		\item $(x)_+=\max\{x,0\}$, $(x)_-=-(-x)_+$ for $x\in\mathbb{R}$
		\item $\left<u\right>=\sqrt{1+|u|^2}$ for $u\in\mathbb{C}$
		\item $A\lesssim B$ means $A\leq CB$ for some constant $C$, $C$ may vary from line to line
		\item $A\approx B$ means $A\lesssim B$ and $B\lesssim A$
	\end{itemize}

	\section{Construction of approximate solutions}\label{sec:appro}
	
	We want to construct a solution $u$ to (\ref{eq:NLW}) with $u_0(t,x)=\lambda^{\frac{1}{2}}(t)W_\alpha(\lambda(t)r)$ being the bulk term, where here and below, we fix $\lambda(t)=t^{-1-\nu}$ with $\nu>0$. More precisely, we will construct a radial solution $u$ of the form
	\begin{equation*}
		u(t,x) = u_0(t,x) + w(t,x),
	\end{equation*}
	with $w$ being small as $t\to 0^+$ in a suitable sense.
	\par To simplify the notation, let $R=\lambda(t)r$. Note that $W_\alpha(R)$ is smooth on $(0,+\infty)$ and has compactly uniformly and absolutely convergent expansions near $R=0$ and $R=+\infty$,
	\begin{equation*}
		\begin{aligned}
			W_\alpha(R) &= (3\beta^2)^{\frac{1}{4}}R^{\frac{\beta-1}{2}}\left(1-\frac{1}{2}R^{2\beta}+\cdots\right)=R^{\frac{\beta-1}{2}}g_1(R^{2\beta}),& 0<R<1,
			\\ &= (3\beta^2)^{\frac{1}{4}}R^{\frac{-\beta-1}{2}}\left(1-\frac{1}{2}R^{-2\beta}+\cdots\right)=R^{\frac{-\beta-1}{2}}g_2(R^{-2\beta}),& R>1,
		\end{aligned}
	\end{equation*}
	where $g_1,g_2$ are two real-analytic functions on $[0,1)$ with convergent radius being $1$ at $0$. This nature of $W_\alpha$ is invariant under applying $R\partial_R$, so we introduce the following function space that captures this kind of property.
	\begin{definition}
		Let $\left<\beta,|\beta-1|,|\beta-2|,|\beta-3|,4-\beta\right>=\mathbb{Z}_{\geq0}\cdot \beta+\mathbb{Z}_{\geq0}\cdot |\beta-1|+\mathbb{Z}_{\geq0}\cdot|\beta-2|+\mathbb{Z}_{\geq0}\cdot|\beta-3|+\mathbb{Z}_{\geq0}\cdot(4-\beta)$ denote the set of real numbers that are $\mathbb{Z}_{\geq0}$-linear combinations of $\beta$, $|\beta-1|$, $|\beta-2|$, $|\beta-3|$ and $4-\beta$. We also denote by $\{m_i\}_{i\geq0}$ an increasing arrangement of $\left<\beta,|\beta-1|,|\beta-2|,|\beta-3|,4-\beta\right>$, such that $0=m_0<m_1<m_2<\cdots$. Furthermore, let $I\geq1$ be the largest positive integer such that $m_I<2$.
	\end{definition}
	\begin{remark}
		We have for any $i,j\geq0$ that $m_i+m_j\geq m_{i+j}$.
	\end{remark}
	\begin{definition}
		Let $S^m(R^k(\log(R))^{\ell})$ denote the function space consisting of smooth functions $f$ on $(0,+\infty)$ with the following two properties:
		\par (1) $f$ admits an absolutely convergent expansion at $R=0$:
		\begin{equation*}
			f(R) = \sum_{i=0}^{\infty}\sum_{j=0}^{\infty}c_{i,j}R^{m+2i+2\beta j},\quad 0<R\ll 1;
		\end{equation*}
		\par (2) $f$ admits an absolutely convergent expansion at $R=\infty$:
		\begin{equation*}
		f(R) = \sum_{i=0}^{\infty}\sum_{j=0}^{\ell+i}d_{i,j}R^{k-m_i}(\log(R))^{j},\quad R\gg 1.
		\end{equation*}
	\end{definition}
	\begin{remark}
		Near $R=\infty$, $f\in S^m(R^k(\log(R))^{\ell})$ behaves like a power series, but with powers being $\mathbb{Z}_{\geq 0}$-combinations of finitely many ``base powers". This kind of power series shares many properties with the usual one, see Appendix \ref{app:A} for more discussion.
	\end{remark}
	\par The above discussion can be recapitulated as $W_\alpha\in S^{\frac{\beta-1}{2}}(R^{\frac{-\beta-1}{2}})$ and $S^m(R^k(\log(R))^\ell)$ is invariant under applying $R\partial_R$.
	\par We define $e_0$ to be the fully nonlinear error of $u_0$, that is 
	\begin{equation*}
		e_0(t,x) = u_0^5 - \left(\partial_t^2 u_0 - \Delta u_0 + \frac{\alpha}{r^2} u_0\right)= -\partial_t^2 u_0.
	\end{equation*}
	By a simple computation, one may check that 
	\begin{equation*}
		t^2e_0 = \lambda^{\frac{1}{2}}(t)\left(c_1 W_\alpha(R) + c_2 (R\partial_R)W_\alpha(R) + c_3 (R\partial_R)^2 W_\alpha(R)\right),
	\end{equation*}
	where $c_1,c_2$ and $c_3$ are constants containing $\nu$. Using the definition we just made, $t^2e_0\in \lambda^{\frac{1}{2}}(t)S^{\frac{\beta-1}{2}}(R^{\frac{-\beta-1}{2}})$.
	
	\par We now ingroduce the function algerbra used in the correction near the tip of the light-cone.

	\begin{definition}[admissible powers]
		We denote by $\mathcal{P}_\beta=\{p_1<p_2<\cdots\}$ the set of positive numbers of the form
		$$\sum_{k_0\geq 0,\mathrm{finite}}\frac{(\beta+4k_0\min\{\beta,1\})\nu+1}{2}+\cdots+\sum_{k_I\geq 0,\mathrm{finite}}\frac{(\beta+4k_I\min\{\beta,1\}+2m_I)\nu+1}{2}$$
		if $0<\beta\leq 2$, the set of numbers of the form
		$$\sum_{k_0\geq 0,\mathrm{finite}}\frac{(4-\beta+4k_0\min\{4-\beta,1\})\nu+1}{2}+\cdots+\sum_{k_I\geq 0,\mathrm{finite}}\frac{(4-\beta+4k_I\min\{4-\beta,1\}+2m_I)\nu+1}{2}$$if $2<\beta<4$.
	\end{definition}
	\begin{remark}
		We have for any $i\geq 1$ that $p_i>\frac{1}{2}$ since $\beta>0$ and $4-\beta>0$.
	\end{remark}
	\begin{remark}
		We have for any $i,j\geq 1$ that $p_i+p_j\geq p_{i+j}$.
	\end{remark}

	\begin{definition}
		We denote by $\mathcal{Q}_\beta$ the function space consisting of smooth functions $q$ on $(0,1)$ satisfying the following properties:
		\par (1) $q$ has an even expansion at $a=0$;
		\par (2) there are functions $q_{i,j}$ analytic at $a=1$, such that 
		\begin{equation*}
		q(a)=q_0(a)+\sum_{i=1}^{+\infty}(1-a)^{p_i}\sum_{j=0}^{\mathrm{finite}}q_{i,j}(a)\left(\log(1-a)\right)^{j},\quad a\approx 1,
		\end{equation*}
		where $p_i\in\mathcal{P}_\beta$.
	\end{definition}
	\begin{definition}
		We denote by $\mathcal{Q}_\beta'$ the function space consisting of smooth functions $q$ on $(0,1)$ satisfying the following properties:
		\par (1) $q$ has an even expansion at $a=0$;
		\par (2) there are functions $q_{i,j}$ analytic at $a=1$, such that 
		\begin{equation*}
		q(a)=q_0(a)+\sum_{i=1}^{+\infty}(1-a)^{p_i-1}\sum_{j=0}^{\mathrm{finite}}q_{i,j}(a)\left(\log(1-a)\right)^{j},\quad a\approx 1,
		\end{equation*}
		where $p_i\in\mathcal{P}_\beta$.
	\end{definition}
	\par It's not hard to see that $\mathcal{Q}_\beta$ is an algebra and $\mathcal{Q}_\beta\subset \mathcal{Q}_\beta'$.

	\begin{definition}
		We denote by $IS^m(R^k(\log(R))^\ell,\mathcal{Q}_\beta)$ the function space consisting of functions $w$ defined on the interior of the light cone, which are of the form
		\begin{equation*}
		w(t,r)=\sum_{i=1}^{\mathrm{finite}}q_i(a)f_i(R),
		\end{equation*}
		where $q_i\in\mathcal{Q}_\beta$ and $f_i\in S^m(R^k(\log(R))^\ell)$.
		\par The space $IS^m(R^k(\log(R))^\ell,\mathcal{Q}_\beta')$ is defined similarly by replacing $\mathcal{Q}_\beta$ with $\mathcal{Q}_\beta'$.
	\end{definition}

	\par Before starting the correction procedure, we first establish two ODE lemmas that will be frequently invoked.

	\begin{lemma}[correction in the cuspidal region]\label{lem:R}
		Let $V$ be the solution to the following equation
		$$\left(-\partial_R^2 -\frac{2}{R}\partial_R +\frac{\alpha}{R^2}-5W_\alpha^4(R)\right)V(R)=g(R)$$
		with zero Cauchy data at $R=0$. If $g\in S^{m}(R^k(\log(R))^\ell)$, and $m>\frac{\beta-5}{2}$, then we have $V\in S^{m+2}(\max\{R^{k+2}(\log(R))^{\ell+1}, R^{\frac{\beta-1}{2}}(\log(R))^{\ell}\})$.
	\end{lemma}
	
	\begin{proof}
		We first conjugate the operator to cancel the $\partial_R$ term. Let
		\begin{equation*}
		\widetilde{L}=R\left(-\partial_R^2 -\frac{2}{R}\partial_R +\frac{\alpha}{R^2}-5W_\alpha^4(R)\right)R^{-1}=-\partial_R^2 + \frac{\alpha}{R^2} - 5W_\alpha^4(R),
		\end{equation*}
		then we need to solve
		\begin{equation*}
		\widetilde{L}(RV(R))=Rg(R),
		\end{equation*}
		with zero Cauchy data at $R=0$.
		\par Since $\{\lambda^{\frac{1}{2}}W_\alpha(\lambda R)\}_{\lambda>0}$ is a family of solutions to the nonlinear problem
		\begin{equation*}
		\left(-\partial_R^2-\frac{2}{R}\partial_R+\frac{\alpha}{R^2}\right)u=u^5,
		\end{equation*}
		we get a solution $\phi(R)$ to $\widetilde{L}u=0$ by differentiating with parameter $\lambda$, that is 
		\begin{equation*}
		\phi(R)=R\frac{d}{d\lambda}\Big|_{\lambda=1}\lambda^{\frac{1}{2}}W_\alpha(\lambda R)=\frac{\beta R^{\frac{\beta+1}{2}}(1-R^{2\beta})}{(1+R^{2\beta})^{\frac{3}{2}}}.
		\end{equation*}
		With $\phi$ at hand, we get another solution $\theta$ to $\widetilde{L}u=0$,
		\begin{equation*}
		\theta(R)=\frac{R^{\frac{-\beta+1}{2}}(1-6R^{2\beta}+R^{4\beta})}{\beta^2(1+R^{2\beta})^{\frac{3}{2}}}
		\end{equation*}
		with Wronski consistancy
		\begin{equation*}
		W(\phi,\theta)(R) = 1.
		\end{equation*}
		Note that
		\begin{equation*}
		\phi\in S^{\frac{\beta+1}{2}}(R^{\frac{-\beta+1}{2}}),\quad \theta\in S^{\frac{-\beta+1}{2}}(R^{\frac{\beta+1}{2}}),
		\end{equation*}
		and we will use $\phi$ and $\theta$ as a fundamental system of solutions to $\widetilde{L}u=0$.
		\par Using the standard Green's function, we can explictly solve $V$ as
		\begin{equation}\label{eq:V_<}
		V(R)=R^{-1}\phi(R)\int_{0}^{R}\theta(R')R'g(R')dR'-R^{-1}\theta(R)\int_{0}^{R}\phi(R')R'g(R')dR'.
		\end{equation}
		\par We now determine the behavior of $V$ near $R=0$ and $R=+\infty$. For $0<R\ll 1$, we write $\phi(R)=R^{\frac{\beta+1}{2}}\phi_<(R^{2\beta})$ and  $\theta(R)=R^{\frac{-\beta+1}{2}}\theta_<(R^{2\beta})$ with $\phi_<,\ \theta_<$ being analytic at $0$. The first term in the RHS of (\ref{eq:V_<}) becomes
		\begin{equation*}
			\begin{aligned}
				R^{-1}\phi(R)\int_{0}^{R}\theta(R')R'g(R')dR'&=R^{\frac{\beta-1}{2}}\phi_<(R^{2\beta})\int_{0}^{R}\sum_{i,j,k=0}^{\infty}c_{ijk}R'^{\frac{-\beta+1}{2}+1+m+2i+2\beta j+2\beta k} dR'\\
				&=R^{\frac{\beta-1}{2}}\phi_<(R^{2\beta})\sum_{i,j,k=0}^{\infty}c_{ijk}'R^{\frac{-\beta+5}{2}+m+2i+2\beta j+2\beta k}\\
				&=\sum_{i,j=0}^{\infty}c_{ijk}''R^{m+2+2i+2\beta j}.\\
			\end{aligned}
		\end{equation*} 
		We have used the fact that $m>\frac{\beta-5}{2}$ hence there is no $\log$ term. Similarly, the second term in the RHS of (\ref{eq:V_<}) becomes
		\begin{equation*}
			\begin{aligned}
			R^{-1}\theta(R)\int_{0}^{R}\phi(R')R'g(R')dR'&=R^{\frac{-\beta-1}{2}}\theta_<(R^{2\beta})\int_{0}^{R}\sum_{i,j,k=0}^{\infty}d_{ijk}R'^{\frac{\beta+1}{2}+1+m+2i+2\beta j+2\beta k} dR'\\
			&=R^{\frac{-\beta-1}{2}}\theta_<(R^{2\beta})\sum_{i,j,k=0}^{\infty}d_{ijk}'R^{\frac{\beta+5}{2}+m+2i+2\beta j+2\beta k}\\
			&=\sum_{i,j=0}^{\infty}d_{ijk}''R^{m+2+2i+2\beta j}.\\
			\end{aligned}
		\end{equation*}
		\par To determine the large $R$ behavior, we write $\phi(R)=R^{\frac{-\beta+1}{2}}\phi_>(R^{-2\beta})$ and $\theta(R)=R^{\frac{\beta+1}{2}}\theta_>(R^{-2\beta})$ with $\phi_>,\ \theta_>$ being analytic at $0$. We also need to modify (\ref{eq:V_<}) to a form convenient for large $R$ expansion:
		\begin{equation}\label{eq:V_>}
			V(R) = c_1 R^{-1}\phi(R) + c_2 R^{-1}\theta(R) + R^{-1}\phi(R) \int_{R_0}^{R}\theta(R')R'g(R')dR' - R^{-1}\theta(R)\int_{R_0}^{R}\phi(R')R'g(R')dR',
		\end{equation}
		where $R_0\gg 1$ is a fixed constant. The third term in the RHS of (\ref{eq:V_>}) becomes
		\begin{equation*}
			\begin{aligned}
				&\quad R^{-1}\phi(R) \int_{R_0}^{R}\theta(R')R'g(R')dR'\\
				&=R^{\frac{-\beta-1}{2}}\phi_>(R^{-2\beta})\int_{R_0}^{R}\sum_{i=0}^{\infty}\sum_{j=0}^{\ell+i}\sum_{p=0}^{\infty}c_{ijp}R'^{\frac{\beta+1}{2}+1+k-m_i-2\beta p}(\log(R))^{j} dR'\\
				&=R^{\frac{-\beta-1}{2}}\phi_>(R^{-2\beta})\sum_{i'=0}^{\infty}\sum_{j=0}^{\ell+i'}\int_{R_0}^{R}c_{i'j}R'^{\frac{\beta+3}{2}+k-m_{i'}}(\log(R))^{j} dR'\\
				&=cR^{-1}\phi(R)+\sum_{i'=0}^{\infty}\sum_{j=0}^{\ell+i'}c_{i'j}'R^{k+2-m_{i'}}(\log(R))^{\ell+j}+R^{-1}\phi(R)\sum_{j=0}^{i_0'+\ell}c_j'(\log(R))^{j+1}\\
				&=cR^{-1}\phi(R)+\sum_{i'=0}^{\infty}\sum_{j=0}^{\ell+1+i'}c_{i'j}''R^{k+2-m_{i'}}(\log(R))^{j},\\
			\end{aligned}
		\end{equation*} 
		where $i_0'$ satisfies $m_{i_0'}=k+\frac{\beta+5}{2}$ (if such $i_0'$ does not exist, then $c_j'=0$ in the second-to-last line). Similarly, the fourth term in the RHS of (\ref{eq:V_>}) becomes
		\begin{equation*}
			\begin{aligned}
				&\quad R^{-1}\theta(R) \int_{R_0}^{R}\phi(R')R'g(R')dR'\\
				&=R^{\frac{\beta-1}{2}}\theta_>(R^{-2\beta})\int_{R_0}^{R}\sum_{i=0}^{\infty}\sum_{j=0}^{\ell+i}\sum_{p=0}^{\infty}d_{ijp}R'^{\frac{-\beta+1}{2}+1+k-m_i-2\beta p}(\log(R))^{j} dR'\\
				&=R^{\frac{\beta-1}{2}}\theta_>(R^{-2\beta})\sum_{i'=0}^{\infty}\sum_{j=0}^{\ell+i'}\int_{R_0}^{R}d_{i'j}R'^{\frac{-\beta+3}{2}+k-m_{i'}}(\log(R))^{j} dR'\\
				&=dR^{-1}\theta(R)+\sum_{i'=0}^{\infty}\sum_{j=0}^{\ell+i'}d_{i'j}'R^{k+2-m_{i'}}(\log(R))^{\ell+j}+R^{-1}\theta(R)\sum_{j=0}^{i_0'}d_j'(\log(R))^{j+1}\\
				&=dR^{-1}\theta(R)+\sum_{i'=0}^{\infty}\sum_{j=0}^{\ell+1+i'}d_{i'j}''R^{k+2-m_{i'}}(\log(R))^{j},\\
			\end{aligned}
		\end{equation*}
		where $i_0'$ satisfies $m_{i_0'}=k+\frac{-\beta+5}{2}$ (if such $i_0'$ does not exist, then $d_j'=0$ in the second-to-last line). As a result, when $R\gg 1$, $V(R)$ is of the form
		\begin{equation*}
			V(R) = C_1 R^{\frac{-\beta-1}{2}}\phi_>(R^{-2\beta}) + C_2 R^{\frac{\beta-1}{2}}\theta_>(R^{-2\beta}) + \sum_{i=0}^{\infty}\sum_{j=0}^{\ell+i}C_{ij}R^{k+2-m_i}(\log(R))^{j+1}.
		\end{equation*}
		\par In conclusion, we have verified that $V\in S^{m+2}(\max\{R^{k+2}(\log(R))^{\ell+1}, R^{\frac{\beta-1}{2}}(\log(R))^{\ell}\})$, which completes the proof.
	\end{proof}

	\begin{lemma}[correction near the tip of the light-cone]\label{lem:a}
		Let $\gamma>0$ be a constant. Let $g\in a^m\mathcal{Q}_{\beta}$ and $W$ be the solution of
		\begin{equation*}
			L_\gamma W(a) = g(a)
		\end{equation*}
		with zero Cauchy data at $a=0$, where
		\begin{equation*}
			L_{\gamma}=(a^2-1)\partial_a^2+\left((3-\gamma\nu)a-2a^{-1}\right)\partial_a+\left(\frac{\gamma\nu-1}{2}\frac{\gamma\nu-3}{2}+\alpha a^{-2}\right).
		\end{equation*}
		If $m>\frac{\beta-5}{2}$, and $\gamma\in \mathcal{P}_\beta$. Then we have $W\in a^{m+2}\mathcal{Q}_\beta$.
	\end{lemma}

	\begin{proof}
		We first analyze the behavior of $W$ near $a=1$. One may check that
		\begin{equation*}
		\phi_1(a)=1+\sum_{\ell=1}^{\infty}d_\ell(1-a)^\ell+c\phi_2(a)\log(1-a),\quad \phi_2(a)=(1-a)^{\frac{\gamma\nu+1}{2}}\left(1+\sum_{\ell=1}^\infty\widetilde{d}_\ell(1-a)^\ell\right)
		\end{equation*}
		is a fundamental system of solutions for $L_\gamma u=0$. Their Wronskian is given by
		\begin{equation*}
		W(\phi_1,\phi_2)(a)=k(1-a^2)^{\frac{\gamma\nu-1}{2}}a^{-2},\quad k\neq 0.
		\end{equation*}
		Using the standard Green's function, we get
		\begin{equation*}
		W(a)=c_1\phi_1(a)+c_2\phi_2(a)+\phi_1(a)\int_{a_0}^{a}\frac{\phi_2(a')g(a')}{W(\phi_1,\phi_2)(a')}da'-\phi_2(a)\int_{a_0}^{a}\frac{\phi_1(a')g(a')}{W(\phi_1,\phi_2)(a')}da'.
		\end{equation*}
		In the following, we will use $q(a)$ to denote a general function analytic at $a=1$.
		\begin{equation*}
			\begin{aligned}
				&\quad\phi_1(a)\int_{a_0}^{a}\frac{\phi_2(a')g(a')}{W(\phi_1,\phi_2)(a')}da'\\&=\left(q(a)+(1-a)^{\frac{\gamma\nu+1}{2}}\log(1-a)q(a)\right)\\&\quad\times\int_{a_0}^{a}\frac{(1-a')^{\frac{\gamma\nu+1}{2}}}{(1-a')^{\frac{\gamma\nu-1}{2}}}\sum_{i=1}^{\infty}(1-a')^{p_i-1}\sum_{j=0}^{\mathrm{finite}}q_{i,j}(a')(\log(1-a'))^{j}da'\\&=c\phi_1(a) + \sum_{i=1}^{\infty}(1-a)^{p_i+1}\sum_{j=0}^{\mathrm{finite}}q_{i,j}'(a)(\log(1-a))^j.\\
			\end{aligned}
		\end{equation*}
		\begin{equation*}
			\begin{aligned}
				&\quad\phi_2(a)\int_{a_0}^{a}\frac{\phi_1(a')g(a')}{W(\phi_1,\phi_2)(a')}da'\\&=(1-a)^{\frac{\gamma\nu+1}{2}}q(a)\\&\quad \times \int_{a_0}^{a}\frac{\left(q(a')+(1-a')^{\frac{\gamma\nu+1}{2}}\log(1-a')q(a')\right)}{(1-a')^{\frac{\gamma\nu-1}{2}}}\sum_{i=1}^{\mathrm{finite}}(1-a')^{p_i-1}\sum_{j=0}^{\infty}q_{i,j}(a')(\log(1-a'))^{j}da'\\&=d\phi_2(a)+\sum_{i=1}^{\infty}(1-a)^{p_i+1}\sum_{j=0}^{\mathrm{finite}}q_{i,j}'(a)(\log(1-a))^j.\\
			\end{aligned}
		\end{equation*}
		\par In conclusion, we get near $a=1$, $W$ is again of the form
		
		\begin{equation*}
		W(a)=\sum_{i=1}^{\infty}(1-a)^{p_i}\sum_{j=0}^{\mathrm{finite}}\widetilde{q}_{i,j}(a)(\log(1-a))^j.
		\end{equation*}

		\par Next, we analyze the behavior of $W$ near $a=0$. As one can check, we have a fundamental system of solutions $\psi_1,\psi_2$ of $L_\gamma u=0$ with the following expansions at $a=0$,
		\begin{equation*}
		\psi_1(a)=a^{\frac{\beta-1}{2}}\left(1+\sum_{\ell=1}^{\infty}e_\ell a^{2\ell}\right),\quad \psi_2(a)=a^{\frac{-\beta-1}{2}}\left(1+\sum_{\ell=1}^{\infty}\widetilde{e}_{\ell}a^{2\ell}\right).
		\end{equation*}
		Their Wronskian is given by
		\begin{equation*}
		W(\psi_1,\psi_2)(a)=k(1-a^2)^{\frac{\gamma\nu-1}{2}}a^{-2},\quad k\neq 0.
		\end{equation*}
		The variation of parameter formula yields
		\begin{equation*}
		W(a)=\psi_1(a)\int_{0}^{a}\frac{\psi_2(a')a'^{m}q(a')}{a'^{-2}}da'-\psi_2(a)\int_{0}^{a}\frac{\psi_1(a')a'^{m}q(a')}{a'^{-2}}da',
		\end{equation*}
		where in this part we use $q(a)$ to denote an arbitrary even analytic funtion which may vary from line to line. For the first term,
		\begin{equation*}
		\begin{aligned}
		\psi_1(a)\int_{0}^{a}\frac{\psi_2(a')a'^{m}q(a')}{a'^{-2}}da'&=a^{\frac{\beta-1}{2}}q(a)\int_{0}^{a}a'^{m+\frac{-\beta+3}{2}}q(a')da'\\&=a^{m+2}q(a)
		\end{aligned}
		\end{equation*}
		where we have used that $m>\frac{\beta-5}{2}$, so $m+\frac{-\beta+3}{2}>-1$, i.e. there are no $\log$ terms. For the second term,
		\begin{equation*}
		\begin{aligned}
		\psi_2(a)\int_{0}^{a}\frac{\psi_1(a')a'^{m}q(a')}{a'^{-2}}da'&=a^{\frac{-\beta-1}{2}}q(a)\int_{0}^{a}a'^{m+\frac{\beta+3}{2}}q(a')da'\\&=a^{m+2}q(a),
		\end{aligned}
		\end{equation*}
		again there are no extra $\log$ terms due to the fact that $m+\frac{\beta+3}{2}>\frac{\beta-5}{2}+\frac{\beta+3}{2}=\beta-1>-1$.
		\par In conclusion, we get near $a=0$ that
		\begin{equation*}
		W(a)=a^{m+2}q(a)
		\end{equation*}
		where $q$ is analytic and has an even expansion at $a=0$, which completes the proof.
	\end{proof}

	\subsection{The first correction}
	
	In order to make $u_0 + \varepsilon$ be an exact soluton, the radial correction term $\varepsilon$ shall satisfy
	\begin{equation*}
		\left(\partial_t^2 - \partial_r^2 - \frac{2}{r}\partial_r + \frac{\alpha}{r^2} - 5u_0^4\right)\varepsilon = N_1(\varepsilon) + e_0,
	\end{equation*}
	where $N_1(\varepsilon)$ denotes the nonlinear error appearing in the first correction, i.e.
	\begin{equation*}
		N_1(\varepsilon) = 10u_0^3\varepsilon^2 +10 u_0^2 \varepsilon^3 + 5u_0\varepsilon^4 + \varepsilon^5.
	\end{equation*}
	
	\par The first correction will be done in the cuspidal region $r\ll t$ and we assume that $\partial_t^2$ derivate is less important (in fact we include it into the error $e_1$), so we seek a correction term $v_1$ to be a solution to (note that the nonlinear term $N_1(\varepsilon)$ is omitted)
	\begin{equation*}
		\left(- \partial_r^2 - \frac{2}{r}\partial_r + \frac{\alpha}{r^2} - 5u_0^4\right)v_1 = e_0.
	\end{equation*}
	By the homogeneity consideration, we make the ansatz that $v_1(t,x)=\frac{\lambda^{\frac{1}{2}}(t)}{(t\lambda(t))^2}V_1(R)$. Together with the relation $t\partial_t=(-1-\nu)R\partial_R$ and $r\partial_r=R\partial_R$, we get the equation for $V_1$ as
	\begin{equation*}
		\left(-\partial_R^2-\frac{2}{R}\partial_R+\frac{\alpha}{R^2}-5W_\alpha^4(R)\right)V_1(R)=g(R),
	\end{equation*}
	where
	\begin{equation*}
		g(R)=t^2\lambda^{-\frac{1}{2}}(t) e_0(t,r)\in S^{\frac{\beta-1}{2}}(R^{\frac{-\beta-1}{2}}).
	\end{equation*}
	Imposing zero Cauchy data of $V_1(R)$ at $R=0$, Lemma \ref{lem:R} shows that $V_1 \in S^{\frac{\beta+3}{2}}(\max\{R^{\frac{-\beta+3}{2}}\log(R),R^{\frac{\beta-1}{2}}\})$. This splits into three cases:
	\begin{equation*}
		\begin{aligned}
			v_1 &\in \frac{\lambda^{\frac{1}{2}}(t)}{(t\lambda(t))^2}S^{\frac{\beta+3}{2}}(R^{\frac{-\beta+3}{2}}\log(R)),& 0<\beta<2;\\
			v_1 &\in \frac{\lambda^{\frac{1}{2}}(t)}{(t\lambda(t))^2}S^{\frac{5}{2}}(R^{\frac{1}{2}}\log(R)),& \beta=2;\\
			v_1 &\in \frac{\lambda^{\frac{1}{2}}(t)}{(t\lambda(t))^2}S^{\frac{\beta+3}{2}}(R^{\frac{\beta-1}{2}}),& 2<\beta<4.\\
		\end{aligned}
	\end{equation*}

	\par Now we let 
	\begin{equation*}
		u_1=u_0+v_1.
	\end{equation*}
	Due to a gaining of the factor $\frac{1}{(t\lambda(t))^2}=t^{2\nu}$, $v_1$ is much smaller than $u_1$ in terms of energy inside the light cone as $t\to 0^+$.

	\subsection{The error from $u_1$}
	
	Now we compute the fully nonlinear error $e_1$ of $u_1$.
	
	\begin{equation*}
		\begin{aligned}
			e_1(t,x) &= u_1^5 - \left(\partial_t^2 u_1 - \Delta u_1 + \frac{\alpha}{r^2} u_1\right)\\&= -\partial_t^2v_1+N_1(v_1).\\
		\end{aligned}
	\end{equation*}
	Using the relation $t\partial_t=(-1-\nu)R\partial_R$ and $t^2\partial_t^2=(t\partial_t)^2-t\partial_t$, we immediately get
	\begin{equation*}
		-t^2\partial_t^2v_1\in \frac{\lambda^{\frac{1}{2}}(t)}{(t\lambda(t))^2}S^{\frac{\beta+3}{2}}(\max\{R^{\frac{-\beta+3}{2}}\log(R),R^{\frac{\beta-1}{2}}\}).
	\end{equation*}
	\par For terms in $N_1(v_1)$, we only present the calculation for two ``extremal" terms $t^2u_0^3v_1^2$ and $t^2v_1^5$, since intermediate terms are similar.
	
	\begin{equation*}
		\begin{aligned}
			t^2u_0^3v_1^2&\in t^2\lambda^{\frac{3}{2}}(t)S^{\frac{3\beta-3}{2}}(R^{\frac{-3\beta-3}{2}})\frac{\lambda(t)}{(t\lambda(t))^{4}}S^{\beta+3}(\max\{R^{-\beta+3}(\log(R))^2,R^{\beta-1}\})\\
			&\subset\frac{\lambda^{\frac{1}{2}}(t)}{(t\lambda(t))^2}S^{\frac{5\beta+3}{2}}(\max\{R^{\frac{-5\beta+3}{2}}(\log(R))^2,R^{\frac{-\beta-5}{2}}\})\\
			&\subset\frac{\lambda^{\frac{1}{2}}(t)}{(t\lambda(t))^2} S^{\frac{5\beta+3}{2}}(\max\{R^{-2\beta}R^{\frac{-\beta+3}{2}}(\log(R))^2,R^{-\beta-2}R^{\frac{\beta-1}{2}}\})\\
			&\subset\frac{\lambda^{\frac{1}{2}}(t)}{(t\lambda(t))^2} S^{\frac{\beta+3}{2}}(\max\{R^{\frac{-\beta+3}{2}}(\log(R))^2,R^{\frac{\beta-1}{2}}\}).\\
		\end{aligned}
	\end{equation*}
	
	\begin{equation*}
		\begin{aligned}
			t^2v_1^5&\in t^2\frac{\lambda^{\frac{5}{2}}(t)}{(t\lambda(t))^{10}}S^{\frac{5\beta+15}{2}}(\max\{R^{\frac{-5\beta+15}{2}}(\log(R))^5,R^{\frac{5\beta-5}{2}}\})\\
			&\subset\frac{\lambda^{\frac{1}{2}}(t)}{(t\lambda(t))^2}b^6S^{\frac{5\beta+15}{2}}(\max\{R^{-2\beta+6}R^{\frac{-\beta+3}{2}}(\log(R))^5,R^{2\beta-2}R^{\frac{\beta-1}{2}}\})\\
			&\subset\frac{\lambda^{\frac{1}{2}}(t)}{(t\lambda(t))^2}a^6S^{\frac{9\beta+3}{2}}(\max\{R^{-2\beta}R^{\frac{-\beta+3}{2}}(\log(R))^5,R^{2\beta-8}R^{\frac{\beta-1}{2}}\})\\
			&\subset\frac{\lambda^{\frac{1}{2}}(t)}{(t\lambda(t))^2}a^6 S^{\frac{\beta+3}{2}}(\max\{R^{\frac{-\beta+3}{2}}(\log(R))^5,R^{\frac{\beta-1}{2}}\}).\\
		\end{aligned}
	\end{equation*}
	\par As a conclusion, we get
	\begin{equation*}
		t^2e_1\in\frac{\lambda^{\frac{1}{2}}(t)}{(t\lambda(t))^2}IS^{\frac{\beta+3}{2}}(\max\{R^{\frac{-\beta+3}{2}}(\log(R))^5,R^{\frac{\beta-1}{2}}\},\mathcal{Q}_\beta),
	\end{equation*}
	which again splits into three cases:
	\begin{equation*}
		\begin{aligned}
			t^2e_1&\in \frac{\lambda^{\frac{1}{2}}(t)}{(t\lambda(t))^2}IS^{\frac{\beta+3}{2}}(R^{\frac{-\beta+3}{2}}(\log(R))^5,\mathcal{Q}_\beta),& 0<\beta<2;\\
			t^2e_1&\in \frac{\lambda^{\frac{1}{2}}(t)}{(t\lambda(t))^2}IS^{\frac{5}{2}}(R^{\frac{1}{2}}(\log(R))^5,\mathcal{Q}_2),& \beta=2;\\
			t^2e_1&\in \frac{\lambda^{\frac{1}{2}}(t)}{(t\lambda(t))^2}IS^{\frac{\beta+3}{2}}(R^{\frac{\beta-1}{2}},\mathcal{Q}_\beta),& 2<\beta<4.\\
		\end{aligned}
	\end{equation*}

	\subsection{The second correction}
	
	We do the second correction near the boundary of the light cone, i.e. at $r\approx t$. We will use the self-similar variable $a=\frac{r}{t}$. The second correction $v_2$ is determined via
	\begin{equation*}
		\left(\partial_t^2-\partial_r^2-\frac{2}{r}\partial_r+\frac{\alpha}{r^2}\right)v_2=tf_1,
	\end{equation*}
	where $f_1$ is the main asymptotic component of $e_1$, that is
	\begin{equation*}
		\begin{aligned}
			t^2f_1&=\frac{\lambda^{\frac{1}{2}}(t)}{(t\lambda(t))^2}\sum_{i=0}^{I}\sum_{j=0}^{i+5}q_{i,j}(a)R^{\frac{-\beta+3}{2}-m_i}\left(\log(R)\right)^{j}\\&=\frac{\lambda^{\frac{1}{2}}(t)}{(t\lambda(t))^{\frac{\beta+1}{2}}}\sum_{i=0}^{I}b^{m_j}\sum_{j=0}^{i+5}a^{\frac{-\beta+3}{2}-m_i}q_{i,j}(a)\left(\log(R)\right)^{j},\quad 0<\beta\leq 2,\\
		\end{aligned}
	\end{equation*}
	or
	\begin{equation*}
		\begin{aligned}
			t^2f_1&=\frac{\lambda^{\frac{1}{2}}(t)}{(t\lambda(t))^2}\sum_{i=0}^{I}\sum_{j=0}^{i}q_{i,j}(a)R^{\frac{\beta-1}{2}-m_i}\left(\log(R)\right)^{j}\\&=\frac{\lambda^{\frac{1}{2}}(t)}{(t\lambda(t))^{\frac{5-\beta}{2}}}\sum_{i=0}^{I}b^{m_j}\sum_{j=0}^{i}a^{\frac{\beta-1}{2}-m_i}q_{i,j}(a)\left(\log(R)\right)^{j},\quad 2<\beta<4,\\
		\end{aligned}
	\end{equation*}
	where $I$ is the largest integer such that $m_I<2$ (note that $I$ depends on $\beta$). 
	\par By the homogeneity consideration, we make the following ansatz for $v_2$:
	\begin{equation*}
	\begin{aligned}
		v_2&=\frac{\lambda^{\frac{1}{2}}(t)}{(t\lambda(t))^{\frac{\beta+1}{2}}}\sum_{i=0}^{I}b^{m_i}\sum_{j=0}^{i+5}W_{i,j}(a)\left(\log(R)\right)^j,& 0<\beta\leq 2,\\
		v_2&=\frac{\lambda^{\frac{1}{2}}(t)}{(t\lambda(t))^{\frac{5-\beta}{2}}}\sum_{i=0}^{I}b^{m_i}\sum_{j=0}^{i}W_{i,j}(a)\left(\log(R)\right)^j,& 2<\beta<4.\\
	\end{aligned}
	\end{equation*}
	For each $i=1,2,\cdots,I$, we solve 
	\begin{equation*}
	\begin{aligned}
		t^2\left(\partial_t^2-\partial_r^2-\frac{2}{r}\partial_r+\frac{\alpha}{r^2}\right)&\left(\frac{\lambda^{\frac{1}{2}}(t)}{(t\lambda(t))^{\frac{\beta+1}{2}+m_i}}\sum_{j=0}^{i+5}W_{i,j}(a)\left(\log(R)\right)^{j}\right)\\
		&=\frac{\lambda^{\frac{1}{2}}(t)}{(t\lambda(t))^{\frac{\beta+1}{2}+m_i}}\sum_{j=0}^{i+5}a^{\frac{-\beta+3}{2}-m_i}q_{i,j}(a)\left(\log(R)\right)^{j},\quad 0<\beta\leq 2,
	\end{aligned}
	\end{equation*}
	or 
	\begin{equation*}
	\begin{aligned}
		t^2\left(\partial_t^2-\partial_r^2-\frac{2}{r}\partial_r+\frac{\alpha}{r^2}\right)&\left(\frac{\lambda^{\frac{1}{2}}(t)}{(t\lambda(t))^{\frac{5-\beta}{2}+m_i}}\sum_{j=0}^{i}W_{i,j}(a)\left(\log(R)\right)^{j}\right)\\&=\frac{\lambda^{\frac{1}{2}}(t)}{(t\lambda(t))^{\frac{5-\beta}{2}+m_i}}\sum_{j=0}^{i}a^{\frac{\beta-1}{2}-m_i}q_{i,j}(a)\left(\log(R)\right)^{j},\quad 2<\beta\leq 4,
	\end{aligned}
	\end{equation*}
	with zero Cauchy data at $a=0$. Equations for $a$ are obtained by matching the powers of $\log(R)$,
	\begin{equation*}
		\begin{aligned}
			L_{\beta+2m_i} W_{i,j}&=a^{\frac{-\beta+3}{2}-m_i}q_{i,j}(a)+\cdots,& 0<\beta\leq 2,\\
			L_{4-\beta+2m_i}W_{i,j}&=a^{\frac{\beta-1}{2}-m_i}q_{i,j}(a)+\cdots,& 2<\beta<4,\\
		\end{aligned}
	\end{equation*}
	where $\cdots$ contains $W_{i,j'}$ for $j'>j$ and their derivatives, but its explicit expression is not so relevant.

	\par Using Lemma \ref{lem:a}, we get $W_{i,j}\in a^{\frac{-\beta+7}{2}-m_j}\mathcal{Q}_\beta$ if $0<\beta\leq 2$, $W_{i,j}\in a^{\frac{\beta+3}{2}}\mathcal{Q}_\beta$ if $2<\beta<4$ and 
	\begin{equation*}
		\begin{aligned}
			v_2&=\frac{\lambda^{\frac{1}{2}}(t)}{(t\lambda(t))^4}\sum_{i=0}^{I}R^{\frac{-\beta+7}{2}-m_i}\sum_{j=0}^{i+5}\underset{\in\mathcal{Q}_\beta}{\underbrace{a^{m_i-\frac{-\beta+7}{2}}W_{i,j}(a)}}\left(\log(R)\right)^j,& 0<\beta\leq 2,\\
			v_2&=\frac{\lambda^{\frac{1}{2}}(t)}{(t\lambda(t))^4}\sum_{i=0}^{I}R^{\frac{\beta+3}{2}-m_i}\sum_{j=0}^{i}\underset{\in\mathcal{Q}_\beta}{\underbrace{a^{m_i-\frac{\beta+3}{2}}W_{i,j}(a)}}\left(\log(R)\right)^j,& 2<\beta<4.\\
		\end{aligned}
	\end{equation*}
	We further modify $v_2$ to be
	\begin{equation*}
		\begin{aligned}
			v_2&=\frac{\lambda^{\frac{1}{2}}(t)}{(t\lambda(t))^4}\sum_{i=0}^{I}R^{\frac{-\beta+7}{2}-m_i}\cdot\frac{R^{\beta+m_i}}{(1+R^{2})^{\frac{\beta+m_i}{2}}}\sum_{j=0}^{i+5}\underset{\in\mathcal{Q}_\beta}{\underbrace{a^{m_i-\frac{-\beta+7}{2}}W_{i,j}(a)}}\left(\frac{1}{2}\log(1+R^{2})\right)^j,& 0<\beta\leq 2,\\
			v_2&=\frac{\lambda^{\frac{1}{2}}(t)}{(t\lambda(t))^4}\sum_{i=0}^{I}R^{\frac{\beta+3}{2}-m_i}\cdot\frac{R^{2+m_i}}{(1+R^{2})^{\frac{2+m_i}{2}}}\sum_{j=0}^{i}\underset{\in\mathcal{Q}_\beta}{\underbrace{a^{m_i-\frac{\beta+3}{2}}W_{i,j}(a)}}\left(\frac{1}{2}\log(1+R^{2})\right)^j,& 2<\beta<4.\\
		\end{aligned}
	\end{equation*}
	
	\par In conclusion, we have
	\begin{equation*}
		v_2\in\frac{\lambda^{\frac{1}{2}}(t)}{(t\lambda(t))^4}IS^{\frac{\beta+7}{2}}(\max\{R^{\frac{-\beta+7}{2}}(\log(R))^5,R^{\frac{\beta+3}{2}}\},\mathcal{Q}_\beta).
	\end{equation*}

	\par Now we let
	\begin{equation*}
		u_2=u_1+v_2=u_0+v_1+v_2.
	\end{equation*}

	\subsection{The error from $u_2$}
	
	The fully nonlinear error $e_2$ of $u_2$ is given by
	\begin{equation*}
	\begin{aligned}
		t^2e_2&=t^2u_2^5-t^2\left(\partial_t^2-\partial_r^2-\frac{2}{r}\partial_r+\frac{\alpha}{r^2}\right)u_2\\&=\left(t^2e_1-t^2e_1^0\right)+t^2\left(e_1^0-\left(\partial_t^2-\partial_r^2-\frac{2}{r}\partial_r+\frac{\alpha}{r^2}\right)v_2\right)+t^2N_2(v_2),
	\end{aligned}
	\end{equation*}
	where
	\begin{equation*}
		\begin{aligned}
			t^2e_1^0&=\frac{\lambda^{\frac{1}{2}}(t)}{(t\lambda(t))^2}\sum_{i=0}^{I}\sum_{j=0}^{i+5}q_{i,j}(a)R^{\frac{-\beta+3}{2}-m_i}\times\frac{R^{\beta+m_i}}{(1+R^{2})^{\frac{\beta+m_i}{2}}}\left(\frac{1}{2}\log\left(1+R^{2}\right)\right)^{j},& 0<\beta\leq 2,\\
			t^2e_1^0&=\frac{\lambda^{\frac{1}{2}}(t)}{(t\lambda(t))^2}\sum_{i=0}^{I}\sum_{j=0}^{i}q_{i,j}(a)R^{\frac{\beta-1}{2}-m_i}\times\frac{R^{2+m_i}}{(1+R^{2})^{\frac{2+m_i}{2}}}\left(\frac{1}{2}\log\left(1+R^{2}\right)\right)^{j},& 2<\beta<4.\\
		\end{aligned}
	\end{equation*}
	Note that
	\begin{equation*}
		\begin{aligned}
			R^{\frac{-\beta+3}{2}-m_i}&\times\frac{R^{\beta+m_i}}{(1+R^{2})^{\frac{\beta+m_i}{2}}}\left(\frac{1}{2}\log\left(1+R^{2}\right)\right)^{j}\\&=R^{\frac{\beta+3}{2}-m_i}\left(\log(R)\right)^{j}\left(c_0+c_1R^{-2}+\cdots\right),& 0<\beta\leq 2,\\
			R^{\frac{\beta-1}{2}-m_i}&\times\frac{R^{2+m_i}}{(1+R^{2})^{\frac{2+m_i}{2}}}\left(\frac{1}{2}\log\left(1+R^{2}\right)\right)^{j}\\&=R^{\frac{\beta+3}{2}-m_i}\left(\log(R)\right)^{j}\left(c_0+c_1R^{-2}+\cdots\right),& 2<\beta<4,\\
		\end{aligned}
	\end{equation*}
	so
	\begin{equation*}
		t^2e_1-t^2e_1^0\in\frac{\lambda^{\frac{1}{2}}(t)}{(t\lambda(t))^2} IS^{\frac{\beta+3}{2}}(\max\{R^{\frac{-\beta-1}{2}}(\log(R))^{J+2},R^{\frac{\beta-5}{2}}(\log(R))^{J+1}\},\mathcal{Q}_\beta).
	\end{equation*}
	
	 \par For nonlinear terms, we only present the calculation for $t^2u_1^4v_2$ and $t^2v_2^5$. Note that 
	 \begin{equation*}
	 	u_1 \in \lambda^{\frac{1}{2}}(t)IS^{\frac{\beta-1}{2}}(\max\{R^{\frac{-\beta-1}{2}}\log(R),R^{\frac{\beta-5}{2}}\},\mathcal{Q}_\beta),
	 \end{equation*}
	 hence
	 \begin{equation*}
	 \begin{aligned}
	 t^2u_1^4v_2&\in t^2\lambda^{2}(t)IS^{2\beta-2}(\max\{R^{-2\beta-2}(\log(R))^4,R^{2\beta-10}\},\mathcal{Q}_\beta)\\&\quad \times\frac{\lambda^{\frac{1}{2}}(t)}{(t\lambda(t))^4}IS^{\frac{\beta+7}{2}}(\max\{R^{\frac{-\beta+7}{2}}\log(R),R^{\frac{\beta+3}{2}}\},\mathcal{Q}_\beta)\\
	 &\subset  \frac{\lambda^{\frac{1}{2}}(t)}{(t\lambda(t))^{2}}IS^{\frac{5\beta+3}{2}}(\max\{R^{\frac{-5\beta+3}{2}}(\log(R))^5,R^{\frac{5\beta-17}{2}}\},\mathcal{Q}_\beta)\\
	 &\subset \frac{\lambda^{\frac{1}{2}}(t)}{(t\lambda(t))^{2}}IS^{\frac{\beta+3}{2}}(\max\{R^{\frac{-\beta-1}{2}+2(1-\beta)}(\log(R))^5,R^{\frac{\beta-5}{2}+2(\beta-3)}\},\mathcal{Q}_\beta).\\
	 \end{aligned}
	 \end{equation*}
	 Likewise, we have
	\begin{equation*}
		\begin{aligned}
			t^2v_2^5&\in \frac{t^2\lambda^{\frac{5}{2}}(t)}{(t\lambda(t))^{10}}IS^{\frac{5\beta+15}{2}}(\max\{R^{\frac{-5\beta+15}{2}}(\log(R))^5,R^{\frac{5\beta-5}{2}}\},\mathcal{Q}_\beta)\\
			&\subset  \frac{\lambda^{\frac{1}{2}}(t)}{(t\lambda(t))^{2}}b^{6}IS^{\frac{5\beta+15}{2}}(\max\{R^{-2\beta+8}R^{\frac{-\beta-1}{2}}(\log(R))^5,R^{2\beta}R^{\frac{\beta-5}{2}}\},\mathcal{Q}_\beta)\\
			&\subset \frac{\lambda^{\frac{1}{2}}(t)}{(t\lambda(t))^{2}}a^6IS^{\frac{5\beta+3}{2}}(\max\{R^{\frac{-\beta-1}{2}+2(1-\beta)}(\log(R))^5,R^{\frac{\beta-5}{2}+2(\beta-3)}\},\mathcal{Q}_\beta)\\
			&\subset\frac{\lambda^{\frac{1}{2}}(t)}{(t\lambda(t))^2}IS^{\frac{\beta+3}{2}}(\max\{R^{\frac{-\beta-1}{2}+2(1-\beta)}(\log(R))^5,R^{\frac{\beta-5}{2}+2(\beta-3)}\},\mathcal{Q}_\beta).\\
		\end{aligned}
	\end{equation*}
	Note that an increasing of order at large $R$ of order $2(1-\beta)_+$ or $2(\beta-3)_+$ is admissile since we have restricted $\beta$ to be in $(0,4)$. The other three terms in $t^2N_2(v_2)$ can be treated similarly.
	\par In conclusion, we have
	\begin{equation*}
			t^2e_2\in\frac{\lambda^{\frac{1}{2}}(t)}{(t\lambda(t))^2}IS^{\frac{\beta+3}{2}}(\max\{R^{\frac{-\beta-1}{2}+2(1-\beta)_+}(\log(R))^{\max\{J+2,5\}},R^{\frac{\beta-5}{2}+2(\beta-3)_+}(\log(R))^{J+1}\},\mathcal{Q}_\beta).
	\end{equation*}

	\subsection{Higher iterations}
	We prove the following by induction.

	\begin{equation}\label{eq:v_1}
		\begin{aligned}
			v_{2k-1}&\in\frac{\lambda^{\frac{1}{2}}(t)}{(t\lambda(t))^{2k}}\\&\times IS^{\frac{\beta+3}{2}+2(k-1)}(\max\{R^{\frac{-\beta+3}{2}+2(k-1)(1-\beta)_+}(\log(R))^{m_k},R^{\frac{\beta-1}{2}+2(k-1)(\beta-3)_+}(\log(R))^{m_k'}\},\mathcal{Q}_\beta),\\
		\end{aligned}
	\end{equation}
	\begin{equation}\label{eq:e_1}
		\begin{aligned}
			t^2e_{2k-1}&\in\frac{\lambda^{\frac{1}{2}}(t)}{(t\lambda(t))^{2k}}\\&\times IS^{\frac{\beta+3}{2}+2(k-1)}(\max\{R^{\frac{-\beta+3}{2}+2(k-1)(1-\beta)_+}(\log(R))^{p_k},R^{\frac{\beta-1}{2}+2(k-1)(\beta-3)_+}(\log(R))^{p_k'}\},\mathcal{Q}_\beta'),\\
		\end{aligned}
	\end{equation}
	\begin{equation}\label{eq:v_2}
		\begin{aligned}
			v_{2k}&\in\frac{\lambda^{\frac{1}{2}}(t)}{(t\lambda(t))^{2k+2}}\\&\times IS^{\frac{\beta+7}{2}+2(k-1)}(\max\{R^{\frac{-\beta+7}{2}+2(k-1)(1-\beta)_+}(\log(R))^{p_k},R^{\frac{\beta+3}{2}+2(k-1)(\beta-3)_+}(\log(R))^{p_k'}\},\mathcal{Q}_\beta),\\
		\end{aligned}
	\end{equation}
	\begin{equation}\label{eq:e_2}
		\begin{aligned}
			t^2e_{2k}&\in \frac{\lambda^{\frac{1}{2}}(t)}{(t\lambda(t))^{2k}}IS^{\frac{\beta-1}{2}+2k}(\max\{R^{\frac{-\beta-1}{2}+2k(1-\beta)_+}(\log(R))^{q_k},R^{\frac{\beta-5}{2}+2k(\beta-3)_+}(\log(R))^{q_k'}\},\mathcal{Q}_\beta)\\
			&+\frac{\lambda^{\frac{1}{2}}(t)}{(t\lambda(t))^{2k+2}}IS^{\frac{\beta+3}{2}+2k}(\max\{R^{\frac{-\beta+3}{2}+2k(1-\beta)_+}(\log(R))^{q_k},R^{\frac{\beta-1}{2}+2k(\beta-3)_+}(\log(R))^{q_k'}\},\mathcal{Q}_\beta').\\
		\end{aligned}
	\end{equation}

	\par \underline{\textbf{Step 0}:} We check (\ref{eq:e_1}) for $k=0$.
	\par This follows directly from $t^2 e_0=-t^2\partial_t^2u_0\in\lambda^{\frac{1}{2}}(t)IS^{\frac{\beta-1}{2}}(R^{\frac{-\beta-1}{2}},\mathcal{Q}_\beta)$. Also note that when $2\leq\beta<4$, $\lambda^{\frac{1}{2}}(t)IS^{\frac{\beta-1}{2}}(R^{\frac{-\beta-1}{2}},\mathcal{Q}_\beta)\subset \lambda^{\frac{1}{2}}(t)IS^{\frac{\beta-1}{2}}(R^{\frac{\beta-5}{2}},\mathcal{Q}_\beta)$.
	
	\par \underline{\textbf{Step 1}:} We show that (\ref{eq:v_1}) holds assuming $e_{2k-2}$ satisfies (\ref{eq:e_2}).
	
	\par In this step, we only take the leading part of $t^2e_{2k-2}$ into account. We set
	\begin{equation*}
		t^2e_{2k-2}=t^2e_{2k-2}^0+t^2e_{2k-2}^1,
	\end{equation*}
	where 
	\begin{equation*}
		\begin{aligned}
			&\quad t^2e_{2k-2}^0\\&\in \frac{\lambda^{\frac{1}{2}}(t)}{(t\lambda(t))^{2k-2}}\\&\times IS^{\frac{\beta-1}{2}+2(k-1)}(\max\{R^{\frac{-\beta-1}{2}+2(k-1)(1-\beta)_+}(\log(R))^{q_{k-1}},R^{\frac{\beta-5}{2}+2(k-1)(\beta-3)_+}(\log(R))^{q_{k-1}'}\},\mathcal{Q}_\beta),\\
		\end{aligned}
	\end{equation*}
	and 
	\begin{equation*}
		\begin{aligned}
			&\quad t^2e_{2k-2}^1\\&\in\frac{\lambda^{\frac{1}{2}}(t)}{(t\lambda(t))^{2k}}\\&\times IS^{\frac{\beta+3}{2}+2(k-1)}(\max\{R^{\frac{-\beta+3}{2}+2(k-1)(1-\beta)_+}(\log(R))^{q_{k-1}},R^{\frac{\beta-1}{2}+2(k-1)(\beta-3)_+}(\log(R))^{q_{k-1}'}\},\mathcal{Q}_\beta').\\
		\end{aligned}
	\end{equation*}
	Note that $t^2e_{2k-2}^1$ can be included in $t^2e_{2k-1}$, see (\ref{eq:e_1}). In the following, we treat $a$ as a parameter and make the following ansatz for $v_{2k-1}$: 
	\begin{equation*}
		v_{2k-1}(t,r)=\frac{\lambda^{\frac{1}{2}}(t)}{(t\lambda(t))^{2k}}V_{2k-1}(R),
	\end{equation*}	
	then $V_{2k-1}$ satisfies the equation
	\begin{equation*}
		\left(-\partial_R^2-\frac{2}{R}\partial_R +\frac{\alpha}{R^2} -5W_\alpha^4(R)\right)V_{2k-1}=t\lambda^{-\frac{1}{2}}(t)(t\lambda(t))^{2k-2}e_{2k-2}^0=:g(R),
	\end{equation*}
	where (recall that we have frozen the $a$ variable)
	\begin{equation*}
		g(R)\in S^{\frac{\beta-1}{2}+2(k-1)}(\max\{R^{\frac{-\beta-1}{2}+2(k-1)(1-\beta)_+}(\log(R))^{q_{k-1}},R^{\frac{\beta-5}{2}+2(k-1)(\beta-3)_+}(\log(R))^{q_{k-1}'}\}).
	\end{equation*}

	\par Invoking Lemma \ref{lem:R}, we see that
	\begin{equation*}
		V_{2k-1}\in S^{\frac{\beta+3}{2}+2(k-1)}(\max\{R^{\frac{-\beta+3}{2}+2(k-1)(1-\beta)_+}(\log(R))^{q_{k-1}+1},R^{\frac{\beta-1}{2}+2(k-1)(\beta-3)_+}(\log(R))^{q_{k-1}'+1}\}).
	\end{equation*}
	
	Let $m_k=q_{k-1}+1$ and $m_k'=q_{k-1}'+1$, thus we have
	\begin{equation*}
		\begin{aligned}
			&\quad v_{2k-1}\\&\in\frac{\lambda^{\frac{1}{2}}(t)}{(t\lambda(t))^{2k}}\\&\times IS^{\frac{\beta+3}{2}+2(k-1)}(\max\{R^{\frac{-\beta+3}{2}+2(k-1)(1-\beta)_+}(\log(R))^{m_k},R^{\frac{\beta-1}{2}+2(k-1)(\beta-3)_+}(\log(R))^{m_k'}\},\mathcal{Q}_\beta),\\
		\end{aligned}
	\end{equation*}
	which concludes the proof of (\ref{eq:v_1}).

	\underline{\textbf{Step 2}:} We begin with (\ref{eq:v_1}) then deduce (\ref{eq:e_1}).
	\par The error $t^2e_{2k-1}$ is given by
	\begin{equation*}
		e_{2k-1}=e_{2k-2}^1+N_{2k-1}(v_{2k-1})+E^tv_{2k-1}+E^av_{2k-1},
	\end{equation*}	
	where
	\begin{equation*}
		N_{2k-1}(v_{2k-1})=5(u_{2k-2}^4-u_0^4)v_{2k-1}+10u_{2k-2}^3v_{2k-1}^2+10u_{2k-2}^2v_{2k-1}^3+5u_{2k-2}v_{2k-1}^4+v_{2k-1}^5,
	\end{equation*} 
	$E^tv_{2k-1}$ contains terms in $-\partial_t^2v_{2k-1}$ where no derivative applies to $a$, while $E^av_{2k-1}$ contains those terms in $\left(-\partial_t^2+\partial_r^2+\frac{2}{r}\partial_r-\frac{\alpha}{r^2}\right)v_{2k-1}$ where at least one derivative applies to $a$ (recall that in Step 1 the parameter $a$ was frozen).
	\par First, using $t\partial_t=(-1-\nu)R\partial_R$ and $t^2\partial_t^2=(t\partial_t)^2-t\partial_t$, we find that
	\begin{equation*}
		\begin{aligned}
			&\quad t^2E^tv_{2k-1}\\&\in \frac{\lambda^{\frac{1}{2}}(t)}{(t\lambda(t))^{2k}}\\& \times IS^{\frac{\beta+3}{2}+2(k-1)}(\max\{R^{\frac{-\beta+3}{2}+2(k-1)(1-\beta)_+}(\log(R))^{m_k},R^{\frac{\beta-1}{2}+2(k-1)(\beta-3)_+}(\log(R))^{m_k'}\},\mathcal{Q}_\beta).\\
		\end{aligned}
	\end{equation*}
	\par Then we consider $E^av_{2k-1}$, which can be computed to be
	\begin{equation*}
	\begin{aligned}
		&\quad t^2E^av_{2k-1}\\&=
		(a\partial_a)^2V_{2k-1}+2\left(4k\nu-\frac{1+\nu}{2}\right)(a\partial_a)V_{2k-1}+2(-1-\nu)(R\partial_R)(a\partial_a)V_{2k-1}+(a\partial_a)V_{2k-1}\\
		&\quad -\partial_a^2V_{2k-1}-2(R\partial_R)a^{-1}\partial_aV_{2k-1}\\
		&\quad -2a^{-1}\partial_aV_{2k-1}.\\
	\end{aligned}
	\end{equation*}
	Note that 
	\begin{equation*}
		(1-a^2)\partial_a^2,\ a\partial_a,\ a^{-1}\partial_a\colon \mathcal{Q}_\beta\to\mathcal{Q}_\beta',
	\end{equation*}
	therefore
	\begin{equation*}
		\begin{aligned}
			&\quad t^2E^av_{2k-1}\\&\in\frac{\lambda^{\frac{1}{2}}(t)}{(t\lambda(t))^{2k}}\\& \times IS^{\frac{\beta+3}{2}+2(k-1)}(\max\{R^{\frac{-\beta+3}{2}+2(k-1)(1-\beta)_+}(\log(R))^{m_k},R^{\frac{\beta-1}{2}+2(k-1)(\beta-3)_+}(\log(R))^{m_k'}\},\mathcal{Q}_\beta').\\
		\end{aligned}
	\end{equation*}
	\par Finally, we consider terms in $N_{2k-1}(v_{2k-1})$. Since
	\begin{equation*}
		u_{2k-2}=u_0+v_1+\cdots+v_{2k-2},
	\end{equation*}
	we have
	\begin{equation*}
		u_{2k-2}-u_0\in\frac{\lambda^{\frac{1}{2}}(t)}{(t\lambda(t))^{2}}IS^{\frac{\beta+3}{2}}(\max\{R^{\frac{-\beta+3}{2}}(\log(R))^{n_k},R^{\frac{\beta-1}{2}}(\log(R))^{n_k'}\},\mathcal{Q}_\beta)
	\end{equation*}
	for some integer $n_k\geq 0$ or $n_k'\geq 0$. Then
	\begin{equation*}
	\begin{aligned}
		t^2(u_{2k-2}^4-u_0^4)v_{2k-1}&=t^2(u_{2k-2}-u_0)^4v_{2k-1}+4t^2(u_{2k-2}-u_0)^3u_0v_{2k-1}\\&\quad +6t^2(u_{2k-2}-u_0)^2u_0^2v_{2k-1}+4t^2(u_{2k-2}-u_0)u_0^3v_{2k-1}.
	\end{aligned}
	\end{equation*}
	We have
	\begin{equation*}
	\begin{aligned}
		&\quad t^2(u_{2k-2}-u_0)^4v_{2k-1}\\
		&\in t^2\frac{\lambda^2(t)}{(t\lambda(t))^8}IS^{2\beta+6}(\max\{R^{-2\beta+6}(\log(R))^{4n_k},R^{2\beta-2}(\log(R))^{4n_k'}\},\mathcal{Q}_\beta)\\&\quad \times\frac{\lambda^{\frac{1}{2}}(t)}{(t\lambda(t))^{2k}}\\&\quad \times IS^{\frac{\beta+3}{2}+2(k-1)}(\max\{R^{\frac{-\beta+3}{2}+2(k-1)(1-\beta)_+}(\log(R))^{m_k},R^{\frac{\beta-1}{2}+2(k-1)(\beta-3)_+}(\log(R))^{m_k'}\},\mathcal{Q}_\beta)\\
		&\subset\frac{\lambda^{\frac{1}{2}}(t)}{(t\lambda(t))^{2k}}a^6\\&\quad \times IS^{\frac{\beta+3}{2}+2(k-1)}(\max\{R^{\frac{-\beta+3}{2}+2(k-1)(1-\beta)_+}(\log(R))^{4n_k+m_k},R^{\frac{\beta-1}{2}+2(k-1)(\beta-3)_+}(\log(R))^{4n_k'+m_k'}\},\mathcal{Q}_\beta),\\
	\end{aligned}
	\end{equation*}
	and
	\begin{equation*}
		\begin{aligned}
			&\quad 4t^2(u_{2k-2}-u_0)u_0^3v_{2k-1}\\
			&\in t^2\frac{\lambda^{\frac{1}{2}}(t)}{(t\lambda(t))^2}IS^{\frac{\beta+3}{2}}(\max\{R^{\frac{-\beta+3}{2}}(\log(R))^{n_k},R^{\frac{\beta-1}{2}}(\log(R))^{n_k'}\},\mathcal{Q}_\beta)\\&\quad \times\lambda^{\frac{3}{2}}(t)S^{\frac{3\beta-3}{2}}(R^{\frac{-3\beta-3}{2}})\frac{\lambda^{\frac{1}{2}}(t)}{(t\lambda(t))^{2k}}\\&\quad \times IS^{\frac{\beta+3}{2}+2(k-1)}(\max\{R^{\frac{-\beta+3}{2}+2(k-1)(1-\beta)_+}(\log(R))^{m_k},R^{\frac{\beta-1}{2}+2(k-1)(\beta-3)_+}(\log(R))^{m_k'}\},\mathcal{Q}_\beta)\\
			&\subset\frac{\lambda^{\frac{1}{2}}(t)}{(t\lambda(t))^{2k}}\\&\quad\times IS^{\frac{\beta+3}{2}+2(k-1)}(\max\{R^{\frac{-\beta+3}{2}+2(k-1)(1-\beta)_+}(\log(R))^{n_k+m_k},R^{\frac{\beta-1}{2}+2(k-1)(\beta-3)_+}(\log(R))^{n_k'+m_k'}\},\mathcal{Q}_\beta).\\
		\end{aligned}
	\end{equation*}
	The other two terms in $t^2(u_{2k-2}^4-u_0^4)v_{2k-1}$ are similar. Next, we compute
	\begin{equation*}
	\begin{aligned}
		&\quad t^2v_{2k-1}^5\\
		&\in t^2\frac{\lambda^{\frac{5}{2}}(t)}{(t\lambda(t))^{10k}}\\&\quad \times IS^{\frac{5\beta+15}{2}+10(k-1)}(\max\{R^{\frac{-5\beta+15}{2}+10(k-1)(1-\beta)_+}(\log(R))^{5m_k},R^{\frac{5\beta-5}{2}+10(k-1)(\beta-3)_+}(\log(R))^{5m_k'}\},\mathcal{Q}_\beta)\\
		&\subset\frac{\lambda^{\frac{1}{2}}(t)}{(t\lambda(t))^{2k}}a^{8k-2}\\&\quad \times IS^{\frac{\beta+3}{2}+2(k-1)}(\max\{R^{\frac{-\beta+3}{2}+2(k-1)(1-\beta)}(\log(R))^{5m_k},R^{\frac{\beta-1}{2}+2(k-1)(\beta-3)_+}(\log(R))^{5m_k'}\},\mathcal{Q}_\beta).\\
	\end{aligned}
	\end{equation*}
	Note that
	\begin{equation*}
		u_{2k-2}\in\lambda^{\frac{1}{2}}(t)IS^{\frac{\beta-1}{2}}(\max\{R^{\frac{-\beta-1}{2}}(\log(R))^{\widetilde{n}_k},R^{\frac{\beta-5}{2}}(\log(R))^{\widetilde{n}_{k}'}\},\mathcal{Q}_\beta),
	\end{equation*}
	hence
	\begin{equation*}
		\begin{aligned}
			&\quad 10t^2u_{2k-2}^3v_{2k-1}^2\\
			&\in t^2\lambda^{\frac{3}{2}}(t)IS^{\frac{3\beta-3}{2}}(\max\{R^{\frac{-3\beta-3}{2}}(\log(R))^{3\widetilde{n}_k},R^{\frac{3\beta-15}{2}}(\log(R))^{3\widetilde{n}_k'}\},\mathcal{Q}_\beta)\\&\quad \times\frac{\lambda(t)}{(t\lambda(t))^{4k}}\\&\quad \times IS^{\beta+3+4(k-1)}(\max\{R^{-\beta+3+4(k-1)(1-\beta)_+}(\log(R))^{2m_k},R^{\frac{\beta-1}{2}+4(k-1)(\beta-3)_+}(\log(R))^{2m_k'}\},\mathcal{Q}_\beta)\\
			&\subset\frac{\lambda^{\frac{1}{2}}(t)}{(t\lambda(t))^{2k}}a^{2k-2}\\&\quad \times IS^{\frac{\beta+3}{2}+2(k-1)}(\max\{R^{\frac{-\beta+3}{2}+2(k-1)(1-\beta)_+}(\log(R))^{3\widetilde{n}_k+2m_k},R^{\frac{\beta-1}{2}+2(k-1)(\beta-3)_+}(\log(R))^{3\widetilde{n}_k'+3m_k'}\},\mathcal{Q}_\beta).\\
		\end{aligned}
	\end{equation*}
	The remaining two terms in $N_{2k-1}(v_{2k-1})$ are similar, so we conclude that (\ref{eq:e_1}) holds.

	\par \underline{\textbf{Step 3}:} We begin with (\ref{eq:e_1}) and show (\ref{eq:v_2}).
	\par The main asymptotic component of $t^2e_{2k-1}$ at $R=\infty$ is given by
	\begin{equation*}
	\begin{aligned}
		t^2f_{2k-1}&=\frac{\lambda^{\frac{1}{2}}(t)}{(t\lambda(t))^{2k}}\sum_{i=0}^{I}\sum_{j=0}^{i+p_k}q_{i,j}(a)R^{\frac{-\beta+3}{2}+2(k-1)(1-\beta)_+-m_i}(\log(R))^{j}\\&=\frac{\lambda^{\frac{1}{2}}(t)}{(t\lambda(t))^{2(k-1)\min\{\beta,1\}+\frac{\beta+1}{2}+m_i}}\sum_{i=0}^{I}\sum_{j=0}^{i+p_k}a^{\frac{-\beta+3}{2}+2(k-1)(1-\beta)_+-m_i}q_{i,j}(a)(\log(R))^{j},\\
	\end{aligned}
	\end{equation*}
	for $0<\beta\leq 2$ or
	\begin{equation*}
	\begin{aligned}
	t^2f_{2k-1}&=\frac{\lambda^{\frac{1}{2}}(t)}{(t\lambda(t))^{2k}}\sum_{i=0}^{I}\sum_{j=0}^{i+p_k'}q_{i,j}(a)R^{\frac{-\beta+3}{2}+2(k-1)(\beta-3)_+-m_i}(\log(R))^{j}\\&=\frac{\lambda^{\frac{1}{2}}(t)}{(t\lambda(t))^{2(k-1)\min\{4-\beta,1\}+\frac{5-\beta}{2}+m_i}}\sum_{i=0}^{I}\sum_{j=0}^{i+p_k'}a^{\frac{\beta-1}{2}+2(k-1)(\beta-3)_+-m_i}q_{i,j}(a)(\log(R))^{j},\\
	\end{aligned}
	\end{equation*}
	for $2<\beta<4$. By the homogeneity consideration, we make the following ansatz for $v_{2k}$:
	\begin{equation*}
	\begin{aligned}
		v_{2k}&=\frac{\lambda^{\frac{1}{2}}(t)}{(t\lambda(t))^{2(k-1)\min\{\beta,1\}+\frac{\beta+1}{2}+m_i}}\sum_{i=0}^{I}\sum_{j=0}^{i+p_k}W_{i,j}(a)(\log(R))^{j},& 0<\beta\leq 2,\\
		v_{2k}&=\frac{\lambda^{\frac{1}{2}}(t)}{(t\lambda(t))^{2(k-1)\min\{4-\beta,1\}+\frac{5-\beta}{2}+m_i}}\sum_{i=0}^{I}\sum_{j=0}^{i+p_k'}W_{i,j}(a)(\log(R))^{j},& 2<\beta<4.\\
	\end{aligned}
	\end{equation*}
	The goal is to show that $W_{i,j}\in a^{\frac{-\beta+7}{2}+2(k-1)(1-\beta)_+-m_i}\mathcal{Q}_\beta$ for $0<\beta\leq 2$, and $W_{i,j}\in a^{\frac{\beta+3}{2}+2(k-1)(\beta-3)_+-m_i}\mathcal{Q}_\beta$ for $2<\beta<4$, hence (after suitable modification) we get
	\begin{equation*}
		\begin{aligned}
			v_{2k}&\in\frac{\lambda^{\frac{1}{2}}(t)}{(t\lambda(t))^{2k+2}}\\&\quad \times IS^{\frac{\beta+7}{2}+2(k-1)}(\max\{R^{\frac{-\beta+7}{2}+2(k-1)(1-\beta)_+}(\log(R))^{p_k},R^{\frac{\beta+3}{2}+2(k-1)(\beta-3)_+}(\log(R))^{p_k'}\},\mathcal{Q}_\beta).
		\end{aligned}
	\end{equation*}

	For each $i=1,2,\cdots,I$, we need to solve 
	\begin{equation*}
	\begin{aligned}
		t^2&\left(\partial_t^2-\partial_r^2-\frac{2}{r}\partial_r+\frac{\alpha}{r^2}\right)\left(\frac{\lambda^{\frac{1}{2}}(t)}{(t\lambda(t))^{2(k-1)\min\{\beta,1\}+\frac{\beta+1}{2}+m_i}}\sum_{j=0}^{i+p_k}W_{i,j}(a)\left(\log(R)\right)^{j}\right)\\&=\frac{\lambda^{\frac{1}{2}}(t)}{(t\lambda(t))^{2(k-1)\min\{\beta,1\}+\frac{\beta+1}{2}+m_i}}\sum_{j=0}^{i+p_k}a^{\frac{-\beta+3}{2}+2(k-1)(1-\beta)_+-m_i}q_{i,j}(a)\left(\log(R)\right)^{j},\quad 0<\beta\leq 2,
	\end{aligned}
	\end{equation*}
	or
	\begin{equation*}
	\begin{aligned}
		t^2&\left(\partial_t^2-\partial_r^2-\frac{2}{r}\partial_r+\frac{\alpha}{r^2}\right)\left(\frac{\lambda^{\frac{1}{2}}(t)}{(t\lambda(t))^{2(k-1)\min\{4-\beta,1\}+\frac{5-\beta}{2}+m_i}}\sum_{j=0}^{i+p_k'}W_{i,j}(a)\left(\log(R)\right)^{j}\right)\\&=\frac{\lambda^{\frac{1}{2}}(t)}{(t\lambda(t))^{2(k-1)\min\{4-\beta,1\}+\frac{5-\beta}{2}+m_i}}\sum_{j=0}^{i+p_k'}a^{\frac{\beta-1}{2}+2(k-1)(\beta-3)_+-m_i}q_{i,j}(a)\left(\log(R)\right)^{j},\quad 2<\beta<4,
	\end{aligned}
	\end{equation*}
	with zero Cauchy data at $a=0$. Equations for $a$ are obtained by matching the powers of $\log(R)$,
	\begin{equation*}
		\begin{aligned}
			L_{\beta+2m_i+2(k-1)\min\{\beta,1\}} W_{i,j}&=a^{\frac{-\beta+3}{2}+2(k-1)(1-\beta)_+-m_i}q_{i,j}(a)+\cdots,& 0<\beta\leq 2,\\
			L_{4-\beta+2m_i+2(k-1)\min\{4-\beta,1\}} W_{i,j}&=a^{\frac{\beta-1}{2}+2(k-1)(\beta-3)_+-m_i}q_{i,j}(a)+\cdots,& 2<\beta<4,\\
		\end{aligned}
	\end{equation*}
	where $\cdots$ contains $W_{i,j'}$ for $j'>j$ and their derivatives, but its explicit expression is not so relevant. Then invoking Lemma \ref{lem:a} gives the desired result.

	\par \underline{\textbf{Step 4}:} We begin with (\ref{eq:v_2}) and show (\ref{eq:e_2}).
	\par First, let
	\begin{equation*}
		\begin{aligned}
			t^2e_{2k-1}^0=\frac{\lambda^{\frac{1}{2}}(t)}{(t\lambda(t))^{2k}}\sum_{i=0}^{I}\sum_{j=0}^{i+p_k}&q_{i,j}(a)R^{\frac{-\beta+3}{2}+2(k-1)(1-\beta)_+-m_i}\\&\times\frac{R^{\beta+m_i+2(k-1)\max\{\beta,1\}}}{(1+R^2)^{\frac{\beta+m_1}{2}+(k-1)\max\{\beta,1\}}}\left(\frac{1}{2}\log(1+R^2)\right)^{j},& 0<\beta\leq 2,\\
			t^2e_{2k-1}^0=\frac{\lambda^{\frac{1}{2}}(t)}{(t\lambda(t))^{2k}}\sum_{i=0}^{I}\sum_{j=0}^{i+p_k'}&q_{i,j}(a)R^{\frac{\beta-1}{2}+2(k-1)(\beta-3)_+-m_i}\\&\times\frac{R^{2+m_i+2(k-1)\max\{4-\beta,1\}}}{(1+R^2)^{\frac{\beta+m_1}{2}+(k-1)\max\{4-\beta,1\}}}\left(\frac{1}{2}\log(1+R^2)\right)^{j},& 2<\beta<4,\\
		\end{aligned}
	\end{equation*}
	then
	\begin{equation*}
	\begin{aligned}
		&\quad t^2e_{2k-1}-t^2e_{2k-1}^0\\&\in \frac{\lambda^{\frac{1}{2}}(t)}{(t\lambda(t))^{2k}}IS^{\frac{\beta-1}{2}+2k}(\max\{R^{\frac{-\beta-1}{2}+2(k-1)(1-\beta)_+}(\log(R))^{p_k},R^{\frac{\beta-5}{2}+2(k-1)(\beta-3)_+}(\log(R))^{p_k'}\},\mathcal{Q}_\beta')\\&\subset\frac{\lambda^{\frac{1}{2}}(t)}{(t\lambda(t))^{2k}}IS^{\frac{\beta-1}{2}+2k}(\max\{R^{\frac{-\beta+3}{2}+2(k-1)(1-\beta)_+}(\log(R))^{p_k},R^{\frac{\beta-1}{2}+2(k-1)(\beta-3)_+}(\log(R))^{p_k'}\},\mathcal{Q}_\beta)\\&\quad +\frac{\lambda^{\frac{1}{2}}(t)}{(t\lambda(t))^{2k+2}}IS^{\frac{\beta+3}{2}+2k}(\max\{R^{\frac{-\beta+3}{2}+2k(1-\beta)_+}(\log(R))^{p_k},R^{\frac{\beta-1}{2}+2k(\beta-3)_+}(\log(R))^{p_k'}\},\mathcal{Q}_\beta'),
	\end{aligned}
	\end{equation*}
	which can be seen from
	\begin{equation*}
		1=(1-a^2)+b^2R^2,
	\end{equation*}
	and $(1-a^2)\mathcal{Q}_\beta'\subset\mathcal{Q}_\beta$. The error $e_{2k}$ is given by
	\begin{equation*}
		t^2e_{2k}=\left(t^2e_{2k-1}-t^2e_{2k-1}^0\right)+t^2\left(e_{2k-1}^0-\left(\partial_t^2-\partial_r^2-\frac{2}{r}\partial_r+\frac{\alpha}{r^2}\right)v_{2k}\right)+t^2N_{2k}(v_{2k}).
	\end{equation*}
	\par For the second term in $t^2e_{2k}$, one may check that at $R=\infty$, it's of the order $\max\{\frac{-\beta-1}{2}+2(k-1)(1-\beta)_+,\frac{\beta-5}{2}+2(k-1)(\beta-3)_+\}+$ (``$+$" stands for $(\log(R))^\ell$ terms). At $R=0$, a direct computation shows that it's of the order $\frac{\beta+3}{2}+2(k-1)=\frac{\beta-1}{2}+2k$. Therefore, we have
	\begin{equation*}
		\begin{aligned}
			&\quad t^2\left(e_{2k-1}^0-\left(\partial_t^2-\partial_r^2-\frac{2}{r}\partial_r+\frac{\alpha}{r^2}\right)v_{2k}\right)\\&\in\frac{\lambda^{\frac{1}{2}}(t)}{(t\lambda(t))^{2k}}IS^{\frac{\beta-1}{2}+2k}(\max\{R^{\frac{-\beta-1}{2}+2(k-1)(1-\beta)_+}(\log(R))^{p_k},R^{\frac{\beta-5}{2}+2(k-1)(\beta-3)_+}(\log(R))^{p_k'}\},\mathcal{Q}_\beta'),
		\end{aligned}
	\end{equation*}
	which is admissible by the decomposition above.
	\par For terms in $t^2N_{2k}(v_{2k})$, we have
	\begin{equation*}
		\begin{aligned}
			&\quad t^2v_{2k}^5\\&\in t^2\frac{\lambda^{\frac{5}{2}}(t)}{(t\lambda(t))^{10k+10}}\\&\times IS^{\frac{5\beta+35}{2}+10(k-1)}(\max\{R^{\frac{-5\beta+35}{2}+10(k-1)(1-\beta)_+}(\log(R))^{5p_k},R^{\frac{5\beta+15}{2}+10(k-1)(\beta-3)_+}(\log(R))^{5p_k'}\},\mathcal{Q}_\beta)\\
			&\subset\frac{\lambda^{\frac{1}{2}}(t)}{(t\lambda(t))^{2k}}a^{8k+8}IS^{\frac{\beta-1}{2}+2k}(\max\{R^{\frac{-\beta-1}{2}+2k(1-\beta)_+}(\log(R))^{5p_k},R^{\frac{\beta-5}{2}+2k(\beta-3)_+}(\log(R))^{5p_k'}\},\mathcal{Q}_\beta).\\
		\end{aligned}
	\end{equation*}
	Note that 
	\begin{equation*}
		u_{2k-1}\in\lambda^{\frac{1}{2}}(t)IS^{\frac{\beta-1}{2}}(\max\{R^{\frac{-\beta-1}{2}}(\log(R))^{\widetilde{m}_k},R^{\frac{\beta-5}{2}}(\log(R))^{\widetilde{m}_k'}\},\mathcal{Q}_\beta)
	\end{equation*}
	for some integer $\widetilde{m}_k$. Hence
	\begin{equation*}
		\begin{aligned}
			&\quad 5t^2u_{2k-1}^4v_{2k}\\&\in t^2\lambda^{2}(t)IS^{2\beta-2}(\max\{R^{-2\beta-2}(\log(R))^{4\widetilde{m}_k},R^{2\beta-10}(\log(R))^{4\widetilde{m}_k'}\},\mathcal{Q}_\beta)\\&\quad \times\frac{\lambda^{\frac{1}{2}}(t)}{(t\lambda(t))^{2k+2}}\\&\quad \times IS^{\frac{\beta+7}{2}+2(k-1)}(\max\{R^{\frac{-\beta+7}{2}+2(k-1)(1-\beta)_+}(\log(R))^{p_k},R^{\frac{\beta+3}{2}+2(k-1)(\beta-3)_+}(\log(R))^{p_k'}\},\mathcal{Q}_\beta)\\
			&\subset\frac{\lambda^{\frac{1}{2}}(t)}{(t\lambda(t))^{2k}}\\&\quad\times IS^{\frac{\beta-1}{2}+2k}(\max\{R^{\frac{-\beta-1}{2}+2k(1-\beta)_+}(\log(R))^{4\widetilde{m}_k+p_k},R^{\frac{\beta-5}{2}+2k(\beta-3)_+}(\log(R))^{4\widetilde{m}_k'+p_k'}\},\mathcal{Q}_\beta).\\
		\end{aligned}
	\end{equation*}
	The other three terms in $t^2N_{2k}(v_{2k})$ are similar, which concludes the proof of (\ref{eq:e_2}).

	\section{The perturbed equation}\label{sec:perturbed}
	
	Although the error $e_{2k-1}$ of $u_{2k-1}$ gains more decay in $t$ as $k$ increases, the sum $\sum_{\ell=0}^{\infty}v_{\ell}$ doesn't really converge (the coefficients involve divergent factors like $k!$), hence we only get an asymptotic expression by summing up $v_{\ell}$'s together. Instead, we will stop at some $u_{2k-1}$ for $k$ large enough, which is the ``correct" object to perturb around.
	\par We seek a radial solution of the form $u(t,r)=u_{2k-1}(t,r)+\varepsilon(t,r)$, which leads to the following equation for the perturbation $\varepsilon$:
	\begin{equation}\label{eq:perturbed}
		\partial_t^2\varepsilon-\partial_r^2\varepsilon-\frac{2}{r}\partial_r\varepsilon+\frac{\alpha}{r^2}\varepsilon-5\lambda^2(t)W_\alpha^4(\lambda(t)r)\varepsilon = N_{2k-1}(\varepsilon) + e_{2k-1}.
	\end{equation}
	We wish to apply the Fourier method to get an exact solution of this equation, where the time-dependent potential $5\lambda^2(t)W_\alpha^4(\lambda(t)r)$ is inadmissible. To eliminate the time-dependent scaling parameter $\lambda(t)$, we change to $(\tau,R)$-variables, where $\tau(t)=\frac{1}{\nu}t^{-\nu}$ and $R(t,r)=\lambda(t)r$. Let 
	\begin{equation*}
		\varepsilon(t,r)=:R^{-1}(t,r)\widetilde{\varepsilon}(\tau(t),R(t,r)).
	\end{equation*}
	Here we multiple $\widetilde{\varepsilon}$ by a factor of $R^{-1}$ in order to get a differential operator without first order derivative in $R$ variable. By the chain rule, we have
	\begin{equation*}
		\begin{aligned}
			\partial_t \varepsilon(t,r)=&-\lambda(t)R^{-1}(t,r)(\partial_\tau\widetilde{\varepsilon})(\tau(t),R(t,r))+\frac{\dot{\lambda}(t)}{\lambda(t)}(\partial_R\widetilde{\varepsilon})(\tau(t),R(t,r))\\&-\frac{\dot{\lambda}(t)}{\lambda(t)}R^{-1}(t,r)\widetilde{\varepsilon}(\tau(t),R(t,r)),\quad \dot{\lambda}(t):=\frac{d\lambda(t)}{dt},
		\end{aligned}
	\end{equation*}
	and
	\begin{equation*}
		\begin{aligned}
		\partial_t^2\varepsilon(t,r)=&+\lambda^2(t)R^{-1}(t,r)(\partial_\tau^2\widetilde{\varepsilon})(\tau(t),R(t,r))+\left(\frac{\dot{\lambda}(t)}{\lambda(t)}\right)^2(R\partial_R^2\widetilde{\varepsilon})(\tau(t),R(t,r))\\&-2\dot{\lambda}(t)(\partial_R\partial_\tau\widetilde{\varepsilon})(\tau(t),R(t,r))+\dot{\lambda}(t)R^{-1}(t,r)(\partial_\tau\widetilde{\varepsilon})(\tau(t),R(t,r))\\&-\left(\left(\frac{\dot{\lambda}(t)}{\lambda(t)}\right)^2-\partial_t\left(\frac{\dot{\lambda}(t)}{\lambda(t)}\right)\right)(\partial_R\widetilde{\varepsilon})(\tau(t),R(t,r))\\&+\left(\left(\frac{\dot{\lambda}(t)}{\lambda(t)}\right)^2-\partial_t\left(\frac{\dot{\lambda}(t)}{\lambda(t)}\right)\right)R^{-1}(t,r)\widetilde{\varepsilon}(\tau(t),R(t,r)).
		\end{aligned}
	\end{equation*}
	Similarily, we have
	\begin{equation*}
		\partial_r \varepsilon(t,r)=\lambda(t)R^{-1}(t,r)(\partial_R\widetilde{\varepsilon})(\tau(t),R(t,r))-\lambda(t)R^{-2}(t,r)\widetilde{\varepsilon}(\tau(t),R(t,r)),
	\end{equation*}
	and
	\begin{equation*}
		\begin{aligned}
			\partial_r^2\varepsilon(t,r)=&+\lambda^2(t)R^{-1}(t,r)(\partial_R^2\widetilde{\varepsilon})(\tau(t),R(t,r))-2\lambda^2(t)R^{-2}(t,r)(\partial_R\widetilde{\varepsilon})(\tau(t),R(t,r))\\&+2\lambda^2(t)R^{-3}(t,r)\widetilde{\varepsilon}(\tau(t),R(t,r)).
		\end{aligned}
	\end{equation*}
	Then (\ref{eq:perturbed}) becomes
	\begin{equation*}
		\begin{aligned}
		&+(\partial_\tau^2\widetilde{\varepsilon})(\tau(t),R(t,r))+\left(\frac{\dot{\lambda}(t)}{\lambda^2(t)}\right)^2(R^2\partial_R^2\widetilde{\varepsilon})(\tau(t),R(t,r))\\&-2\frac{\dot{\lambda}(t)}{\lambda^2(t)}(R\partial_R\partial_\tau\widetilde{\varepsilon})(\tau(t),R(t,r))+\frac{\dot{\lambda}(t)}{\lambda^2(t)}(\partial_\tau\widetilde{\varepsilon})(\tau(t),R(t,r))\\&-\left(\left(\frac{\dot{\lambda}(t)}{\lambda^2(t)}\right)^2-\lambda^{-2}(t)\partial_t\left(\frac{\dot{\lambda}(t)}{\lambda(t)}\right)\right)(R\partial_R\widetilde{\varepsilon})(\tau(t),R(t,r))\\&+\left(\left(\frac{\dot{\lambda}(t)}{\lambda^2(t)}\right)^2-\lambda^{-2}(t)\partial_t\left(\frac{\dot{\lambda}(t)}{\lambda(t)}\right)\right)\widetilde{\varepsilon}(\tau(t),R(t,r))\\&-(\partial_R^2\widetilde{\varepsilon})(\tau(t),R(t,r))+2R^{-1}(t,r)(\partial_R\widetilde{\varepsilon})(\tau(t),R(t,r))-2R^{-2}(t,r)\widetilde{\varepsilon}(\tau(t),R(t,r))\\&-2R^{-1}(t,r)(\partial_R\widetilde{\varepsilon})(\tau(t),R(t,r))+2R^{-2}(t,r)\widetilde{\varepsilon}(\tau(t),R(t,r))\\&+\alpha R^{-2}(t,r)\widetilde{\varepsilon}(\tau(t),R(t,r))-5W_\alpha^4(R)\widetilde{\varepsilon}(\tau(t),R(t,r))\\&=\lambda^{-2}(t)R(t,r)\left(N_{2k-1}(R^{-1}\widetilde{\varepsilon})+e_{2k-1}\right).
		\end{aligned}
	\end{equation*}
	Denote
	\begin{equation*}
		\mathcal{D}_1=\partial_\tau+\frac{\lambda_{\tau}}{\lambda}R\partial_R,\quad \lambda_\tau=\frac{d\lambda(\tau)}{d\tau},
	\end{equation*}
	where $\lambda(\tau)=(\nu \tau)^{1+\nu^{-1}}$, and
	\begin{equation*}
		\mathcal{L}=-\partial_R^2+\frac{\alpha}{R^2}-5W_\alpha^4(R).
	\end{equation*}
	After simplification, we get
	\begin{equation*}
		\mathcal{D}_1^2\widetilde{\varepsilon}-\frac{\lambda_{\tau}}{\lambda}\mathcal{D}_1\widetilde{\varepsilon}+\mathcal{L}\widetilde{\varepsilon}=\lambda^{-2}R\left(N_{2k-1}(R^{-1}\widetilde{\varepsilon})+e_{2k-1}\right)+\partial_\tau\left(\frac{\lambda_\tau}{\lambda}\right)\widetilde{\varepsilon}.
	\end{equation*}
	We may modify it a little bit further to absorb the last term $\partial_\tau(\lambda_{\tau}/\lambda)\widetilde{\varepsilon}$. Let 
	\begin{equation*}
		\mathcal{D}=\mathcal{D}_1-\frac{\lambda_\tau}{\lambda}=\partial_\tau+\frac{\lambda_{\tau}}{\lambda}(R\partial_R-1),
	\end{equation*}
	finally the perturbed equation becomes
	\begin{equation}\label{eq:perturb_symp}
		\mathcal{D}^2\widetilde{\varepsilon}+\frac{\lambda_{\tau}}{\lambda}\mathcal{D}\widetilde{\varepsilon}+\mathcal{L}\widetilde{\varepsilon}=\lambda^{-2}R\left(N_{2k-1}(R^{-1}\widetilde{\varepsilon})+e_{2k-1}\right).
	\end{equation}
	
	\par We remark that although we successfully eliminated the time-dependent potential, the cost we pay is that the $\partial_\tau$ operator is ``dilated", i.e. it becomes $\mathcal{D}$ with an additional $R\partial_R-1$ term.
	\par In the following section, we analyze various spectral properties of $\mathcal{L}$ and establish the spectral transform $\mathcal{F}\colon L^2((0,+\infty))\to L^2(\sigma(\mathcal{L}),d\rho)$, where $d\rho$ is the spectral measure of $\mathcal{L}$, which diagonalizes $\mathcal{L}$, so we can use the Fourier method to solve this ``dilated" wave equation.

	\section{Spectral analysis of the linearized operator}\label{sec:spec}
	In this section, we study the spectrum of the operator
	\begin{equation*}
		\mathcal{L}=\mathcal{L}^0-5W_\alpha^4(R),\quad \mathcal{L}^0=-\partial_R^2+\frac{\alpha}{R^2}
	\end{equation*}
	on the half line $(0,+\infty)$, and establish the \textit{distorted Hankel transform} for $\mathcal{L}$. (Recall that the Hankel transform diagonalizes $\mathcal{L}^0$.)
	\par $\mathcal{L}$ is self-adjoint on the domain
	\begin{equation*}
		\begin{aligned}
			D(\mathcal{L})&=\{f\in L^2((0,+\infty))\mid f,f'\in AC_{\mathrm{loc}}((0,+\infty)),\ f(0)=0,\ \mathcal{L}f\in L^2((0,+\infty))\},& 0<\beta<2,\\
			D(\mathcal{L})&=\{f\in L^2((0,+\infty))\mid f,f'\in AC_{\mathrm{loc}}((0,+\infty)),\ \mathcal{L}f\in L^2((0,+\infty))\},& 2\leq\beta<4.\\
		\end{aligned}
	\end{equation*}
	There is no boundary condition when $2\leq\beta<4$ at $R=0$ since $\mathcal{L}$ is in the limit point case. In the following proposition, we collect some basic properties of $\mathcal{L}$, whose proofs are all standard.

	\begin{proposition}
		The spectrum of $\mathcal{L}$, which we denote by $\sigma(\mathcal{L})$, is given by
		\begin{equation*}
			\sigma(\mathcal{L})=\sigma_{ac}(\mathcal{L})\cup \sigma_{pp}(\mathcal{L}),
		\end{equation*}
		where
		\begin{equation*}
			\sigma_{ac}(\mathcal{L})=\sigma_{ess}(\mathcal{L})=[0,+\infty),\quad \sigma_{pp}(\mathcal{L})=\{\xi_d<0,0\ (\text{only for $\beta>2$})\}
		\end{equation*}
		are absolutely continuous spectrum and point spectrum (the set of eigenvalues) respectively. The negative eigenfunction $\phi_d(R)$ decays exponentially fast as $R\to\infty$.
		\par Furthermore, $\mathcal{L}$ is in the limit circle case at $0$ when $0<\beta<2$, is in the limit point case at $0$ when $\beta\geq2$. $\mathcal{L}$ is always in the limit point case at $+\infty$ for all $\beta>0$.
	\end{proposition}

	\begin{proof}
		The essential spectrum follows from \cite[Lemma 9.35]{Tes14}. The statement of negative eigenvalue follows from the fact that $\phi_0(R)$ has one positive zero and Strum oscillation theorem \cite[Chapter XIII.7, Theorem 55, Corollary 56]{DS88}. When $\beta\leq 2$, $\phi_0\notin L^2((0,+\infty))$ hence $0\notin\sigma_{pp}(\mathcal{L})$. When $\beta>2$, $\phi_0\in L^2((0,+\infty))$ hence $0\in\sigma_{pp}(\mathcal{L})$. No positive eigenvalue follows from Kato's theorem \cite[Theorem XIII.58]{RS78}. The continuous part of the spectrum is purely absolutely continuous due to the integrability of $\frac{\alpha}{R^2}-5W_\alpha^4(R)$ at $R=\infty$ \cite[Theorem 15.3]{Wei06}. 
	\end{proof}

	For each $z\in\mathbb{C}$, there is a fundamental system of solutions $\phi(R,z)$ and $\theta(R,z)$ for $\mathcal{L}-z$ with the following asymptotic behavior as $R\to 0$:
	\begin{equation*}
		\phi(R,z)\sim \beta R^{\frac{\beta+1}{2}},\quad \theta(R,z)\sim \beta^{-2}R^{\frac{-\beta+1}{2}}.
	\end{equation*}
	If $0<\beta<2$, $\mathcal{L}$ is in the limit circle case at $R=0$ and $\phi(R,z),\ \theta(R,z)$ are both square integrable near $R=0$. If $2\leq \beta<4$, $\mathcal{L}$ is in the limit point case at $R=0$ and $\phi(R,z)$ is the only solution (modulo constant multiples) that belongs to $L^2([0,1])$, which we call the Weyl-Titchmarsh solution of $\mathcal{L}-z$ at $R=0$.
	\par Since $\mathcal{L}$ is always in the limit point case at $R=\infty$, there exists a unique $m(z)\in\mathbb{C}$ such that $\theta(R,z)+m(z)\phi(R,z)$ belongs to $L^2([1,\infty])$ for each $z\in\mathbb{C}$, $\mathrm{Im}(z)>0$. We call $m(z)$ the Weyl-Titchmarsh function of $\mathcal{L}-z$ at $R=\infty$. After normalization, we denote by $\psi^+(R,z)$ the Weyl-Titchmarsh solution of $\mathcal{L}-z$ at $R=\infty$ with $\psi^+(R,z)\sim z^{-\frac{1}{4}}e^{iz^{\frac{1}{2}}R}$. This is a multiple of the Jost solution, which is the fixed point of 
	\begin{equation*}
		f_+(R,z)=e^{iz^{\frac{1}{2}}R}+\int_{R}^{\infty}\frac{\sin(z^{\frac{1}{2}}(R-R'))}{z^{\frac{1}{2}}}V(R')f_+(R',z)dR',
	\end{equation*}
	where $V(R)=-\frac{\alpha}{R^2}+5W_\alpha^4(R)$. Further, we let $\psi^+(R,\xi)=\psi^+(R,\xi+i0)$ for $\xi>0$ and $\psi^-(R,\xi)=\overline{\psi^+(R,\xi)}$, then $\psi^+(R,\xi)$ and $\psi^-(R,\xi)$ forms a fundamental system of solutions of $\mathcal{L}-\xi$ with asymptotics $\psi^{\pm}(R,\xi)\sim \xi^{-\frac{1}{4}}e^{\pm i\xi^{\frac{1}{2}}R}$.

	\begin{proposition}[distorted Hankel transform]
		The spectral measure of $\mathcal{L}$ is given by
		\begin{equation*}
			\begin{aligned}
				d\rho&=\delta_{\xi_d}+\rho(\xi)d\xi,& 0<\beta<2,\\
				d\rho&=\delta_{\xi_d}+\delta_{0}+\rho(\xi)d\xi,& 2\leq\beta<4,\\
			\end{aligned}
		\end{equation*}
		where
		\begin{equation*}
			\rho(\xi)=\frac{1}{\pi}m(\xi+i0),\quad \xi> 0.
		\end{equation*}
		
		The distorted Hankel transform $\mathcal{F}$ defined as
		\begin{equation*}
			\mathcal{F}f(\xi)=\lim\limits_{r\to\infty}\int_{0}^{r}\phi(R,\xi)f(R)dR,\quad \xi\in\sigma(\mathcal{L}),
		\end{equation*}
		is a unitary operator from $L^2((0,+\infty))$ to $L^2(\sigma(\mathcal{L}),\rho)$ and its inverse is given by 
		\begin{equation*}
			\begin{aligned}
				\mathcal{F}^{-1}g(R)&=g(\xi_d)\phi_d(R)+\lim\limits_{\mu\to\infty}\int_{0}^{\mu}\phi(R,\xi)g(\xi)\rho(\xi)d\xi,& 0<\beta<2,\\
				\mathcal{F}^{-1}g(R)&=g(\xi_d)\phi_d(R)+g(0)\phi_0(R)+\lim\limits_{\mu\to\infty}\int_{0}^{\mu}\phi(R,\xi)g(\xi)\rho(\xi)d\xi,& 2\leq \beta<4.\\
			\end{aligned}
		\end{equation*}
		The limits in the above expressions refer to $L^2$ limits in the corresponding spaces.
	\end{proposition}
	
	\begin{proof}
		See, for example \cite{GZ06}, \cite{Tes14}, \cite{KST12} and \cite{KT11}.
	\end{proof}

	\begin{proposition}\label{prop:phi}
		For any $z\in\mathbb{C}$, we have the following absolutely convergent expansion for $\phi(R,z)$:
		\begin{equation}\label{eq:phi}
			\phi(R,z)=\phi_0(R)+R^{\frac{\beta+1}{2}}\sum_{j=1}^{\infty}(R^2z)^j\phi_j(R^{2\beta}),
		\end{equation}
		where $\phi_j$'s ($j\geq 1$) are holomorphic in $U=\{u\in\mathbb{C}\mid \mathrm{Re}(u)>-\frac{1}{2}\}$ and satisfy the bounds for $u\in U$
		\begin{equation*}
			\begin{aligned}
				|\phi_j(u)|&\leq\frac{C^{j+1}}{(j-1)!}\left<u\right>^{-\frac{1}{2}},& 0<\beta<2,\\
				|\phi_j(u)|&\leq\frac{C^{j+1}}{(j-1)!}\left<u\right>^{-\frac{1}{2}}\log(2+|u|),& \beta=2,\\
				|\phi_j(u)|&\leq\frac{C^{j+1}}{(j-1)!}\left<u\right>^{-\frac{1}{\beta}},& \beta>2.\\
			\end{aligned}
		\end{equation*}
		Furthermore, we have
		\begin{equation*}
			\phi_1(u)=\left\{\begin{aligned}
			&\frac{\beta}{2(\beta+2)}(1+o(1)),& u\to 0,\quad &\beta>0,\\
			&\frac{\beta}{2(2-\beta)}u^{-\frac{1}{2}}(1+o(1)),& u\to \infty,\quad& 0<\beta<2,\\
			&\frac{1}{4}u^{-\frac{1}{2}}\log(u)(1+o(1)),& u\to\infty,\quad &\beta=2.\\
			&cu^{-\frac{1}{\beta}}(1+o(1)),& u\to\infty,\quad &\beta>2,\\
			\end{aligned}\right.
		\end{equation*}
		with $c\neq 0$.
	\end{proposition}
	
	\begin{proof}
		We make the following ansatz for $\phi(R,z)$:
		\begin{equation*}
			\phi(R,z)=R^{\frac{\beta+1}{2}}\sum_{j=0}^{\infty}(R^2z)^jf_j(R),
		\end{equation*}
		where $f_0(R)=R^{\frac{1+\beta}{2}}\phi_0(R)$. Substituting into $(L-z)\phi(R,z)=0$ formally, we obtain
		\begin{equation*}
			L(R^{\frac{\beta+1}{2}}R^{2j}f_j(R))=R^{\frac{\beta+1}{2}}R^{2(j+1)}f_{j-1}(R),\quad j\geq1.
		\end{equation*}
		By the variation of parameter formula, we get the following expression for $f_j$ in terms of $f_{j-1}$:
		\begin{equation*}
			\begin{aligned}
				&\quad f_j(R)\\&=\int_{0}^{R}\frac{R'^{\frac{\beta+1}{2}+2(j-1)}}{R^{\frac{\beta+1}{2}+2j}}\left(\phi_0(R)\theta_0(R')-\phi_0(R')\theta_0(R)\right)f_{j-1}(R')dR'\\
				&=\int_{0}^{R}\frac{R^{-2j}(1-R^{2\beta})R'^{1+2(j-1)}(1-6R'^{2\beta}+R'^{4\beta})-R'^{1+\beta+2(j-1)}(1-R'^{2\beta})R^{-\beta-2j}(1-6R^{2\beta}+R^4\beta)}{\beta(1+R^{2\beta})^{\frac{3}{2}}(1+R'^{2\beta})^{\frac{3}{2}}}\\&\quad \quad\quad \times f_{j-1}(R')dR'.\\
			\end{aligned}
		\end{equation*}
		Define $\phi_j(R)=f_j(R^{\frac{1}{2\beta}})$ for $R>0$ and $j\geq0$, we have for $R>0$
		\begin{equation*}
		\begin{aligned}
			\phi_j(R)&=\int_{0}^{R}\frac{R^{-\frac{j}{\beta}}(1-R)R'^{\frac{j}{\beta}-1}(1-6R'+R'^2)-R'^{\frac{j}{\beta}-\frac{1}{2}}(1-R')R^{-\frac{1}{2}-\frac{j}{\beta}}(1-6R+R^2)}{2\beta^2(1+R)^{\frac{3}{2}}(1+R')^{\frac{3}{2}}}\phi_{j-1}(R')dR'\\
			&=\int_{0}^{R}\frac{R'^{\frac{j}{\beta}-1}}{R^{\frac{j}{\beta}+\frac{1}{2}}}\frac{R^{\frac{1}{2}}(1-R)(1-6R'+R'^2)-R'^{\frac{1}{2}}(1-R')(1-6R+R^2)}{2\beta^2(1+R)^{\frac{3}{2}}(1+R')^{\frac{3}{2}}}\phi_{j-1}(R')dR'.
		\end{aligned}
		\end{equation*}
		Since by induction hypothesis, $\phi_{j-1}(u)$ is holomorphic in $U$, it's clear that $\phi_j$ can be extended to a holomorphic function in any simply connected domain of $U$ not containing $0$, say $U-\mathbb{R}_{\leq0}$. In fact, $\phi_j$ can be extended to a holomorphic function in $U-\{0\}$. To show this, consider the difference of $\phi_j(-R+i0)$ and $\phi_j(-R-i0)$ for $0<R<\frac{1}{2}$, that is
		\begin{equation*}
		\begin{aligned}
			&\phi_j(-R+i0)-\phi_j(-R-i0)\\&=\int_{0}^{R}\left(1-\frac{e^{2\pi i(\frac{j}{\beta}-1)}}{e^{2\pi i(\frac{j}{\beta}+\frac{1}{2})}}\times e^{\pi i}\right)\\&\quad \times\frac{(-R)^{\frac{j}{\beta}-1}}{(-R)^{\frac{j}{\beta}+\frac{1}{2}}}\frac{(-R)^{\frac{1}{2}}(1+R)(1+6R'+R'^2)-(-R')^{\frac{1}{2}}(1+R')(1+6R+R^2)}{2\beta^2(1-R)^{\frac{3}{2}}(1-R')^{\frac{3}{2}}}\phi_{j-1}(-R')dR'\\&=0.\\
		\end{aligned}
		\end{equation*}
		Furthermore, for $|u|\ll1$, we have
		\begin{equation*}
			|\phi_j(u)|\lesssim \int_{0}^{|u|}\left(\frac{|v|^{\frac{j}{\beta}-1}}{|u|^{\frac{j}{\beta}}}+\frac{|v|^{\frac{j}{\beta}-\frac{1}{2}}}{|u|^{\frac{j}{\beta}+\frac{1}{2}}}\right)d|v|\lesssim 1,
		\end{equation*}
		so the singularity at $0$ is removable, i.e. $\phi_j$ can be extended to a holomorphic function in the entire $U$.
		\par To prove the bounds in the statement, we first estimate the integal kernel as 
		\begin{equation*}
			\begin{aligned}
				&\left|\frac{R^{\frac{1}{2}}(1-R)(1-6R'+R'^2)-R'^{\frac{1}{2}}(1-R')(1-6R+R^2)}{2\beta^2(1+R)^{\frac{3}{2}}(1+R')^{\frac{3}{2}}}\right|\mathbf{1}_{|R'|<|R|}\\&\lesssim\frac{|R^\frac{1}{2}-R'^{\frac{1}{2}}|+|R^{\frac{3}{2}}-R'^{\frac{3}{2}}|+|R|^{\frac{1}{2}}|R'|^{\frac{1}{2}}|R^{\frac{1}{2}}-R'^{\frac{1}{2}}|}{|1+R|^{\frac{3}{2}}|1+R'|^{\frac{3}{2}}}\mathbf{1}_{|R'|<|R|}\\&\quad +\frac{|R||R'||R^{\frac{1}{2}}-R'^{\frac{1}{2}}|+|R|^{\frac{1}{2}}|R'|^{\frac{1}{2}}|R^2-R'^2|+|R|^{\frac{3}{2}}|R'|^{\frac{3}{2}}|R^{\frac{1}{2}}-R'^{\frac{1}{2}}|}{|1+R|^{\frac{3}{2}}|1+R'|^{\frac{3}{2}}}\mathbf{1}_{|R'|<|R|}\\&\lesssim |R|^{\frac{1}{2}},
			\end{aligned}
		\end{equation*}
		where we have used $|1+R|\geq\frac{1}{2}$, $|1+R'|\geq\frac{1}{2}$, and 
		\begin{equation*}
			\frac{|R|}{|1+R|}\leq 1+\frac{1}{|1+R|}\lesssim 1,\quad \frac{|R'|}{|1+R'|}\leq 1+\frac{1}{|1+R'|}\lesssim 1.
		\end{equation*}
		Hence the integral kernel satisfies the bound
		\begin{equation*}
			\left|\frac{R'^{\frac{j}{\beta}-1}}{R^{\frac{j}{\beta}+\frac{1}{2}}}\frac{R^{\frac{1}{2}}(1-R)(1-6R'+R'^2)-R'^{\frac{1}{2}}(1-R')(1-6R+R^2)}{2\beta^2(1+R)^{\frac{3}{2}}(1+R')^{\frac{3}{2}}}\right|\mathbf{1}_{|R'|<|R|}\leq C_0\frac{|R'|^{\frac{j}{\beta}-1}}{|R|^{\frac{j}{\beta}}}
		\end{equation*}
		for some constant $C_0>0$.
		\par By the inductive hypothesis, $|\phi_{j-1}(u)|\leq \frac{C^{j}}{(j-2)!}$, so we have
		\begin{equation*}
		\begin{aligned}
			|\phi_j(u)|&\leq\int_{0}^{|u|}C_0\frac{|v|^{\frac{j}{\beta}-1}}{|u|^{\frac{j}{\beta}}}|\phi_{j-1}(v)|d|v|\\
			&\leq \frac{ C_0C^{j}}{(j-2)!}\int_{0}^{|u|}\frac{|v|^{\frac{j}{\beta}-1}}{|u|^{\frac{j}{\beta}}}d|v|=\frac{C_0\beta C^{j-1}}{j!}.
		\end{aligned}
		\end{equation*}
		On the other hand, when $0<\beta<2$, we have $|\phi_{j-1}(u)|\leq\frac{C^j}{(j-2)!}|u|^{-\frac{1}{2}}$, hence
		\begin{equation*}
		\begin{aligned}
			|\phi_j(u)|&\leq \frac{ C_0C^{j}}{(j-2)!}\int_{0}^{|u|}\frac{|v|^{\frac{j}{\beta}-1}}{|u|^{\frac{j}{\beta}}}\left<v\right>^{-\frac{1}{2}}d|v|\\&\leq \frac{ C_0C^{j}}{(j-2)!}\int_{0}^{|u|}\frac{|v|^{\frac{j}{\beta}-\frac{3}{2}}}{|u|^{\frac{j}{\beta}}}d|v|\\&\leq \frac{\beta C_0\beta C^{j-1}}{(j-1)!}|u|^{-\frac{1}{2}}.
		\end{aligned}
		\end{equation*}
		Taking $C$ large enough, we have
		\begin{equation*}
			|\phi_j(u)|\leq\frac{\beta C_0'\beta C^{j}}{(j-1)!}\left<u\right>^{-\frac{1}{2}}\leq\frac{C^{j}}{(j-1)!}\left<u\right>^{-\frac{1}{2}}.
		\end{equation*}
		\par If $\beta=2$, we have for $|u|\leq1$ that
		\begin{equation*}
			|\phi_1(u)|\leq C_0\int_{0}^{|u|}\frac{|v|^{-\frac{1}{2}}}{|u|^{\frac{1}{2}}}|\phi_0(v)|d|v|\leq 2C_0C.
		\end{equation*}
		For $|u|>1$, we have
		\begin{equation*}
			\begin{aligned}
				|\phi_1(u)|&\leq C_0C\int_{0}^{1}\frac{|v|^{-\frac{1}{2}}}{|u|^{\frac{1}{2}}}d|v|+C_0C\int_{1}^{|u|}\frac{|v|^{-1}}{|u|^{\frac{1}{2}}}d|v|\\&\leq 2C_0C|u|^{-\frac{1}{2}}+C_0C|u|^{-\frac{1}{2}}\log(|u|).
			\end{aligned}
		\end{equation*}
		Combining with the estimate for small $u$, we get
		\begin{equation*}
		|\phi_1(u)|\leq C^2\left<u\right>^{-\frac{1}{2}}\log(2+|u|).
		\end{equation*}
		For $j\geq2$, we use the inductive hypothesis $|\phi_{j-1}(u)|\leq\frac{C^{j}}{(j-2)!}\left<u\right>^{-\frac{1}{2}}\log(2+|u|)$ to get for $|u|\leq1$ that
		\begin{equation*}
			\begin{aligned}
				|\phi_j(u)|&\leq\frac{C_0C^{j}}{(j-2)!}\int_{0}^{|u|}\frac{|v|^{\frac{j}{2}-1}}{|u|^{\frac{j}{2}}}\left<v\right>^{-\frac{1}{2}}\log(2+|v|)d|v|\\
				&\leq 	\frac{C_0C^{j}}{(j-2)!}\int_{0}^{|u|}\frac{|v|^{\frac{j}{2}-1}}{|u|^{\frac{j}{2}}}\log(2+|u|)d|v|\\&\leq \frac{2C_0C^{j}}{(j-1)!}\log(2+|u|),
			\end{aligned}
		\end{equation*}
		and for $|u|>1$
		\begin{equation*}
			\begin{aligned}
				|\phi_j(u)|&\leq\frac{C_0C^{j}}{(j-2)!}\int_{0}^{|u|}\frac{|v|^{\frac{j}{2}-1}}{|u|^{\frac{j}{2}}}\left<v\right>^{-\frac{1}{2}}\log(2+|v|)d|v|\\
				&\leq \frac{C_0C^{j}}{(j-2)!}\int_{0}^{|u|}\frac{|v|^{\frac{j}{2}-\frac{3}{2}}}{|u|^{\frac{j}{2}}}\log(2+|u|)d|v|\\&\leq \frac{2C_0C^{j}}{(j-1)!}|u|^{-\frac{1}{2}}\log(2+|u|).
			\end{aligned}
			\end{equation*}
		Hence, we get for $j\geq2$,
		\begin{equation*}
			|\phi_j(u)|\leq\frac{C_0'C^{j}}{(j-1)!}\left<u\right>^{-\frac{1}{2}}\log(2+|u|)\leq\frac{C^{j+1}}{(j-1)!}\left<u\right>^{-\frac{1}{2}}\log(2+|u|).
		\end{equation*}
		\par Finally, we deal with the case $\beta>2$. For $|u|\leq 1$, we have
		\begin{equation*}
			|\phi_1(u)|\leq C_0\int_{0}^{|u|}\frac{|v|^{\frac{1}{\beta}-1}}{|u|^{\frac{1}{\beta}}}|\phi_0(v)|d|v|\leq \beta C_0C.
		\end{equation*}
		For $|u|>1$, we have
		\begin{equation*}
			\begin{aligned}
				|\phi_1(u)|&\leq C_0\int_{0}^{1}\frac{|v|^{\frac{1}{\beta}-1}}{|u|^{\frac{1}{\beta}}}d|v|+C_0C\int_{1}^{|u|}\frac{|v|^{\frac{1}{\beta}-\frac{3}{2}}}{|u|^{\frac{1}{\beta}}}d|v|\\
				&\leq \beta C_0C|u|^{-\frac{1}{\beta}}+\frac{2\beta}{2-\beta}C_0C |u|^{-\frac{1}{2}}-\frac{2\beta}{2-\beta}C_0C|u|^{-\frac{1}{\beta}},\\
				&\leq C_0'C|u|^{-\frac{1}{\beta}}.\\
			\end{aligned}
		\end{equation*}
		Combining with the estimate for small $u$, we get
		\begin{equation*}
			|\phi_1(u)|\leq C^2\left<u\right>^{-\frac{1}{\beta}}.
		\end{equation*}
		For $j\geq 2$, we use the imductive hypothesis $|\phi_{j-1}(u)|\leq\frac{C^j}{(j-2)!}\left<u\right>^{-\frac{1}{\beta}}$ to get for $|u|\leq 1$ that
		\begin{equation*}
			\begin{aligned}
				|\phi_j(u)|&\leq\frac{C_0C^j}{(j-2)!}\int_{0}^{|u|}\frac{|v|^{\frac{j}{\beta}-1}}{|u|^{\frac{j}{\beta}}}\left<v\right>^{-\frac{1}{\beta}}d|v|\\
				&\leq\frac{C_0C^j}{(j-2)!}\int_{0}^{|u|}\frac{|v|^{\frac{j}{\beta}-1}}{|u|^{\frac{j}{\beta}}}d|v|\\
				&\leq\frac{\beta C_0C^j}{(j-1)!},
			\end{aligned}
		\end{equation*}
		and for $|u|>1$ that
		\begin{equation*}
			\begin{aligned}
				|\phi_j(u)|&\leq\frac{C_0C^j}{(j-2)!}\int_{0}^{|u|}\frac{|v|^{\frac{j}{\beta}-1}}{|u|^{\frac{j}{\beta}}}\left<v\right>^{-\frac{1}{\beta}}d|v|\\
				&\leq \frac{C_0C^j}{(j-2)!}\int_{0}^{1}\frac{|v|^{\frac{j}{\beta}-1}}{|u|^{\frac{j}{\beta}}}d|v| +\frac{C_0C^j}{(j-2)!}\int_{1}^{|u|}\frac{|v|^{\frac{j-1}{\beta}-1}}{|u|^{\frac{j}{\beta}}}d|v|\\
				&\leq \frac{\beta C_0C^j}{j(j-2)!}|u|^{-\frac{j}{\beta}}+\frac{\beta C_0C^j}{(j-1)!}|u|^{-\frac{1}{\beta}}-\frac{\beta C_0C^j}{(j-1)!}|u|^{-\frac{j}{\beta}}\\
				&\leq \frac{C_0'C^j}{(j-1)!}|u|^{-\frac{1}{\beta}}.\\
			\end{aligned}
		\end{equation*}
		Hence, we get for $j\geq2$ that
		\begin{equation*}
			|\phi_j(u)|\leq \frac{C^{j+1}}{(j-1)!}\left<u\right>^{-\frac{1}{\beta}}.
		\end{equation*}
		
		\par We remark that the bounds on $\phi_j$ also imply that the formal series $\sum_{j=0}^{\infty}(R^2z)^jf_j(R)=\sum_{j=0}^{\infty}(R^2z)^j\phi_j(R^{2\beta})$ is absolutely convergent for $R>0$, so one can differentiate term by term at any times.
	\end{proof}
	\begin{remark}
		There is also a similar expansion for $\theta(R,\xi)$, but we don't need the explicit expression.
	\end{remark}

	\begin{corollary}\label{cor:small}
		If $R>0$, $R^2|z|\leq\delta$ and $|\mathrm{Im}(z)|\lesssim 1$, we have for all $k,\ell\geq0$ that
		\begin{equation*}
			\begin{aligned}
				|(R\partial_R)^k\partial_z^\ell\phi(R,z)|&\leq C_k R^{\frac{\beta+1}{2}+2\ell}\left<R^{2\beta}\right>^{-\frac{1}{2}},& 0<\beta<2,\\
				|(R\partial_R)^k\partial_z^\ell\phi(R,z)|&\leq C_k R^{\frac{3}{2}+2\ell}\log(2+R)\left<R\right>^{-2},& \beta=2,\\
				|(R\partial_R)^k\partial_z^\ell\phi(R,z)|&\leq C_k R^{\frac{\beta+1}{2}+2\ell}\left<R^{2\beta}\right>^{-\frac{1}{\beta}},& \beta>2.\\
			\end{aligned}
		\end{equation*}
		\par If $\delta$ is small enough (but not zero), we also have
		\begin{equation*}
			\begin{aligned}
				|(R\partial_R)^k\partial_z^\ell\phi(R,z)|&\geq C_k'R^{\frac{\beta+1}{2}}\left<R^{2\beta}\right>^{-\frac{1}{2}},& 0<\beta<2,\\
				|(R\partial_R)^k\partial_z^\ell\phi(R,z)|&\geq C_k' R^{\frac{3}{2}+2\ell}\log(2+R)\left<R\right>^{-2},&\beta=2,\\
				|(R\partial_R)^k\partial_z^\ell\phi(R,z)|&\geq C_k'R^{\frac{\beta+1}{2}+2\ell}\left<R^{2\beta}\right>^{-\frac{1}{\beta}},&\beta>2.\\
			\end{aligned}
		\end{equation*}
	\end{corollary}
		
	\begin{proof}
		Differentiate (\ref{eq:phi}) by $R\partial_R$, we get
		\begin{equation*}
			\begin{aligned}
				(R\partial_R)\phi(R,z)&=(R\partial_R)\phi_0(R)+R^{\frac{\beta+1}{2}}\sum_{j=1}^{\infty}(2j+\frac{1-3\beta}{2})(R^2z)^{j}\phi_j(R^{2\beta})\\&\quad +R^{\frac{1-3\beta}{2}}\sum_{j=1}^{\infty}2\beta(R^2z)^jR^{2\beta}(u\phi_j(u))'(R^{2\beta}).
			\end{aligned}
		\end{equation*}
		\par When $0<\beta<2$, we have for $j\geq 1$ that $|\phi_j(u)|\leq\frac{C^{j+1}}{(j-1)!}\left<u\right>^{-\frac{1}{2}}$, so
		\begin{equation*}
			\left|R^{\frac{\beta+1}{2}}\sum_{j=1}^{\infty}(2j+\frac{1-3\beta}{2})(R^2z)^j\phi_j(R^{2\beta})\right|\leq C\delta R^{\frac{\beta+1}{2}}\left<R^{2\beta}\right>^{-\frac{1}{2}}.
		\end{equation*}
		To bound the term involving $(u\phi_j(u))'$, we use the gradient estimate for holomorphic function $(u\phi_j(u))'$ in $U$: for $r>0$, we have (here we use our asymption that $|\mathrm{Im}(z)|\lesssim 1$)
		\begin{equation*}
			|(u\phi_j(u))'(r)|\leq \frac{\max\limits_{u\in\partial B(r,r)}|u\phi_j(u)|}{r}\leq  \frac{C'C^{j+1}}{(j-1)!}\left<r\right>^{-\frac{1}{2}},
		\end{equation*}
		which yields
		\begin{equation*}
			\left|R^{\frac{1-3\beta}{2}}\sum_{j=1}^{\infty}2\beta(R^2z)^jR^{2\beta}(u\phi_j(u))'(R^{2\beta})\right|\leq C\delta R^{\frac{\beta+1}{2}}\left<R^{2\beta}\right>^{-\frac{1}{2}}.
		\end{equation*}
		Together with the same bound for $(R\partial_R)\phi_0$, the desired result follows.
		\par When $\beta=2$, we have for $j\geq1$ that $|\phi_j(u)|\leq\frac{C^{j+1}}{(j-1)!}\log(2+|u|)\left<u\right>^{-\frac{1}{2}}$, so
		\begin{equation*}
			|(u\phi_j(u))'(r)|\leq \frac{\max\limits_{u\in\partial B(r,r)}|u\phi_j(u)|}{r}\leq \frac{C'C^{j+1}}{(j-1)!}\log(2+r) \left<r\right>^{-\frac{1}{2}}.
		\end{equation*}
		The remaining part and the $\beta>2$ case are both similar.
		\par For higher order derivatives, use the fact that $(u\partial_u)^k(u\phi_j(u))$ have the same estimate on $(0,+\infty)$ as $u\phi_j(u)$ does for all $k\geq0$. For $\partial_z$ and mixed derivatives, the change is obivious.
	\end{proof}

	Although the expansion in Proposition \ref{prop:phi} is valid for all $z\in\mathbb{C}$, we can not use it to estimate the behavior of $\phi(z)$ when $R^2|z|\geq\delta^2$. Instead, we examine the large $\xi$ behavior of $\phi(R,\xi)$ via the Jost solution.

	\par In the next proposition, we assume $\alpha\neq 0$ and the case $\alpha=0$ can be found in \cite{KST09a}.
	
	\begin{proposition}\label{prop:psi}
		Let $\alpha\neq 0$. For any $\xi>0$, the Jost solution $\psi^+(\cdot,\xi)$ is of the form
		\begin{equation*}\label{eq:psi}
			\psi^+(R,\xi)=\xi^{-\frac{1}{4}}e^{iR\xi^{\frac{1}{2}}}\sigma(R\xi^{\frac{1}{2}},R),\quad R\xi^{\frac{1}{2}}\geq1,
		\end{equation*}
		where $\sigma$ is a zeroth order symbol in $(q,R)$ and admits the asymptotic expansion
		\begin{equation*}
			\sigma(q,R)\sim \sum_{j=0}^{\infty}q^{-j}\psi_j^+(R),\quad \psi_0^+(R)=1,\quad \psi_!^+(R)=\frac{i\alpha}{2}+O(\max\{R^{2\beta},1\}\left<R\right>^{-4\beta}),
		\end{equation*}
		with $\psi_j^+$'s being zeroth order symbols in $R$, i.e.
		\begin{equation*}
			\sup_{R>0}|(R\partial_R)^k\psi_j^+(R)|<+\infty,
		\end{equation*}
		in the sense that for all nonnegative integers $k,\ell$ and $j_0$, we have
		\begin{equation*}
			\sup_{R>0}\left|(R\partial_R)^k(q\partial_q)^\ell\left(\sigma(q,R)-\sum_{j=0}^{j_0}q^{-j}\psi_j^+(R)\right)\right|\leq C_{k,\ell,j_0}q^{-j_0-1}
		\end{equation*}
		for all $q\geq\delta^{\frac{1}{2}}$, where $\delta>0$ can be chosen to be arbitrarily small.
	\end{proposition}
	
	\begin{proof}
		Let
		\begin{equation*}
			\sigma(\xi,R)=\psi^+(R,\xi)\xi^{\frac{1}{4}}e^{-iR\xi^{\frac{1}{2}}},
		\end{equation*}
		which satisfies the conjugated equation $(\xi^{-\frac{1}{4}}e^{iR\xi^{\frac{1}{2}}})^{-1}(L-\xi)(\xi^{-\frac{1}{4}}e^{iR\xi^{\frac{1}{2}}})\sigma=0$, i.e.
		\begin{equation*}
			\left(-\partial_R^2-2i\xi^{\frac{1}{2}}\partial_R+\frac{\alpha}{R^2}-5W_\alpha^4(R)\right)\sigma(\xi,R)=0.
		\end{equation*}
		We make the following ansatz for $\sigma(\xi,R)$:
		\begin{equation*}
			\sigma(\xi,R)=\sum_{j=0}^{\infty}\xi^{-\frac{j}{2}}f_j(R),
		\end{equation*}
		which yields the iterative relation (with vanishing Cauchy data at $R=\infty$)
		\begin{equation*}
			2i\partial_R f_j=\left(-\partial_R^2+\frac{\alpha}{R^2}-5W_\alpha^4(R)\right)f_{j-1},\quad f_0=1.
		\end{equation*}
		Integrating on both sides from $R$ to $\infty$ gives
		\begin{equation*}
			f_j(R)=\frac{i}{2}\partial_R f_{j-1}+\frac{i}{2}\int_{R}^{\infty}\left(\frac{\alpha}{R'^2}-\frac{15\beta^2R'^{2\beta-2}}{(1+R'^{2\beta})^2}\right)f_{j-1}(R')dR'.
		\end{equation*}
		From this, we see $f_1(R)=\frac{i\alpha}{2}R^{-1}+O(\max\{R^{-1+2\beta},1\})$ as $R\to0$ and $f_1(R)=\frac{i\alpha}{2}R^{-1}+O(R^{-1-2\beta})$ as $R\to\infty$. After differentiation, we get
		\begin{equation*}
			\partial_R f_j=\frac{i}{2}\partial_R^2f_{j-1}-\frac{i}{2}\left(\frac{\alpha}{R^2}-\frac{15\beta^2R^{2\beta-2}}{(1+R^{2\beta})^2}\right)f_{j-1}(R).
		\end{equation*}
		From this, we see $\partial_Rf_1=-\frac{i\alpha}{2}R^{-2}+O(R^{-2+2\beta})$ as $R\to 0$ and $\partial_Rf_1=-\frac{i\alpha}{2}R^{-2}+O(R^{-2-2\beta})$ as $R\to \infty$. The behavior of higher order derivatives can be estabilshed similarly, which turns out to be
		\begin{equation*}
			\sup_{R>0}|(R\partial_R)^kf_1|< C_{1,k}R^{-1}.
		\end{equation*}
		By induction, one may check that we have
		\begin{equation*}
			\sup_{R>0}|(R\partial_R)^kf_j|\leq C_{j,k}R^{-j},
		\end{equation*}
		which implies that $\psi^+_j(R):=R^jf_j(R)$ are zeroth order symbols for all $j\geq1$. As is well-known in symbolic calculus (Borel's theorem), there is a symbol $\sigma_{\mathrm{ap}}(q,R)$, unique upto a residual factor, admiting the asympototic series
		\begin{equation*}
			\sigma_{\mathrm{ap}}(q,R)\sim \sum_{j=0}^{\infty}q^{-j}\psi^+_j(R).
		\end{equation*}
		By definition, the error 
		\begin{equation*}
			e(\xi,R):=\left(-\partial_R^2-2i\xi^{\frac{1}{2}}\partial_R+\frac{\alpha}{R^2}-5W_\alpha^4(R)\right)\sigma_{\mathrm{ap}}(R\xi^{\frac{1}{2}},R)
		\end{equation*}
		of $\sigma_{\mathrm{ap}}(R\xi^{\frac{1}{2}},R)$ satisfies
		\begin{equation*}
		|(R\partial_R)^k(\xi\partial_\xi)^\ell e|\leq c_{k,\ell,j}(R\xi^{\frac{1}{2}})^{-j}R^{-2}
		\end{equation*}
		for all nonnegative intergers $k,\ell$ and $j$. It remains to prove that $\sigma_1(\xi,R)=\sigma(\xi,R)-\sigma_{\mathrm{ap}}(R\xi^{\frac{1}{2}},R)$ is residual. Note that $\sigma_1(\xi,R)$ is zero at $R=\infty$ and satisfies
		\begin{equation*}
			\left(-\partial_R^2-2i\xi^{\frac{1}{2}}\partial_R+\frac{\alpha}{R^2}-5W_\alpha^4(R)\right)\sigma_1(\xi,R)=-e(\xi,R).
		\end{equation*}
		We write the above equation as a first order system for $(v_1,v_2)=(\sigma_1,(R\partial_R)\sigma_1)$:
		\begin{equation*}
			\partial_R\begin{pmatrix}
				v_1\\v_2
			\end{pmatrix}=\begin{pmatrix}
				0&\frac{1}{R}\\
				\frac{\alpha}{R}-\frac{15\beta^2R^{2\beta-1}}{(1+R^{2\beta})^2}&\frac{1}{R}-2i\xi^{\frac{1}{2}}
			\end{pmatrix}\begin{pmatrix}
				v_1\\v_2
			\end{pmatrix}+\begin{pmatrix}
				0\\Re
			\end{pmatrix}.
		\end{equation*}
		Multiplying the above equation by $(\overline{v_1},\overline{v_2})$ and then taking the real part, we get
		\begin{equation*}
			\frac{d|v|^2}{dR}=\left(\frac{\alpha+1}{R}-\frac{15\beta^2R^{2\beta-1}}{(1+R^{2\beta})^2}\right)2\mathrm{Re}(v_1\overline{v_2})+\frac{2}{R}|v_2|^2+2R\mathrm{Re}(e\overline{v_2}),
		\end{equation*}
		which gives
		\begin{equation*}
			\frac{d|v|}{dR}\geq -C(R^{-1}|v|^2+R|e||v|).
		\end{equation*}
		By Gronwall's inequality,
		\begin{equation*}
			|v(R)|\leq C\int_{R}^{\infty}\left(\frac{R'}{R}\right)^CR'|e(R')|dR'.
		\end{equation*}
		Recall that for large $j$ we have
		\begin{equation*}
			|e(\xi,R)|\lesssim \xi^{-\frac{j}{2}}R^{-j-2},
		\end{equation*}
		which implies
		\begin{equation*}
			|v(R)|\lesssim \xi^{-\frac{j}{2}}R^{-j}=(R\xi^{\frac{1}{2}})^{-j}.
		\end{equation*}
		To estimate the derivatives of $v$, we apply the operator $R\partial_R$ and commute it with the matrix to get
		\begin{equation*}
			\begin{aligned}
				\partial_R(R\partial_R)\begin{pmatrix}
				v_1\\v_2
				\end{pmatrix}&=\begin{pmatrix}
				0&\frac{1}{R}\\
				\frac{\alpha}{R}-\frac{15\beta^2R^{2\beta-1}}{(1+R^{2\beta})^2}&\frac{1}{R}-2i\xi^{\frac{1}{2}}
				\end{pmatrix}(R\partial_R)\begin{pmatrix}
				v_1\\v_2
				\end{pmatrix}\\&\quad +\begin{pmatrix}
				0&-\frac{1}{R}\\
				-\frac{\alpha}{R}-\frac{15\beta^2(2\beta-1)}{(1+R^{2\beta})^2}+\frac{60\beta^3R^{4\beta-1}}{(1+R^{2\beta})^3} & -\frac{1}{R}
				\end{pmatrix}\begin{pmatrix}
				v_1\\v_2
				\end{pmatrix}+\begin{pmatrix}
				0\\R\partial_R (Re)
				\end{pmatrix}.
			\end{aligned}
		\end{equation*}
		The right hand side is bounded by $\xi^{-\frac{j}{2}}R^{-j-1}$ from the previous step, hence a similar calculation yields
		\begin{equation*}
			|(R\partial_R)v|\lesssim (R\xi^{\frac{1}{2}})^{-j}.
		\end{equation*}
		For $\xi\partial_\xi$ derivatives, we proceed similarly to get the commuted equation
		\begin{equation*}
			\partial_R(\xi\partial_\xi)\begin{pmatrix}
				v_1\\v_2
			\end{pmatrix}=\begin{pmatrix}
					0&\frac{1}{R}\\
				\frac{\alpha}{R}-\frac{15\beta^2R^{2\beta-1}}{(1+R^{2\beta})^2}&\frac{1}{R}-2i\xi^{\frac{1}{2}}
			\end{pmatrix}(\xi\partial_\xi)\begin{pmatrix}
				v_1\\v_2
			\end{pmatrix}+\begin{pmatrix}
				0&0\\
				0&-i\xi^{\frac{1}{2}}
			\end{pmatrix}\begin{pmatrix}
				v_1\\v_2
			\end{pmatrix}+\begin{pmatrix}
				0\\\xi\partial_\xi(Re)
			\end{pmatrix}.
		\end{equation*}
		This time, the right hand side is bounded by $\xi^{\frac{1}{2}}\cdot (R\xi^{\frac{1}{2}})^{-j-1}+\xi^{-\frac{j}{2}}R^{-j-1}\approx \xi^{-\frac{j}{2}}R^{-j-1}$ from the proven bound for $v$ with $j$ replaced by $j+1$, so the similar result holds for $(\xi\partial_\xi)v$.
		\par Finally, higher order derivatives are estimated by induction using the above argument at each step.
	\end{proof}
	\begin{remark}
		We have implictly used the fact that 
		\begin{equation*}
			\Phi\colon (0,+\infty)\times(0,+\infty)\to (0,+\infty)\times(0,+\infty),\quad (\xi,R)\to (R\xi^{\frac{1}{2}},R)
		\end{equation*}
		is a bijective smooth map whose Jacobi matrix at each $(\xi,R)\in (0,+\infty)\times (0,+\infty)$,
		\begin{equation*}
			d\Phi(\xi,R)=\begin{pmatrix}
				\frac{1}{2}R\xi^{-\frac{1}{2}}& \xi^{\frac{1}{2}}\\
				0&1
			\end{pmatrix}
		\end{equation*}
		is nonsingular. So by the inverse function theorem, $\Phi$ is a diffeomophism. As a result, any smooth function of $\xi$ and $R$ can be seen as a smooth function of $q=R\xi^{\frac{1}{2}}$ and $R$ via composing with $\Phi$. Moreover, since
		\begin{equation*}
			\xi\partial_\xi|_{R}=\frac{1}{2}q\partial_q|_{R},\quad R\partial_R|_{\xi}=q\partial_q|_{R}+R\partial_R|_{q},
		\end{equation*}
		the $(\xi\partial_\xi|_{R},R\partial_R|_{\xi})$ derivatives and $(q\partial_q|_R,R\partial_R|_q)$ derivatives are comparable.
	\end{remark}

	\begin{corollary}\label{cor:large}
		If $R^2\xi\geq \delta^2$ for some fixed $\delta$ as in Corollary \ref{cor:small} and Corollary \ref{cor:large}, we have for all $k,\ell\geq0$ that
		\begin{equation*}
			|(R\partial_R)^k\partial_\xi^\ell\psi^+(R\xi^{\frac{1}{2}},R)|\leq C_k R^k\xi^{-\frac{1}{4}+\frac{k-\ell}{2}}.
		\end{equation*}
	\end{corollary}

	\begin{proof}
		This follows directly from 
		\begin{equation*}
			R\partial_R\psi^+(R,\xi)=\xi^{-\frac{1}{4}}e^{iR\xi^{\frac{1}{2}}}\left(iR\xi^{\frac{1}{2}}\sigma(R\xi^{\frac{1}{2}},R)+(q\partial_q\sigma)(R\xi^{\frac{1}{2}},R)+(R\partial_R\sigma)(R\xi^{\frac{1}{2}},R)\right),
		\end{equation*}
		\begin{equation*}
			\partial_\xi\psi^+(R,\xi)=\xi^{-\frac{1}{4}}e^{iR\xi^{\frac{1}{2}}}\left(-\frac{1}{4}\xi^{-1}\sigma(R\xi^{\frac{1}{2}},R)+\frac{i}{2}R\xi^{-\frac{1}{2}}\sigma(R\xi^{\frac{1}{2}},R)+\frac{1}{2}\xi^{-1}(q\partial_q)\sigma(R\xi^{\frac{1}{2}},R)\right),
		\end{equation*}
		and the fact that $\sigma(q,R)$ is a zeroth order symbol.
	\end{proof}

	\begin{proposition}\label{prop:a}
		(a) We have
			\begin{equation*}
			\phi(R,\xi)=a(\xi)\psi^+(R,\xi)+\overline{a(\xi)}\psi^-(R,\xi),
			\end{equation*}
			where $a$ is a smooth function, always nonzero, and has size
			\begin{equation*}
			|a(\xi)|\approx\left\{\begin{aligned}
			&\xi^{\frac{\beta}{4}},& \xi\to 0,\quad& 0<\beta<2,\\
			&\xi^{\frac{1}{2}}|\log(\xi)|,& \xi\to 0,\quad &\beta=2,\\
			&\xi^{1-\frac{\beta}{4}},&\xi\to 0,\quad &\beta>2,\\
			&\xi^{-\frac{\beta}{4}},&\xi\to \infty,\quad &\beta>0.\\
			\end{aligned}\right.
			\end{equation*}
			Moreover, it satisfies the symbol type bounds
			\begin{equation*}
			|(\xi\partial_\xi)^ka(\xi)|\leq c_k|a(\xi)|,\quad \forall k\geq 0.
			\end{equation*}
			\par (b) The spectral measure $\rho(\xi)d\xi$ has density
			\begin{equation*}
			\rho(\xi)=\frac{1}{\pi}|a(\xi)|^{-2}
			\end{equation*}
			and therefore satisfies
			\begin{equation*}
			\rho(\xi)\approx\left\{\begin{aligned}
			&\xi^{-\frac{\beta}{2}},&\xi\to 0,\quad& 0<\beta\leq 2,\\
			&\frac{1}{\xi(\log(\xi))^2},& \xi\to 0,\quad &\beta=2,\\
			&\xi^{\frac{\beta}{2}-2},& \xi\to 0,\quad &\beta>2,\\
			&\xi^{\frac{\beta}{2}},& \xi\to \infty,\quad& \beta>0.\\
			\end{aligned}\right.
			\end{equation*}
			Moreover, $\rho$ also satisfies the symbol type bounds
			\begin{equation*}
				|(\xi\partial_\xi)^k\rho(\xi)|\leq C_k|\rho(\xi)|,\quad \forall k\geq 0.
			\end{equation*}
	\end{proposition}

	\begin{proof}
		(a) As is shown in \cite{KST09a}, we have
		\begin{equation}\label{eq:upper}
			a(\xi)=-\frac{i}{2}W(\phi(R,\xi),\psi^+(R,\xi)),
		\end{equation}
		and
		\begin{equation}\label{eq:lower}
			a(\xi)\geq\frac{|(R\partial_R)\phi(R,\xi)|}{2|(R\partial_R)\psi^+(R,\xi)|}.
		\end{equation}
		\par For upper bounds of $a$, we evaluate (\ref{eq:upper}) at $R=\xi^{-\frac{1}{2}}$, which gives
		\begin{equation*}
			a(\xi)=-\frac{i\xi^{\frac{1}{2}}}{2}\left(\phi(\xi^{-\frac{1}{2}},\xi)(R\partial_R\psi^+)(\xi^{-\frac{1}{2}},\xi)-(R\partial_R\phi)(\xi^{-\frac{1}{2}},\xi)\psi^+(\xi^{-\frac{1}{2}},\xi)\right).
		\end{equation*}
		By Corollary \ref{cor:small}, we have
		\begin{equation*}
			\begin{aligned}
				|(R\partial_R)^k\phi(\xi^{-\frac{1}{2}},\xi)|&\lesssim \xi^{-\frac{\beta+1}{4}}\left<\xi^{-\beta}\right>^{-\frac{1}{2}}=\xi^{\frac{\beta-1}{4}}\left<\xi^\beta\right>^{-\frac{1}{2}},& k=0,1,\quad & 0<\beta<2,\\
				|(R\partial_R)^k\phi(\xi^{-\frac{1}{2}},\xi)|&\lesssim \xi^{-\frac{3}{4}}\log(2+\xi^{-\frac{1}{2}})\left<\xi^{-2}\right>^{-\frac{1}{2}}=\frac{\xi^{\frac{1}{4}}}{1+\xi}\log(2+\xi^{-\frac{1}{2}}),& k=0,1,\quad &\beta=2,\\
				|(R\partial_R)^k\phi(\xi^{-\frac{1}{2}},\xi)|&\lesssim \xi^{-\frac{\beta+1}{4}}\left<\xi^{-\beta}\right>^{-\frac{1}{\beta}}=\xi^{\frac{-\beta+3}{4}}\left<\xi^\beta\right>^{-\frac{1}{\beta}},& k=0,1,\quad & \beta>2.\\
			\end{aligned}
		\end{equation*}
		By Corollary \ref{cor:large}, we have
		\begin{equation*}
			|(R\partial_R)^k\psi^+(\xi^{-\frac{1}{2}},\xi)|\lesssim\xi^{-\frac{1}{4}},\quad k=0,1,\quad \beta>0.
		\end{equation*}
		As a result, we get
		\begin{equation*}
			\begin{aligned}
				|a(\xi)|&\lesssim \xi^{\frac{1}{2}}\cdot \xi^{\frac{\beta-1}{4}}\left<\xi^\beta\right>^{-\frac{1}{2}}\cdot \xi^{-\frac{1}{4}}\lesssim \xi^{-\frac{\beta}{4}},&\xi\to\infty,\quad& \beta>0,\\
				|a(\xi)|&\lesssim \xi^{\frac{1}{2}}\cdot \xi^{\frac{\beta-1}{4}}\left<\xi^\beta\right>^{-\frac{1}{2}}\cdot \xi^{-\frac{1}{4}}\lesssim\xi^{\frac{\beta}{4}},& \xi\to 0,\quad& 0<\beta<2,\\
				|a(\xi)|&\lesssim \xi^{\frac{1}{2}}\cdot\xi^{\frac{1}{4}}|\log(\xi)|\cdot\xi^{-\frac{1}{4}}\lesssim\xi^{\frac{1}{2}}|\log(\xi)|,&\xi\to 0,\quad &\beta=2,\\
				|a(\xi)|&\lesssim \xi^{\frac{1}{2}}\cdot \xi^{\frac{-\beta+3}{4}}\left<\xi^\beta\right>^{-\frac{1}{\beta}}\cdot \xi^{-\frac{1}{4}}\lesssim \xi^{1-\frac{\beta}{4}},&\xi\to 0,\quad& \beta>2.\\
			\end{aligned}
		\end{equation*}
		\par For lower bounds of $a$, we evaluate (\ref{eq:lower}) at $R=\delta\xi^{-\frac{1}{2}}$ for $\delta$ small enough, such that inequalities in Corollary \ref{cor:small} holds, i.e.
		\begin{equation*}
			\begin{aligned}
				|(R\partial_R\phi)(\delta\xi^{-\frac{1}{2}},\xi)|&\geq C\xi^{-\frac{\beta+1}{4}}\left<\xi^{-\beta}\right>^{-\frac{1}{2}},& 0<\beta<2,\\
				|(R\partial_R\phi)(\delta\xi^{-\frac{1}{2}},\xi)|&\geq C\xi^{-\frac{3}{4}}\log(2+\delta\xi^{-\frac{1}{2}})\left<\xi^{-2}\right>^{-\frac{1}{2}},&\beta=2,\\
				|(R\partial_R\phi)(\delta\xi^{-\frac{1}{2}},\xi)|&\geq C\xi^{-\frac{\beta+1}{4}}\left<\xi^{-\beta}\right>^{-\frac{1}{\beta}},& \beta>2.\\
			\end{aligned}
		\end{equation*}
		As a result, we get
		\begin{equation*}
			\begin{aligned}
				|a(\xi)|&\geq C\frac{\xi^{-\frac{\beta+1}{4}}\left<\xi^{-\beta}\right>^{-\frac{1}{2}}}{\xi^{-\frac{1}{4}}}\geq C\xi^{-\frac{\beta}{4}},&\xi\to\infty,\quad& \beta>0, \\
				|a(\xi)|&\geq C\frac{\xi^{-\frac{\beta+1}{4}}\left<\xi^{-\beta}\right>^{-\frac{1}{2}}}{\xi^{-\frac{1}{4}}}\geq C\xi^{\frac{\beta}{4}},&\xi\to 0,\quad& 0<\beta<2,\\
				|a(\xi)|&\geq C\frac{\xi^{-\frac{3}{4}}\log(2+\delta\xi^{-\frac{1}{2}})\left<\xi^{-2}\right>^{-\frac{1}{2}}}{\xi^{-\frac{1}{4}}}\geq C\xi^{\frac{1}{2}}|\log(\xi)|,& \xi\to 0,\quad& \beta=2,\\
				|a(\xi)|&\geq C\frac{\xi^{-\frac{\beta+1}{4}}\left<\xi^{-\beta}\right>^{-\frac{1}{\beta}}}{\xi^{-\frac{1}{4}}}\geq C\xi^{1-\frac{\beta}{4}},& \xi\to 0,\quad& \beta>2.\\
			\end{aligned}
		\end{equation*}
		This concludes the proof of (a).
		\par (b) This part follows from (a) and 
		\begin{equation*}
			\rho(\xi)=\frac{1}{\pi}|a(\xi)|^{-2}
		\end{equation*}
		(this is also shown in \cite{KST09a}) directly.
	\end{proof}

		\par The following corollary is a combination of Corollary \ref{cor:small}, Corollary \ref{cor:large} and Proposition \ref{prop:a}, which gives a better estimate on $\phi(R,\xi)$ and its derivatives. This will also become useful when we study the decay of the integral kernel of the error operator $\mathcal{K}$ in Section \ref{sec:trans}.
	\begin{corollary}
		For $0<\beta<2$, the following estimates hold for all $k,\ell\geq 0$:
		\begin{equation*}
		\begin{aligned}
		|(R\partial_R)^k\partial_\xi^\ell\phi(R,\xi)|&\lesssim R^{\frac{\beta+1}{2}+2\ell},&\xi\geq 1,\quad & R^2\xi\leq\delta^2\\
		|(R\partial_R)^k\partial_\xi^\ell\phi(R,\xi)|&\lesssim R^{k+\ell}\xi^{-\frac{\beta+1}{4}+\frac{k-\ell}{2}},&\xi\geq 1,\quad & R^2\xi\geq\delta^2\\
		|(R\partial_R)^k\partial_\xi^\ell\phi(R,\xi)|&\lesssim R^{\frac{\beta+1}{2}+2\ell},& 0<\xi\leq 1,\quad & R^2\xi\leq\delta^2,& 0<R\leq 1,\\
		|(R\partial_R)^k\partial_\xi^\ell\phi(R,\xi)|&\lesssim R^{\frac{-\beta+1}{2}+2\ell},& 0<\xi\leq 1,\quad & R^2\xi\leq\delta^2,& R\geq 1,\\
		|(R\partial_R)^k\partial_\xi^\ell\phi(R,\xi)|&\lesssim R^{k+\ell}\xi^{\frac{\beta-1}{4}+\frac{k-\ell}{2}},& 0<\xi\leq 1,\quad & R^2\xi\geq\delta^2.\\
		\end{aligned}
		\end{equation*}
		\par For $\beta=2$, the estimates become as (omited items are the same as in the $0<\beta<2$ case)
		\begin{equation*}
		\begin{aligned}
		|(R\partial_R)^k\partial_\xi^\ell\phi(R,\xi)|&\lesssim R^{-\frac{1}{2}+2\ell}\log(1+R),& 0<\xi\leq 1,\quad&R^2\xi\leq\delta^2,\quad R\geq 1,\\
		|(R\partial_R)^k\partial_\xi^\ell\phi(R,\xi)|&\lesssim R^{k+\ell}\xi^{\frac{1}{4}+\frac{k-\ell}{2}}|\log(\xi/2)|,& 0<\xi\leq 1,\quad & R^2\xi\geq\delta^2.\\
		\end{aligned}
		\end{equation*}
		\par For $\beta>2$, the estimates become as (omited items are the same as in the $0<\beta<2$ case)
		\begin{equation*}
		\begin{aligned}
		|(R\partial_R)^k\partial_\xi^\ell\phi(R,\xi)|&\lesssim R^{\frac{\beta-3}{2}+2\ell},& 0<\xi\leq 1,\quad & R^2\xi\leq\delta^2,\quad R\geq 1,\\
		|(R\partial_R)^k\partial_\xi^\ell\phi(R,\xi)|&\lesssim R^{k+\ell}\xi^{\frac{-\beta+3}{4}+\frac{k-\ell}{2}},& 0<\xi\leq 1,\quad & R^2\xi\geq\delta^2.\\
		\end{aligned}
		\end{equation*}
	\end{corollary}
	
	\begin{proof}
		Just a direct computation using the fact that $a(\xi)$ is symbol-like at $\xi=0$ and $\xi=\infty$.
	\end{proof}

	\section{The transference identity}\label{sec:trans}
		
	In order to solve the perturbed equation using Fourier method, we apply the distorted Hankel transform $\mathcal{F}$, which digonalize the operator $\mathcal{L}$. However, as we have seen in Section \ref{sec:perturbed}, the time-derivative $\partial_\tau$ is ``transported" by $2\frac{\lambda_\tau}{\lambda}R\partial_R$, which is not digonalized under $\mathcal{F}$. Inspired by the fact that $R\partial_R-2\xi\partial_\xi$ vanishes $e^{iR\xi^{\frac{1}{2}}}$, we define the error operator $\mathcal{K}$ to satisfy (the $\partial_\xi$ derivative acts on $\widehat{u}|_{\mathbb{R}_{>0}}$ only)
	\begin{equation}\label{eq:trans}
		\widehat{R\partial_R u}=-2\xi\partial_\xi\widehat{u}+\mathcal{K}\widehat{u}.
	\end{equation}
	(\ref{eq:trans}) is called the \textit{transference identity}, since it transfers $R\partial_R$ derivative in the physics space to $-2\xi\partial_\xi$ in the frequency space. Since the spectral measure of $\mathcal{L}$ has two or three parts, it's useful to write a function $\widehat{u}$ in $L^2(\mathbb{R},d\rho)$ as a vector $(\widehat{u}(\xi_{d}),\widehat{u}|_{\mathbb{R}_{>0}})$ for $0<\beta\leq 2$, or $(\widehat{u}(\xi_{d}),\widehat{u}(0),\widehat{u}|_{\mathbb{R}_{>0}})$ for $\beta>2$. We also write the operator $\mathcal{K}$ as a matrix
	\begin{equation*}
		\mathcal{K}=\begin{pmatrix}
			\mathcal{K}_{dd} & \mathcal{K}_{dc}\\
			\mathcal{K}_{cd} & \mathcal{K}_{cc}\\
		\end{pmatrix}\quad \text{for 0<$\beta\leq 2$},\quad \text{or}\quad \mathcal{K}=\begin{pmatrix}
		\mathcal{K}_{dd} & \mathcal{K}_{d0} & \mathcal{K}_{dc}\\
		\mathcal{K}_{0d} & \mathcal{K}_{00} & \mathcal{K}_{0c}\\
		\mathcal{K}_{cd} & \mathcal{K}_{c0} & \mathcal{K}_{cc}\\
		\end{pmatrix} \quad \text{for $\beta$>2}.
	\end{equation*}
	A direct computation obtains for $f\in C^\infty_0((0,+\infty))$ that
	\begin{equation*}
		\begin{aligned}
			\mathcal{K}_{dd}&=\left<R\partial_R\phi_d(R),\phi_d(R)\right>_{L^2_R},\\
			\mathcal{K}_{d0}&=\left<R\partial_R\phi_0(R),\phi_d(R)\right>_{L^2_R},\\
			\mathcal{K}_{dc}f&=\left<\int_{0}^{\infty}f(\xi)R\partial_R\phi(R,\xi)\rho(\xi)d\xi,\phi_d(R)\right>_{L^2_R},\\
			\mathcal{K}_{0d}&=\left<R\partial_R\phi_d(R),\phi_0(R)\right>_{L^2_R},\\
			\mathcal{K}_{00}&=\left<R\partial_R\phi_0(R),\phi_0(R)\right>_{L^2_R},\\
			\mathcal{K}_{0c}f&=\left<\int_{0}^{\infty}f(\xi)R\partial_R\phi(R,\xi)\rho(\xi)d\xi,\phi_0(R)\right>_{L^2_R},\\
			\mathcal{K}_{cd}(\eta)&=\left<R\partial_R\phi_d(R),\phi(R,\eta)\right>_{L^2_R},\\
			\mathcal{K}_{c0}(\eta)&=\left<R\partial_R\phi_0(R),\phi(R,\eta)\right>_{L^2_R},\\
			\mathcal{K}_{cc}f(\eta)&=\left<\int_{0}^{\infty}f(\xi)R\partial_R\phi(R,\xi)\rho(\xi)d\xi,\phi(R,\eta)\right>_{L^2_R}+\left<\int_{0}^{\infty}2\xi\partial_\xi f(\xi)\phi(R,\xi)\rho(\xi)d\xi,\phi(R,\eta)\right>_{L^2_R},\\
		\end{aligned}
	\end{equation*}
	where the inner product is understood as
	\begin{equation*}
		\left<f(R),\phi(R,\xi)\right>_{L^2_R}:=\lim\limits_{r\to\infty}\int_{0}^{r}f(R)\phi(R,\xi)dR.
	\end{equation*}
	When $f(R)\phi(R,\xi)$ is integrable, this is just the usual $L^2$ inner product.
	
	Integrating by parts with repect to $R$ in the expressions for $\mathcal{K}_{dd}$ and $\mathcal{K}_{00}$, we get
	\begin{equation*}
		\mathcal{K}_{dd}=-\frac{1}{2}\left<\phi_d(R),\phi_d(R)\right>_{L^2_R}=-\frac{1}{2},\quad \mathcal{K}_{00}=-\frac{1}{2}\left<\phi_0(R),\phi_0(R)\right>_{L^2_R}=-\frac{1}{2}.
	\end{equation*}
	Similarily, using the fact that $\phi_d(R)$ and $\phi_0(R)$ are orthogonal in $L^2_R$ we get
	\begin{equation*}
		\mathcal{K}_{d0}=-\mathcal{K}_{0d}\in\mathbb{R}.
	\end{equation*}
	Since the double integral in the expression of $\mathcal{K}_{dc}f$ is absolutely convergent, we can change the order of integration to get
	\begin{equation*}
		\mathcal{K}_{dc}f=\int_{0}^{\infty}K_{dc}(\xi)f(\xi)d\xi,
	\end{equation*}
	where
	\begin{equation*}
		K_{dc}(\xi)=\rho(\xi)\left<R\partial_R\phi(R,\xi),\phi_d(R)\right>_{L^2_R}=-\rho(\xi)\left<R\partial_R\phi_d(R),\phi(R,\xi)\right>_{L^2_R}.
	\end{equation*}
	\par For $\mathcal{K}_{0c}f$, we can not change the order of integration directly because of the slow decay of $\phi_0(R)$ as $R\to+\infty$. Instead, we first do an integration by parts
	\begin{equation*}
		\begin{aligned}
			\mathcal{K}_{0c}f&=\left<\int_{0}^{\infty}\xi^{-1}f(\xi)R\partial_R\mathcal{L}\phi(R,\xi)\rho(\xi)d\xi,\phi_0(R)\right>_{L^2_R}\\
			&=\left<\int_{0}^{\infty}\xi^{-1}f(\xi)R\partial_R\phi(R,\xi)\rho(\xi)d\xi,\mathcal{L}\phi_0(R)\right>_{L^2_R}+\left<\int_{0}^{\infty}\xi^{-1}f(\xi)[R\partial_R,\mathcal{L}]\phi(R,\xi)\rho(\xi)d\xi,\phi_0(R)\right>_{L^2_R}\\
			&=-2\mathcal{K}_{0c}f + \int_{0}^{\infty}3K_{0c}(\xi)f(\xi)d\xi,\\
		\end{aligned}
	\end{equation*}
	where
	\begin{equation*}
		K_{0c}(\xi)=-\frac{2}{3}\xi^{-1}\rho(\xi)\left<U(R)\phi(R,\xi),\phi_0(R)\right>_{L^2_R}.
	\end{equation*}
	\par For $\mathcal{K}_{c0}$, we use the similar method
	\begin{equation*}
		\begin{aligned}
			\eta\mathcal{K}_{c0}(\eta)&=\lim\limits_{r\to+\infty}\int_{0}^{r}R\partial_R\phi_0(R)\mathcal{L}\phi(R,\eta)dR\\
			&=\lim\limits_{r\to+\infty}\int_{0}^{r}R\partial_R\mathcal{L}\phi_0(R)\phi(R,\eta)dR + \lim\limits_{r\to+\infty}\int_{0}^{r}[\mathcal{L},R\partial_R]\phi_0(R)\phi(R,\eta)dR\\
			&=\left<U(R)\phi_0(R),\phi(R,\eta))\right>_{L^2_R}.\\
		\end{aligned}
	\end{equation*}
	
	\par Finally, an integration by parts with respect to $\xi$ in the expression for $\mathcal{K}_{cc}$ yields
	\begin{equation}\label{eq:K_cc}
		\mathcal{K}_{cc}f(\eta)=\left<\int_{0}^{\infty}f(\xi)\left(R\partial_R-2\xi\partial_\xi\right)\phi(R,\xi)\rho(\xi)d\xi,\phi(R,\eta)\right>_{L^2_R}-2\left(1+\frac{\eta\rho'(\eta)}{\rho(\eta)}\right)f(\eta).
	\end{equation}

	\par Since as a priori, we have 
	\begin{equation*}
		\mathcal{K}_{cc}\colon C_0^\infty((0,+\infty))\to C^\infty((0,+\infty)),
	\end{equation*}
	by the Schwartz kernel theorem, there is a distribution $K_{cc}$ (called the Schwartz kernel of $\mathcal{K}_{cc}$) on $(0,+\infty)\times (0,+\infty)$ such that
	\begin{equation*}
		\mathcal{K}_{cc}f(\eta)=\int_{0}^{\infty}K_{cc}(\eta,\xi)f(\xi)d\xi.
	\end{equation*}

	\begin{proposition}\label{prop:bound}
		The Schwartz kernel $K_{cc}$ of the operator $\mathcal{K}_{cc}$ is of the form
		\begin{equation}\label{eq:form_of_K}
			K_{cc}(\eta,\xi)=\left(\frac{3}{2}+\frac{\eta\rho'(\eta)}{\rho(\eta)}\right)\delta(\xi-\eta)+\mathrm{p.v.}\frac{\rho(\xi)}{\xi-\eta}F(\xi,\eta),
		\end{equation}
		where $F(\xi,\eta)$ is of class $C^1$ in $(0,+\infty)\times(0,+\infty)$, of class $C^2$ in $(0,+\infty)\times(0,+\infty)-\{\xi=\eta\}$ and is symmertic in $\xi$ and $\eta$. Moreover, if $\beta>\frac{1}{2}$, then $F(\xi,\eta)$ is $C^2$ in $(0,+\infty)\times(0,+\infty)$.
	\end{proposition}
	
	\begin{proof}
		We first establish the off-diagonal behavior of $\mathcal{K}_{cc}$. Let $f\in C^\infty_0((0,+\infty))$ and let
		\begin{equation*}
			u(R)=\int_{0}^{\infty}f(\xi)\left(R\partial_R-2\xi\partial_\xi\right)\phi(R,\xi)\rho(\xi)d\xi
		\end{equation*}
		as in (\ref{eq:K_cc}). Since we have restrict $\xi$ to be in a compact subset of $(0,+\infty)$, $R\to 0$ and $R\to\infty$ is equivalent to $R^2\xi\to 0$ and $R^2\xi\to\infty$. So for small $R$, we have $|\left(R\partial_R-2\xi\partial_\xi\right)\phi(R,\xi)|\lesssim R^{\frac{\beta+1}{2}}$, hence $|u(R)|\lesssim R^{\frac{\beta+1}{2}}$. For large $R$, $u(R)$ behaves like a Schwartz function because of the oscillation of $\phi(R,\xi)$. Note that $\phi(R,\eta)$ also behaves like $R^{\frac{\beta+1}{2}}$ at $R=0$ and is bounded with bounded derivatives at $R=\infty$, the following integration by parts is verified:
		\begin{equation*}
			\eta\left<u(R),\phi(R,\eta)\right>_{L^2_R}=\left<u(R),\mathcal{L}\phi(R,\eta)\right>_{L^2_R}=\left<\mathcal{L}u(R),\phi(R,\eta)\right>_{L^2_R}.
		\end{equation*}
		Here,
		\begin{equation*}
			\begin{aligned}
				\mathcal{L}u(R)&=\int_{0}^{\infty}f(\xi)\left(R\partial_R-2\xi\partial_\xi\right)\xi\phi(R,\xi)\rho(\xi)d\xi+\int_{0}^{\infty}f(\xi)[\mathcal{L},R\partial_R]\phi(R,\xi)\rho(\xi)d\xi\\
				&=\int_{0}^{\infty}\xi f(\xi)\left(R\partial_R-2\xi\partial_\xi\right)\phi(R,\xi)\rho(\xi)d\xi\\
				&\quad -2\int_{0}^{\infty}\xi f(\xi)\phi(R,\xi)\rho(\xi)d\xi+\int_{0}^{\infty}f(\xi)[\mathcal{L},R\partial_R]\phi(R,\xi)\rho(\xi)d\xi,\\
			\end{aligned}
		\end{equation*}
		with
		\begin{equation*}
			[\mathcal{L},R\partial_R]=2\mathcal{L}+U(R).
		\end{equation*}
		Thus,
		\begin{equation*}
			\mathcal{L}u(R)=\int_{0}^{\infty}\xi f(\xi)\left(R\partial_R-2\xi\partial_\xi\right)\phi(R,\xi)\rho(\xi)d\xi+\int_{0}^{\infty}f(\xi)U(R)\phi(R,\xi)\rho(\xi)d\xi.
		\end{equation*}
		Hence, we obtain
		\begin{equation*}
			\eta\mathcal{K}_{cc}f(\eta)-\mathcal{K}_{cc}(\xi f)(\eta)=\left<\int_{0}^{\infty}f(\xi)U(R)\phi(R,\xi)\rho(\xi)d\xi,\phi(R,\eta)\right>_{L^2_R}.
		\end{equation*}
		Since the double integral on the RHS is absolutely convergent, we may change the order of integration to get
		\begin{equation*}
			(\eta-\xi)K_{cc}(\eta,\xi)=\rho(\xi)\left<U(R)\phi(R,\xi),\phi(R,\eta)\right>_{L^2_R}.
		\end{equation*}
		Let $F(\xi,\eta)=\left<U(R)\phi(R,\xi),\phi(R,\eta)\right>_{L^2_R}$, which gives the desired form of $K_{cc}$ as in (\ref{eq:form_of_K}) (of course, the coefficient in front of $\delta(\xi-\eta)$ has to be determined).
		\par Tthe $\delta$ measure component of $K_{cc}$ can be determined similarly as in \cite{KST09a}. 
	\end{proof}

	We now derive various bounds for $F(\xi,\eta)$ and it's partial derivatives in different regions. Due to the symmetry of $F(\xi,\eta)$, we will only consider the bounds below the diagonal $\{\xi\geq\eta\}$. 
	
	\subsection{Bounds for $F(\xi,\eta)$}
	Recall that $F(\xi,\eta)$ is given by the expression
	\begin{equation*}
		F(\xi,\eta)=\int_{0}^{\infty}U(R)\phi(R,\xi)\phi(R,\eta)dR.
	\end{equation*}

	\subsubsection{Bounds for $F(\xi,\eta)$ when $1\lesssim\frac{\xi}{4}\leq\eta\leq\xi$.}
	In this region, we have
		\begin{equation*}
			\begin{aligned}
				|F(\xi,\eta)|&\lesssim \int_{0}^{\xi^{-1/2}}R^{2\beta-2}R^{\frac{\beta+1}{2}}R^{\frac{\beta+1}{2}}dR+\int_{\xi^{-1/2}}^{\eta^{-1/2}}R^{2\beta-2}R^{\frac{\beta+1}{2}}\xi^{-\frac{\beta+1}{4}}dR\\
				&\quad +\int_{\eta^{-1/2}}^{1}R^{2\beta-2}\xi^{-\frac{\beta+1}{4}}\eta^{-\frac{\beta+1}{4}}dR+\int_{1}^{\infty}R^{-2\beta-2}\xi^{-\frac{\beta+1}{4}}\eta^{-\frac{\beta+1}{4}}dR\\
				&\lesssim \xi^{-\frac{3\beta}{2}}+\xi^{-\frac{\beta+1}{4}}\eta^{\frac{1-5\beta}{4}}+\xi^{-\frac{\beta+1}{4}}\eta^{-\frac{\beta+1}{4}}\\
				&\lesssim \xi^{-\frac{3\beta}{2}}+\xi^{-\frac{\beta+1}{2}},\\
			\end{aligned}
		\end{equation*}
		which seperates into three cases:
		\begin{equation*}
			\begin{aligned}
				|F(\xi,\eta)|&\lesssim\xi^{-\frac{3\beta}{2}},& 0<\beta\leq\frac{1}{2},\\
				|F(\xi,\eta)|&\lesssim\xi^{-\frac{3}{4}}\log(\xi),& \beta=\frac{1}{2},\\
				|F(\xi,\eta)|&\lesssim\xi^{-\frac{\beta+1}{2}},& \frac{1}{2}<\beta<4.\\
			\end{aligned}
		\end{equation*}
		
		\subsubsection{Bounds for $F(\xi,\eta)$ when $1\lesssim\eta\leq\frac{\xi}{2}$.} In this region, a similar calculation gives
		\begin{equation*}
			\begin{aligned}
				|F(\xi,\eta)|&\lesssim\xi^{-\frac{3\beta}{2}},& 0<\beta<\frac{1}{5},\\
				|F(\xi,\eta)|&\lesssim\xi^{-\frac{3}{10}}\log(\xi/\eta),& \beta=\frac{1}{5},\\
				|F(\xi,\eta)|&\lesssim\xi^{-\frac{\beta+1}{4}}\eta^{\frac{1-5\beta}{4}},& \frac{1}{2}<\beta<\frac{1}{2},\\
				|F(\xi,\eta)|&\lesssim\xi^{-\frac{3}{8}}\eta^{-\frac{3}{8}}\log(\xi/\eta),& \beta=\frac{1}{2},\\
				|F(\xi,\eta)|&\lesssim\xi^{-\frac{\beta+1}{4}}\eta^{-\frac{\beta+1}{4}},& \frac{1}{2}<\beta<4.\\
			\end{aligned}
		\end{equation*}

		\subsubsection{Bounds for $F(\xi,\eta)$ when $0<\eta\lesssim 1\lesssim\xi$.} In this region, we have
		\begin{equation*}
			\begin{aligned}
				|F(\xi,\eta)|&\lesssim \int_{0}^{\xi^{-1/2}}R^{2\beta-2}R^{\frac{\beta+1}{2}}R^{\frac{\beta+1}{2}}dR+\int_{\xi^{-1/2}}^{1}R^{2\beta-2}R^{\frac{\beta+1}{2}}\xi^{-\frac{\beta+1}{4}}dR\\
				&\quad +\int_{1}^{\eta^{-1/2}}R^{-2\beta-2}R^{\frac{-\beta+1}{2}}\xi^{-\frac{\beta+1}{4}}dR+\int_{\eta^{-1/2}}^{\infty}R^{-2\beta-2}\xi^{-\frac{\beta+1}{4}}\eta^{\frac{\beta-1}{4}}dR\\
				&\lesssim \xi^{-\frac{3\beta}{2}}+\xi^{-\frac{\beta+1}{4}},\\
			\end{aligned}
		\end{equation*}
		for $0<\beta<2$,
		\begin{equation*}
			\begin{aligned}
				|F(\xi,\eta)|&\lesssim \int_{0}^{\xi^{-1/2}}R^{2\beta-2}R^{\frac{\beta+1}{2}}R^{\frac{\beta+1}{2}}dR+\int_{\xi^{-1/2}}^{1}R^{2\beta-2}R^{\frac{\beta+1}{2}}\xi^{-\frac{\beta+1}{4}}dR\\
				&\quad +\int_{1}^{\eta^{-1/2}}R^{-2\beta-2}R^{\frac{-\beta+1}{2}}\log(R)\xi^{-\frac{\beta+1}{4}}dR+\int_{\eta^{-1/2}}^{\infty}R^{-2\beta-2}\xi^{-\frac{\beta+1}{4}}\eta^{\frac{\beta-1}{4}}|\log(\eta)|dR\\
				&\lesssim \xi^{-\frac{\beta+1}{4}},\\
			\end{aligned}
		\end{equation*}
		for $\beta=2$, and
		\begin{equation*}
			\begin{aligned}
				|F(\xi,\eta)|&\lesssim \int_{0}^{\xi^{-1/2}}R^{2\beta-2}R^{\frac{\beta+1}{2}}R^{\frac{\beta+1}{2}}dR+\int_{\xi^{-1/2}}^{1}R^{2\beta-2}R^{\frac{\beta+1}{2}}\xi^{-\frac{\beta+1}{4}}dR\\
				&\quad +\int_{1}^{\eta^{-1/2}}R^{-2\beta-2}R^{\frac{\beta-3}{2}}\xi^{-\frac{\beta+1}{4}}dR+\int_{\eta^{-1/2}}^{\infty}R^{-2\beta-2}\xi^{-\frac{\beta+1}{4}}\eta^{\frac{-\beta+3}{4}}dR\\
				&\lesssim\xi^{-\frac{\beta+1}{4}},\\
			\end{aligned}
		\end{equation*}
		for $\beta>2$. More precisely, we have
		\begin{equation*}
			\begin{aligned}
				|F(\xi,\eta)|&\lesssim \xi^{-\frac{3\beta}{2}},& 0<\beta<\frac{1}{5},\\
				|F(\xi,\eta)|&\lesssim \xi^{-\frac{3}{10}}\log(\xi),& \beta=\frac{1}{5},\\
				|F(\xi,\eta)|&\lesssim \xi^{-\frac{\beta+1}{4}},&\beta>\frac{1}{5}.\\
			\end{aligned}
		\end{equation*}

		\subsubsection{Bounds for $F(\xi,\eta)$ when $0<\eta\leq\xi\lesssim 1$.} In this region, we have
		\begin{equation*}
			\begin{aligned}
				|F(\xi,\eta)|&\lesssim \int_{0}^{1}R^{2\beta-2}R^{\frac{\beta+1}{2}}R^{\frac{\beta+1}{2}}dR+\int_{1}^{\xi^{-1/2}}R^{-2\beta-2}R^{\frac{-\beta+1}{2}}R^{\frac{-\beta+1}{2}}dR\\
				&\quad +\int_{\xi^{-1/2}}^{\eta^{-1/2}}R^{-2\beta-2}R^{\frac{-\beta+1}{2}}\xi^{\frac{\beta-1}{4}}dR+\int_{\eta^{-1/2}}^{\infty}R^{-2\beta-2}\xi^{\frac{\beta-1}{4}}\eta^{\frac{\beta-1}{4}}dR\\
				&\lesssim 1.\\
			\end{aligned}
		\end{equation*}
		However, it turns out that this constant bound is far from being optimal. In fact, by direct integration, we have $F(0,0)=0$. Furthermore, since when $\xi\leq 1$, we have
		\begin{equation*}
			\begin{aligned}
				|\phi(R,\xi)|&\lesssim\max\{\xi^{\frac{\beta-1}{4}},\xi^{\frac{3-\beta}{4}}\}\mathbf{1}_{R^2\xi\geq 1}+R^{\frac{\beta+1}{2}}\max\left\{\left<R^{2\beta}\right>^{-\frac{1}{2}},\left<R^{2\beta}\right>^{-\frac{1}{\beta}}\right\}\mathbf{1}_{R^2\xi\leq1}\\&\lesssim R^{\frac{\beta+1}{2}}\mathbf{1}_{R<1}+\max\{R^{\frac{-\beta+1}{2}},R^{\frac{\beta-3}{2}},1\}\mathbf{1}_{R\geq1},
			\end{aligned}
		\end{equation*}
		which implies
		\begin{equation*}
			\begin{aligned}
				|U(R)\phi(R,\xi)\phi(R,\eta)|&\lesssim |U(R)|\left(R^{\beta+1}\mathbf{1}_{R<1}+\max\{R^{-\beta+1},R^{\frac{\beta-3}{2}},1\}\mathbf{1}_{R\geq 1}\right)\\&\lesssim R^{3\beta-1}\mathbf{1}_{R<1}+\max\{R^{-3\beta-1},R^{-\beta-5},R^{-2\beta-2}\}\mathbf{1}_{R\geq 1},	
			\end{aligned}
		\end{equation*}
		where the RHS is integrable on $(0,+\infty)$. By the dominated convergence theorem, we get the continuity of $F(\xi,\eta)$ at $(0,0)$. We will use the bound for $\partial_\xi F(\xi,\eta)$ and $\partial_\eta F(\xi,\eta)$ (derived in subsequent), the fact that $F(0,0)=0$ and the continuity of $F(\xi,\eta)$ at $(0,0)$ to get a better bound for $F(\xi,\eta)$ when $\xi+\eta\lesssim1$:
		\begin{equation*}
		\begin{aligned}
			F(\xi,\eta)&=\lim\limits_{(\xi,\eta)\to(0^+,0^+)}F(\xi,\eta)+\int_{0}^{1}\left(\xi \partial_\xi F(t\xi,t\eta) +\eta \partial_\eta F(t\xi,t\eta)\right) dt
			\\&=\int_{0}^{1}\left(\xi \partial_\xi F(t\xi,t\eta) +\eta \partial_\eta F(t\xi,t\eta)\right) dt.
		\end{aligned}
		\end{equation*}
		As a result, we find better bounds for $F(\xi,\eta)$ compared to constant bound.
		\begin{equation*}
			\begin{aligned}
				|F(\xi,\eta)|&\lesssim \int_{0}^{1} \left(\xi^{\frac{\beta+1}{4}}\eta^{\frac{5\beta-1}{4}}+ \xi^{\frac{\beta-1}{4}}\eta^{\frac{5\beta+1}{4}}\right)t^{\frac{3\beta}{2}-1}dt\lesssim \xi^{\frac{\beta+1}{4}}\eta^{\frac{5\beta-1}{4}},& 0<\beta<\frac{1}{5},\\
				|F(\xi,\eta)|&\lesssim \int_{0}^{1}\left(\xi^{\frac{3}{10}}\log(\xi/\eta)+\xi^{-\frac{1}{5}}\eta^{\frac{1}{2}}\right)t^{-\frac{7}{10}}dt\lesssim \xi^{\frac{3}{10}}\log(\xi/\eta),& \beta=\frac{1}{5},\\
				|F(\xi,\eta)|&\lesssim \int_{0}^{1} \left(\xi^{\frac{3\beta}{2}}+ \xi^{\frac{\beta-1}{4}}\eta^{\frac{5\beta+1}{4}}\right)t^{\frac{3\beta}{2}-1}dt\lesssim \xi^{\frac{3\beta}{2}},&\frac{1}{5}<\beta<\frac{3}{5},\\
				|F(\xi,\eta)|&\lesssim \int_{0}^{1}\left(\xi^{\frac{9}{10}}+\xi^{-\frac{1}{10}}\eta\log(\xi/\eta)\right)t^{-\frac{1}{10}}dt\lesssim \xi^{\frac{9}{10}}\log(\xi/\eta),& \beta=\frac{3}{5},\\
				|F(\xi,\eta)|&\lesssim \int_{0}^{1}\left(\xi^{\frac{3\beta}{2}}+\xi^{\frac{3\beta}{2}-1}\eta\right)t^{\frac{3\beta}{2}-1}dt\lesssim \xi^{\frac{3\beta}{2}},& \frac{3}{5}<\beta<\frac{2}{3},\\
				|F(\xi,\eta)|&\lesssim \int_{0}^{1}\left(-t\xi\log(t\xi/2)-t\eta\log(t\xi/2)\right)dt\lesssim -\xi\log(\xi/2),& \beta=\frac{2}{3},\\
				|F(\xi,\eta)|&\lesssim \int_{0}^{1}\left(\xi+\eta\right)dt\lesssim \xi,&\beta>\frac{2}{3}.\\
			\end{aligned}
		\end{equation*}

		\subsection{Bounds for $\nabla F(\xi,\eta)$.}
		The expressions of $\nabla F(\xi,\eta)$ are given by
		\begin{equation*}
			\begin{aligned}
				\partial_\xi F(\xi,\eta) &= \int_{0}^{\infty}U(R)\partial_\xi\phi(R,\xi)\phi(R,\eta)dR,\\
				\partial_\eta F(\xi,\eta) &= \int_{0}^{\infty}U(R)\phi(R,\xi)\partial_\eta\phi(R,\eta)dR.\\
			\end{aligned}
		\end{equation*}
		Through integration by parts, we get
		\begin{equation*}
			\begin{aligned}
				\eta\partial_\xi F(\xi,\eta)&=\left<U(R)\partial_\xi\phi(R,\xi),\mathcal{L}\phi(R,\eta)\right>_{L^2_R}\\&=\xi\partial_\xi F(\xi,\eta)+F(\xi,\eta)+\left<[\mathcal{L},U(R)]\partial_\xi\phi(R,\xi),\phi(R,\eta)\right>_{L^2_R},
			\end{aligned}
		\end{equation*}
		where 
		\begin{equation*}
			[\mathcal{L},U(R)]=2U'(R)\partial_R+U''(R).
		\end{equation*}
		Similarly, we have
		\begin{equation*}
			(\xi-\eta)\partial_\eta F(\xi,\eta)=\int_{0}^{\infty}\phi(R,\xi)[\mathcal{L},U(R)]\partial_\eta \phi(R,\eta)dR + F(\xi,\eta).
		\end{equation*}
		We will use last two expressions to get the off-diagonal estimate of $\nabla F(\xi,\eta)$.

		\subsubsection{Bounds for $\nabla F(\xi,\eta)$ when $1\lesssim \frac{\xi}{4}\leq\eta\leq\xi$.} In this region ,we have
		\begin{equation*}
			\begin{aligned}
				|\partial_\xi F(\xi,\eta)|&\lesssim \int_{0}^{\xi^{-\frac{1}{2}}}R^{2\beta-2}R^{\frac{\beta+5}{2}}R^{\frac{\beta+1}{2}}dR+\int_{\xi^{-\frac{1}{2}}}^{\eta^{-\frac{1}{2}}}R^{2\beta-2}R\xi^{-\frac{\beta+3}{4}}R^{\frac{\beta+1}{2}}dR\\
				&\quad +\int_{\eta^{-\frac{1}{2}}}^{1}R^{2\beta-2}R\xi^{-\frac{\beta+3}{4}}\eta^{-\frac{\beta+1}{4}}dR+\int_{1}^{\infty}R^{-2\beta-2}R\xi^{-\frac{\beta+3}{4}}\eta^{-\frac{\beta+1}{4}}dR\\
				&\lesssim \xi^{-\frac{\beta}{2}-1}.\\
			\end{aligned}
		\end{equation*}
		The same estimate holds for $\partial_\eta F(\xi,\eta)$ in this region.

		\subsubsection{Bounds for $\nabla F(\xi,\eta)$ when $1\lesssim\eta\leq\frac{\xi}{2}$.} In this region, we have
		\begin{equation*}
			\begin{aligned}
				|(\eta-\xi)\partial_\xi F(\xi,\eta)|&\lesssim \int_{0}^{\xi^{-\frac{1}{2}}}R^{2\beta-4}R^{\frac{\beta+5}{2}}R^{\frac{\beta+1}{2}}dR+\int_{\xi^{-\frac{1}{2}}}^{\eta^{-\frac{1}{2}}}R^{2\beta-4}R^2\xi^{-\frac{\beta+1}{4}}R^{\frac{\beta+1}{2}}dR\\
				&\quad +\int_{\eta^{-\frac{1}{2}}}^{1}R^{2\beta-4}R^2\xi^{-\frac{\beta+1}{4}}\eta^{-\frac{\beta+1}{4}}dR+\int_{1}^{\infty}R^{-2\beta-4}R^2\xi^{-\frac{\beta+1}{4}}\eta^{-\frac{\beta+1}{4}}dR+|F(\xi,\eta)|\\
				&\lesssim \xi^{-\frac{3\beta}{2}}+\xi^{-\frac{\beta+1}{4}}\eta^{\frac{1-5\beta}{4}}+\xi^{-\frac{\beta+1}{4}}\eta^{-\frac{\beta+1}{4}},\\
			\end{aligned}
		\end{equation*}
		hence $(\eta-\xi)\partial_\xi F(\xi,\eta)$ has the same bound as $F(\xi,\eta)$. Similarly, one can show that $\partial_\eta F(\xi,\eta)$ also has the same bound.

		\subsubsection{Bounds for $\nabla F(\xi,\eta)$ when $0<\eta\lesssim 1\lesssim\xi$.}

		In this region, we have
		\begin{equation*}
			\begin{aligned}
				|(\eta-\xi)\partial_\xi F(\xi,\eta)|&\lesssim \int_{0}^{\xi^{-\frac{1}{2}}}R^{2\beta-4}R^{\frac{\beta+5}{2}}R^{\frac{\beta+1}{2}}dR+\int_{\xi^{-\frac{1}{2}}}^{1}R^{2\beta-4}R^2\xi^{-\frac{\beta+1}{4}}R^{\frac{\beta+1}{2}}dR\\
				&\quad +\int_{1}^{\eta^{-\frac{1}{2}}}R^{-2\beta-4}R^2\xi^{-\frac{\beta+1}{4}}R^{\frac{-\beta+1}{2}}dR+\int_{\eta^{-\frac{1}{2}}}^{\infty}R^{-2\beta-4}R^2\xi^{-\frac{\beta+1}{4}}\eta^{\frac{\beta-1}{4}}dR\\&\quad +|F(\xi,\eta)|\\
				&\lesssim \xi^{-\frac{3\beta}{2}}+\xi^{-\frac{\beta+1}{4}},\\
			\end{aligned}
		\end{equation*}
		for $0<\beta<2$,
		\begin{equation*}
			\begin{aligned}
				|(\eta-\xi)\partial_\xi F(\xi,\eta)|&\lesssim \int_{0}^{\xi^{-\frac{1}{2}}}R^{2\beta-4}R^{\frac{\beta+5}{2}}R^{\frac{\beta+1}{2}}dR+\int_{\xi^{-\frac{1}{2}}}^{1}R^{2\beta-4}R^2\xi^{-\frac{\beta+1}{2}}R^{\frac{\beta+1}{2}}dR\\
				&\quad +\int_{1}^{\eta^{-\frac{1}{2}}}R^{-2\beta-4}R^2\xi^{-\frac{\beta+1}{4}}R^{\frac{-\beta+1}{2}}\log(R)dR\\&\quad +\int_{\eta^{-\frac{1}{2}}}^{\infty}R^{-2\beta-4}R^2\xi^{-\frac{\beta+1}{4}}\eta^{\frac{\beta-1}{4}}|\log(\eta)|dR\\&\quad +|F(\xi,\eta)|\\
				&\lesssim \xi^{-\frac{\beta+1}{4}},\\
			\end{aligned}
		\end{equation*}
		for $\beta=2$, and
		\begin{equation*}
			\begin{aligned}
				|(\eta-\xi)\partial_\xi F(\xi,\eta)|&\lesssim \int_{0}^{\xi^{-\frac{1}{2}}}R^{2\beta-4}R^{\frac{\beta+5}{2}}R^{\frac{\beta+1}{2}}dR+\int_{\xi^{-\frac{1}{2}}}^{1}R^{2\beta-4}R^2\xi^{-\frac{\beta+1}{4}}R^{\frac{\beta+1}{2}}dR\\
				&\quad +\int_{1}^{\eta^{-\frac{1}{2}}}R^{-2\beta-4}R^2\xi^{-\frac{\beta+1}{4}}R^{\frac{\beta-3}{2}}dR+\int_{\eta^{-\frac{1}{2}}}^{\infty}R^{-2\beta-4}R^2\xi^{-\frac{\beta+1}{4}}\eta^{\frac{-\beta+3}{4}}dR\\&\quad +|F(\xi,\eta)|\\
				&\lesssim \xi^{-\frac{\beta+1}{4}},\\
			\end{aligned}
		\end{equation*}
		for $\beta>2$. That is to say, $(\eta-\xi)\partial_\xi F(\xi,\eta)$ has the same bound as $F(\xi,\eta)$ in this region. Similarly, one can show that $(\eta-\xi)\partial_\eta F(\xi,\eta)$ also has the same bound.

		\subsubsection{Bounds for $\nabla F(\xi,\eta)$ when $0<\eta\leq\xi\lesssim 1$.} In this region, we have
		\begin{equation*}
			\begin{aligned}
				|\partial_\xi F(\xi,\eta)|&\lesssim \int_{0}^{1}R^{2\beta-2}R^{\frac{\beta+5}{2}}R^{\frac{\beta+1}{2}}dR+\int_{1}^{\xi^{-\frac{1}{2}}}R^{-2\beta-2}R^{\frac{-\beta+5}{2}}R^{\frac{-\beta+1}{2}}dR\\
				&\quad +\int_{\xi^{-\frac{1}{2}}}^{\eta^{-\frac{1}{2}}}R^{-2\beta-2}R\xi^{\frac{\beta-3}{4}}R^{\frac{-\beta+1}{2}}dR+\int_{\eta^{-\frac{1}{2}}}^{\infty}R^{-2\beta-2}R\xi^{\frac{\beta-3}{4}}\eta^{\frac{\beta-1}{4}}dR\\
				&\lesssim 1+\xi^{\frac{3\beta}{2}-1}+\xi^{\frac{\beta-3}{4}}\eta^{\frac{5\beta-1}{4}},\\
			\end{aligned}
		\end{equation*}
		for $0<\beta<2$, 
		\begin{equation*}
			\begin{aligned}
				|\partial_\xi F(\xi,\eta)|&\lesssim \int_{0}^{1}R^{2\beta-2}R^{\frac{\beta+5}{2}}R^{\frac{\beta+1}{2}}dR+\int_{1}^{\xi^{-\frac{1}{2}}}R^{-2\beta-2}R^{\frac{-\beta+5}{2}}R^{\frac{-\beta+1}{2}}(\log(R))^2dR\\
				&\quad +\int_{\xi^{-\frac{1}{2}}}^{\eta^{-\frac{1}{2}}}R^{-2\beta-2}R\xi^{\frac{\beta-3}{4}}|\log(\xi)|R^{\frac{-\beta+1}{2}}\log(R)dR\\&\quad +\int_{\eta^{-\frac{1}{2}}}^{\infty}R^{-2\beta-2}R\xi^{\frac{\beta-3}{4}}|\log(\xi)|\eta^{\frac{\beta-1}{4}}|\log(\eta)|dR\\
				&\lesssim 1,\\
			\end{aligned}
		\end{equation*}
		for $\beta=2$, and
		\begin{equation*}
			\begin{aligned}
				|\partial_\xi F(\xi,\eta)|&\lesssim \int_{0}^{1}R^{2\beta-2}R^{\frac{\beta+5}{2}}R^{\frac{\beta+1}{2}}dR+\int_{1}^{\xi^{-\frac{1}{2}}}R^{-2\beta-2}R^{\frac{\beta+1}{2}}R^{\frac{\beta-3}{2}}dR\\
				&\quad +\int_{\xi^{-\frac{1}{2}}}^{\eta^{-\frac{1}{2}}}R^{-2\beta-2}R\xi^{\frac{-\beta+1}{4}}R^{\frac{\beta-3}{2}}dR+\int_{\eta^{-\frac{1}{2}}}^{\infty}R^{-2\beta-2}R\xi^{\frac{-\beta+1}{4}}\eta^{\frac{-\beta+3}{4}}dR\\
				&\lesssim 1,\\
			\end{aligned}
		\end{equation*}
		for $\beta>2$. More precisely, we have
		\begin{equation*}
			\begin{aligned}
				|\partial_\xi F(\xi,\eta)|&\lesssim \xi^{\frac{\beta-3}{4}}\eta^{\frac{5\beta-1}{4}},& 0<\beta<\frac{1}{5},\\
				|\partial_\xi F(\xi,\eta)|&\lesssim \xi^{-\frac{7}{10}}\log(\xi/\eta),& \beta=\frac{1}{5},\\
				|\partial_\xi F(\xi,\eta)|&\lesssim \xi^{\frac{3\beta}{2}-1},& \frac{1}{5}<\beta<\frac{2}{3},\\
				|\partial_\xi F(\xi,\eta)|&\lesssim -\log(\xi/2),& \beta=\frac{2}{3},\\
				|\partial_\xi F(\xi,\eta)|&\lesssim 1,& \beta>\frac{2}{3}.\\
			\end{aligned}
		\end{equation*}
		
		Similarly, one can show that $\partial_\eta F(\xi,\eta)$ satisfies

		\begin{equation*}
			\begin{aligned}
				|\partial_\eta F(\xi,\eta)|&\lesssim \xi^{\frac{\beta-1}{4}}\eta^{\frac{5\beta-3}{4}},& 0<\beta<\frac{3}{5},\\
				|\partial_\eta F(\xi,\eta)|&\lesssim \xi^{-\frac{1}{10}}\log(\xi/\eta),&\beta=\frac{3}{5},\\
				|\partial_\eta F(\xi,\eta)|&\lesssim \xi^{\frac{3\beta}{2}-1},& \frac{3}{5}<\beta<\frac{2}{3},\\
				|\partial_\eta F(\xi,\eta)|&\lesssim -\log(\xi/2),&\beta=\frac{2}{3},\\
				|\partial_\eta F(\xi,\eta)|&\lesssim 1,&\beta>\frac{2}{3}.\\
			\end{aligned}
		\end{equation*}

		\subsection{Bounds for $\nabla^2 F(\xi,\eta)$.}

		Note that only when $\beta>\frac{1}{2}$, $U(R)\partial_\xi^2\phi(R,\xi)\phi(R,\eta)$ is integrable at $R=+\infty$, so we can verify that 
		\begin{equation}\label{eq:21}
			\partial_\xi^2 F(\xi,\eta)=\left<U(R)\partial_\xi^2\phi(R,\xi),\phi(R,\eta)\right>_{L^2_R}
		\end{equation}
		using the dominated convergence theorem. For general $\beta>0$, we can use the following indetity
		\begin{equation*}
			(\eta-\xi)\partial_\xi F(\xi,\eta) = F(\xi,\eta) + \left<[\mathcal{L},U(R)]\partial_\xi\phi(R,\xi),\phi(R,\eta)\right>_{L^2_R},
		\end{equation*}
		and the fact that $[\mathcal{L},U(R)]\partial_\xi^2\phi(R,\xi)\phi(R,\eta)$ is integrable at $R=+\infty$, which implies that away from the diagonal $\{\xi=\eta\}$, $\partial_\xi^2 F(\xi,\eta)$ actually exists:
		\begin{equation*}
			(\eta-\xi)\partial_\xi^2 F(\xi,\eta)=2\partial_\xi F(\xi,\eta)+\left<[\mathcal{L},U(R)]\partial_\xi^2\phi(R,\xi),\phi(R,\eta)\right>_{L^2_R}.
		\end{equation*}
		Similarly, we have
		\begin{equation*}
			(\xi-\eta)\partial_\eta^2 F(\xi,\eta)=2\partial_\eta F(\xi,\eta)+\left<[\mathcal{L},U(R)]\partial_\eta^2\phi(R,\eta),\phi(R,\xi)\right>_{L^2_R},
		\end{equation*}
		and
		\begin{equation*}
			(\eta-\xi)\partial_{\xi\eta}^2 F(\xi,\eta)=\partial_\eta F(\xi,\eta)-\partial_\xi F(\xi,\eta)+\left<[\mathcal{L},U(R)]\partial_\xi\phi(R,\xi),\partial_\eta\phi(R,\eta)\right>_{L^2_R}.
		\end{equation*}
		Moreover, another integration by parts gives
		\begin{equation}\label{eq:22}
			(\eta-\xi)^2\partial_\xi^2 F(\xi,\eta)=4(\eta-\xi)\partial_\xi F(\xi,\eta)-2F(\xi,\eta)+\left<[\mathcal{L},[\mathcal{L},U(R)]]\partial_\xi^2\phi(R,\xi),\phi(R,\eta)\right>_{L^2_R},
		\end{equation}
		where
		\begin{equation*}
			[\mathcal{L},[\mathcal{L},U(R)]]=4U''(R)\partial_R^2+4U'''(R)\partial_R+U''''(R)+2R^{-1}U'(R)\left(2\alpha R^{-2}+R\partial_R(5W_\alpha^4(R))\right).
		\end{equation*}

		\subsubsection{Bounds for $\nabla^2F(\xi,\eta)$ when $1\lesssim\frac{\xi}{4}\leq\eta\leq\xi$.} In this region, using (\ref{eq:21}), we have

		\begin{equation*}
		\begin{aligned}
		|\partial_\xi^2 F(\xi,\eta)|&\lesssim \int_{0}^{\xi^{-\frac{1}{2}}}R^{2\beta-2}R^{\frac{\beta+9}{2}}R^{\frac{\beta+1}{2}}dR+\int_{\xi^{-\frac{1}{2}}}^{\eta^{-\frac{1}{2}}}R^{2\beta-2}R^{2}\xi^{-\frac{\beta+5}{4}}R^{\frac{\beta+1}{2}}dR\\
		&\quad +\int_{\eta^{-\frac{1}{2}}}^{1}R^{2\beta-2}R^2\xi^{-\frac{\beta+5}{4}}\eta^{-\frac{\beta+1}{4}}dR+\int_{1}^{\infty}R^{-2\beta-2}R^2\xi^{-\frac{\beta+5}{4}}\eta^{-\frac{\beta+1}{4}}dR\\
		&\lesssim \xi^{-\frac{\beta+3}{2}}.\\
		\end{aligned}
		\end{equation*}
		
		Using (\ref{eq:22}), we have
		\begin{equation*}
		\begin{aligned}
		|(\eta-\xi)^2\partial_\xi^2 F(\xi,\eta)|&\lesssim \int_{0}^{\xi^{-\frac{1}{2}}}R^{2\beta-6}R^{\frac{\beta+9}{2}}R^{\frac{\beta+1}{2}}dR+\int_{\xi^{-\frac{1}{2}}}^{\eta^{-\frac{1}{2}}}R^{2\beta-6}R^4\xi^{-\frac{\beta+1}{4}}R^{\frac{\beta+1}{2}}dR\\
		&\quad +\int_{\eta^{-\frac{1}{2}}}^{1}R^{2\beta-6}R^4\xi^{-\frac{\beta+1}{4}}\eta^{-\frac{\beta+1}{4}}dR+\int_{1}^{\infty}R^{-2\beta-6}R^4\xi^{-\frac{\beta+1}{4}}\eta^{-\frac{\beta+1}{4}}dR\\
		&\quad +4|(\eta-\xi)\partial_\xi F(\xi,\eta)|+2|F(\xi,\eta)|\\
		&\lesssim \xi^{-\frac{3\beta}{2}}+\xi^{-\frac{\beta+1}{2}}.\\
		\end{aligned}
		\end{equation*}
		More precisely,
		\begin{equation*}
		\begin{aligned}
		|(\eta-\xi)^2\partial_\xi^2 F(\xi,\eta)|&\lesssim \xi^{-\frac{3\beta}{2}},& 0<\beta<\frac{1}{2},\\
		|(\eta-\xi)^2\partial_\xi^2 F(\xi,\eta)|&\lesssim \xi^{-\frac{3}{4}}\log(\xi),& \beta=\frac{1}{2},\\
		|(\eta-\xi)^2\partial_\xi^2 F(\xi,\eta)|&\lesssim \xi^{-\frac{\beta+1}{2}},& \frac{1}{2}<\beta<4.\\
		\end{aligned}
		\end{equation*}

		Note that for $\beta>\frac{1}{2}$, the estimate derived from (\ref{eq:21}) would be better when $|\xi^{\frac{1}{2}}-\eta^{\frac{1}{2}}|\lesssim 1$, while the estimate from (\ref{eq:22}) would be better when $|\xi^{\frac{1}{2}}-\eta^{\frac{1}{2}}|\gtrsim 1$. Combinaing them, we get
		\begin{equation*}
			|\partial_\xi^2 F(\xi,\eta)|\lesssim \frac{\xi^{-\frac{\beta+3}{2}}}{1+|\xi^{\frac{1}{2}}-\eta^{\frac{1}{2}}|^2}.
		\end{equation*}

		Similarly, we have
		\begin{equation*}
			|\partial_\xi^2 F(\xi,\eta)|\lesssim \frac{\xi^{-\frac{\beta+3}{2}}}{1+|\xi^{\frac{1}{2}}-\eta^{\frac{1}{2}}|^2},\quad|\partial^2_{\xi\eta} F(\xi,\eta)|\lesssim \frac{\xi^{-\frac{\beta+3}{2}}}{1+|\xi^{\frac{1}{2}}-\eta^{\frac{1}{2}}|^2}. 
		\end{equation*}

		\subsubsection{Bounds for $\nabla^2 F(\xi,\eta)$ when $0<\frac{\xi}{4}\leq\eta\leq\xi\lesssim 1$.} 
		
		In this region, using (\ref{eq:21}), we have
		\begin{equation*}
			\begin{aligned}
				|\partial_\xi^2 F(\xi,\eta)|&\lesssim \int_{0}^{1}R^{2\beta-2}R^{\frac{\beta+9}{2}}R^{\frac{\beta+1}{2}}dR+\int_{1}^{\xi^{-\frac{1}{2}}}R^{-2\beta-2}R^{\frac{-\beta+9}{2}}R^{\frac{-\beta+1}{2}}dR\\
				&\quad +\int_{\xi^{-\frac{1}{2}}}^{\eta^{-\frac{1}{2}}}R^{-2\beta-2}R^2\xi^{\frac{\beta-5}{4}}R^{\frac{-\beta+1}{2}}dR+\int_{\eta^{-\frac{1}{2}}}^{\infty}R^{-2\beta-2}R^2\xi^{\frac{\beta-5}{4}}\eta^{\frac{\beta-1}{4}}dR\\
				&\lesssim 1+\xi^{\frac{3\beta}{2}-2}|\log(\xi)|,\\
			\end{aligned}
		\end{equation*}
		for $\frac{1}{2}<\beta<2$,
		\begin{equation*}
			\begin{aligned}
				|\partial_\xi^2 F(\xi,\eta)|&\lesssim \int_{0}^{1}R^{2\beta-2}R^{\frac{\beta+9}{2}}R^{\frac{\beta+1}{2}}dR+\int_{1}^{\xi^{-\frac{1}{2}}}R^{-2\beta-2}R^{\frac{-\beta+9}{2}}R^{\frac{-\beta+1}{2}}(\log(R))^2dR\\
				&\quad +\int_{\xi^{-\frac{1}{2}}}^{\eta^{-\frac{1}{2}}}R^{-2\beta-2}R^2\xi^{\frac{\beta-5}{4}}|\log(\xi)|R^{\frac{-\beta+1}{2}}\log(R)dR\\&\quad+\int_{\eta^{-\frac{1}{2}}}^{\infty}R^{-2\beta-2}R^2\xi^{\frac{\beta-5}{4}}|\log(\xi)|\eta^{\frac{\beta-1}{4}}|\log(\eta)|dR\\
				&\lesssim 1,\\
			\end{aligned}
		\end{equation*}
		for $\beta=2$, and
		\begin{equation*}
			\begin{aligned}
				|\partial_\xi^2 F(\xi,\eta)|&\lesssim \int_{0}^{1}R^{2\beta-2}R^{\frac{\beta+9}{2}}R^{\frac{\beta+1}{2}}dR+\int_{1}^{\xi^{-\frac{1}{2}}}R^{-2\beta-2}R^{\frac{\beta+5}{2}}R^{\frac{\beta-3}{2}}dR\\
				&\quad +\int_{\xi^{-\frac{1}{2}}}^{\eta^{-\frac{1}{2}}}R^{-2\beta-2}R^2\xi^{\frac{-\beta-1}{4}}R^{\frac{\beta-3}{2}}dR+\int_{\eta^{-\frac{1}{2}}}^{\infty}R^{-2\beta-2}R^2\xi^{\frac{-\beta-1}{4}}\eta^{\frac{-\beta+3}{4}}dR\\
				&\lesssim 1,\\
			\end{aligned}
		\end{equation*}
		for $\beta>2$.

		Till now, we conclude the derivation of bounds for $F$, $\nabla F$ and $\nabla^2 F$ in various regions. Based on these bounds, we come to study the $L^p$ boundedness of $\mathcal{K}$ and $[\xi\partial_\xi,\mathcal{K}]$.

	\begin{definition}
		For $s\geq0$, we define the weighted $L^2$ norm of $f\in L^2(\mathbb{R},d\rho)$ as
		\begin{equation*}
		\|f\|_{L^{2,s}_\rho}^2:=|f(\xi_d)|^2+|f(0)|^2+\int_{0}^{\infty}\left<\xi\right>^{2s} |f(\xi)|^2\rho(\xi)d\xi.
		\end{equation*}
		The weighted $L^2$ space is then defined as
		\begin{equation*}
		L^{2,s}_\rho:=\left\{f\in L^2(\mathbb{R},d\rho)\colon \|f\|_{L^{2,s}_\rho}<+\infty\right\}.
		\end{equation*}
	\end{definition}

	\begin{proposition}
		Let $-\frac{1}{4}+\frac{1}{16}<\alpha<\frac{15}{4}$, i.e. $\frac{1}{2}<\beta<4$. The following two operators are bounded
		\begin{equation*}
			\mathcal{K}\colon L^{2,s}_\rho \to L^{2,s}_\rho,\quad [\mathcal{K},\xi\partial_\xi]\colon L^{2,s}_\rho \to L^{2,s}_\rho,
		\end{equation*}
		for any $0\leq s <\min\{\frac{5\beta}{4}+\frac{1}{2},\frac{3}{4}\}$.
	\end{proposition}

	\begin{proof}
		We devide this proof into several parts.
		\par \underline{\textbf{(1) Boundedness of $\mathcal{K}_{cc}\colon L^{2,\alpha}_{\rho}((0,+\infty))\to L^{2,\alpha}_{\rho}((0,+\infty))$}.} This is equivalent to prove that the integral kernel
		\begin{equation*}
			\widetilde{K}_0(\eta,\xi):=\left<\eta\right>^{s}\left<\xi\right>^{-s}\frac{\sqrt{\rho(\xi)\rho(\eta)}}{\xi-\eta}F(\xi,\eta)=:\frac{\widetilde{F}(\xi,\eta)}{\xi-\eta}.
		\end{equation*}
		induces a bounded operator on $L^2((0,+\infty))$.
		\par For this purpose, we devide the double space $(0,+\infty)\times (0,+\infty)$ into several regions as follows. The strategy is then to prove the boundedness of kernel aftering multiplying the indicator functions of these region seperately.
		\par Let $A=[0,1]\times[0,1]$ be the unit square and $D=\{(\xi,\eta)\mid 2\eta\leq\xi\leq\frac{\eta}{2}\}$ be a ``linear region" near the diagonal $\{\xi=\eta\}$. 
		For those regions that are far away from the diagonal, we denote them by $F_1=\{(\xi,\eta)\mid \xi\geq1, 1\leq\eta\leq\frac{\xi}{2}\}$, $B_1=\{(\xi,\eta)\mid \xi\geq1, \eta\leq1\}$, as well as their transpose, i.e. $F_2=\{(\xi,\eta)\mid \eta\geq1, 1\leq\xi\leq\frac{\eta}{2}\}$, $B_2=\{(\xi,\eta)\mid \eta\geq1, \xi\leq1\}$. We also let $E=\{(\xi,\eta)\mid 0<\eta\leq\xi\leq1\}$.
		
		\par \underline{\textit{(1.1) Boundness of $\widetilde{K}_0\mathbf{1}_{D\cap A^c}$}.} A Hilbert transform type operator is extracted through
		\begin{equation*}
			\widetilde{K}_0(\xi,\eta)\mathbf{1}_{D\cap A^c}=\frac{\widetilde{F}(\eta,\eta)}{\xi-\eta}\mathbf{1}_{D\cap A^c}+\frac{\widetilde{F}(\xi,\eta)-\widetilde{F}(\eta,\eta)}{\xi-\eta}\mathbf{1}_{D\cap A^c}.
		\end{equation*}
		By the estimates given in Proposition \ref{prop:bound}, we have
		\begin{equation*}
			|\widetilde{F}(\eta,\eta)\mathbf{1}_{D\cap A^c}|\lesssim \eta^{\frac{\beta}{2}}\max\{\eta^{-\frac{3\beta}{2}},\eta^{-\frac{\beta+1}{2}}\log(\eta)\}\mathbf{1}_{D\cap A^c}\lesssim\max\{\eta^{-\beta},\eta^{-\frac{\beta+1}{2}}\log(\eta)\}\mathbf{1}_{D\cap A^c},
		\end{equation*}
		so the first term $\frac{\widetilde{F}(\eta,\eta)}{\xi-\eta}\mathbf{1}_{D\cap A^c}$ is bounded from Hilbert transfrom. For the second term with a difference quotient, we decompose further it into 
		\begin{equation*}
			\frac{\widetilde{F}(\xi,\eta)-\widetilde{F}(\eta,\eta)}{\xi-\eta}=\frac{\left<\eta\right>^s\left<\xi\right>^{-s}\sqrt{\rho(\xi)\rho(\eta)}-\rho(\eta)}{\xi-\eta}F(\xi,\eta)+\frac{F(\xi,\eta)-F(\eta,\eta)}{\xi-\eta}\rho(\eta).
		\end{equation*}
		Using the bound of $\partial_\xi F(\xi,\eta)$,
		\begin{equation*}
			\left|\frac{F(\xi,\eta)-F(\eta,\eta)}{\xi-\eta}\right|\rho(\eta)\mathbf{1}_{D\cap A^c}\lesssim \eta^{-\frac{\beta}{2}-1}\eta^{\frac{\beta}{2}}\mathbf{1}_{D\cap A^c}\lesssim \eta^{-1}\mathbf{1}_{D\cap A^c},
		\end{equation*}
		and similarly we have
		\begin{equation*}
			\begin{aligned}
				\left|\frac{\left<\eta\right>^s\left<\xi\right>^{-s}\sqrt{\rho(\xi)\rho(\eta)}-\rho(\eta)}{\xi-\eta}\right|F(\xi,\eta)\mathbf{1}_{D\cap A^c}&\lesssim \eta^{\frac{\beta}{2}-1}\max\{\eta^{-\frac{3\beta}{2}},\eta^{-\frac{\beta+1}{2}}\log(\eta)\}\mathbf{1}_{D\cap A^c}\\&\lesssim \max\{\eta^{-\beta-1},\eta^{-\frac{3}{2}}\log(\eta)\}\mathbf{1}_{D\cap A^c}.\\
			\end{aligned}
		\end{equation*}
		As a result, we have
		\begin{equation*}
			\left|\frac{\widetilde{F}(\xi,\eta)-\widetilde{F}(\eta,\eta)}{\xi-\eta}\right|\mathbf{1}_{D\cap A^c}\lesssim \eta^{-1}\mathbf{1}_{D\cap A^c},
		\end{equation*}
		which is bounded by Schur's lemma. This concludes the proof of the boundedness of $\widetilde{K}_0\mathbf{1}_{D\cap A^c}$.
		\par \underline{\textit{(1.2) Boundness of $\widetilde{K}_0\mathbf{1}_{F}$}.} For $\widetilde{K}_0\mathbf{1}_{F_1}$, we have
		\begin{equation*}
			\begin{aligned}
				|\widetilde{K}_0(\eta,\xi)\mathbf{1}_{F_1}|&\lesssim \xi^{-1}\xi^{\frac{\beta}{4}}\eta^{\frac{\beta}{4}}\xi^{-\frac{3\beta}{2}}\mathbf{1}_{F_1}\lesssim \xi^{-\frac{5\beta}{4}-1}\eta^{\frac{\beta}{4}}\mathbf{1}_{F_1},& 0<\beta<\frac{1}{5},\\
				|\widetilde{K}_0(\eta,\xi)\mathbf{1}_{F_1}|&\lesssim \xi^{-1}\xi^{\frac{1}{20}}\eta^{\frac{1}{20}}\xi^{-\frac{3}{10}}\log(\xi/\eta)\mathbf{1}_{F_1}\lesssim \xi^{-\frac{5}{4}}\eta^{\frac{1}{20}}\log(\xi)\mathbf{1}_{F_1},& \beta=\frac{1}{5},\\
				|\widetilde{K}_0(\eta,\xi)\mathbf{1}_{F_1}|&\lesssim \xi^{-1}\xi^{\frac{\beta}{4}}\eta^{\frac{\beta}{4}}\xi^{-\frac{\beta+1}{4}}\eta^{\frac{1-5\beta}{4}}\mathbf{1}_{F_1}\lesssim \xi^{-\frac{5}{4}}\eta^{\frac{1}{4}-\beta}\mathbf{1}_{F_1},& \frac{1}{5}<\beta<\frac{1}{2},\\
				|\widetilde{K}_0(\eta,\xi)\mathbf{1}_{F_1}|&\lesssim \xi^{-1}\xi^{\frac{1}{8}}\eta^{\frac{1}{8}}\xi^{-\frac{3}{8}}\eta^{-\frac{3}{8}}\log(\xi/\eta)\mathbf{1}_{F_1}\lesssim \xi^{-\frac{5}{4}}\eta^{-\frac{1}{4}}\log(\xi)\mathbf{1}_{F_1},& \beta=\frac{1}{2},\\
				|\widetilde{K}_0(\eta,\xi)\mathbf{1}_{F_1}|&\lesssim \xi^{-1}\xi^{\frac{\beta}{4}}\eta^{\frac{\beta}{4}}\xi^{-\frac{\beta+1}{4}}\eta^{-\frac{\beta+1}{4}}\mathbf{1}_{F_1}\lesssim \xi^{-\frac{5}{4}}\eta^{-\frac{1}{4}}\mathbf{1}_{F_1},& \beta>\frac{1}{2},\\
			\end{aligned}
		\end{equation*}
		and in all cases it is bounded by Schur's lemma. Similarly for $\widetilde{K}_0\mathbf{1}_{F_2}$, these estimates become
		\begin{equation*}
		\begin{aligned}
			|\widetilde{K}_0(\eta,\xi)\mathbf{1}_{F_2}|&\lesssim \eta^s\xi^{-s}\eta^{-1}\xi^{\frac{\beta}{4}}\eta^{\frac{\beta}{4}}\eta^{-\frac{3\beta}{2}}\mathbf{1}_{F_2}\lesssim \eta^{-\frac{5\beta}{4}-1+s}\xi^{\frac{\beta}{4}-s}\mathbf{1}_{F_2},& 0<\beta<\frac{1}{5},\\
			|\widetilde{K}_0(\eta,\xi)\mathbf{1}_{F_2}|&\lesssim \eta^s\xi^{-s}\eta^{-1}\xi^{\frac{1}{20}}\eta^{\frac{1}{20}}\xi^{-\frac{3}{10}}\log(\eta/\xi)\mathbf{1}_{F_2}\lesssim \eta^{-\frac{5}{4}+s}\xi^{\frac{1}{20}-s}\log(\eta)\mathbf{1}_{F_2},&\beta=\frac{1}{5},\\
			|\widetilde{K}_0(\eta,\xi)\mathbf{1}_{F_2}|&\lesssim \eta^s\xi^{-s}\eta^{-1}\xi^{\frac{\beta}{4}}\eta^{\frac{\beta}{4}}\eta^{-\frac{\beta+1}{4}}\xi^{\frac{1-5\beta}{4}}\mathbf{1}_{F_2}\lesssim \eta^{-\frac{5}{4}+s}\xi^{\frac{1}{4}-\beta-s}\mathbf{1}_{F_2},& \frac{1}{5}<\beta<\frac{1}{2},\\
			|\widetilde{K}_0(\eta,\xi)\mathbf{1}_{F_2}|&\lesssim \eta^s\xi^{-s}\eta^{-1}\xi^{\frac{1}{8}}\eta^{\frac{1}{8}}\eta^{-\frac{3}{8}}\xi^{-\frac{3}{8}}\log(\eta/\xi)\mathbf{1}_{F_2}\lesssim \eta^{-\frac{5}{4}+s}\xi^{-\frac{1}{4}-s}\log(\eta)\mathbf{1}_{F_2},&\beta=\frac{1}{2},\\
			|\widetilde{K}_0(\eta,\xi)\mathbf{1}_{F_2}|&\lesssim \eta^s\xi^{-s}\eta^{-1}\xi^{\frac{\beta}{4}}\eta^{\frac{\beta}{4}}\eta^{-\frac{\beta+1}{4}}\xi^{-\frac{\beta+1}{4}}\mathbf{1}_{F_2}\lesssim \eta^{-\frac{5}{4}+s}\xi^{-\frac{1}{4}-s}\mathbf{1}_{F_2},& \beta>\frac{1}{2},\\
		\end{aligned}
		\end{equation*}
		and also in all cases it is Hilbert-Schmidt if $0\leq s<\min\{\frac{5\beta}{4}+\frac{1}{2},\frac{3}{4}\}$.

		\par \underline{\textit{(1.3) Boundness of $\widetilde{K}_0\mathbf{1}_{B}$}.} For $\widetilde{K}_0\mathbf{1}_{B_1}$, we have
		\begin{equation*}
			\begin{aligned}
				|\widetilde{K}_0(\eta,\xi)\mathbf{1}_{B_1}|&\lesssim \xi^{-s}\xi^{-1}\xi^{\frac{\beta}{4}}\eta^{-\frac{\beta}{4}}\xi^{-\frac{3\beta}{2}}\mathbf{1}_{B_1}\lesssim \xi^{-\frac{5\beta}{4}-1-s}\eta^{-\frac{\beta}{4}}\mathbf{1}_{B_1},&0<\beta<\frac{1}{5},\\
				|\widetilde{K}_0(\eta,\xi)\mathbf{1}_{B_1}|&\lesssim \xi^{-s}\xi^{-1}\xi^{\frac{1}{20}}\eta^{-\frac{1}{20}}\xi^{-\frac{3}{10}}\log(\xi)\mathbf{1}_{B_1}\lesssim \xi^{-\frac{5}{4}-s}\eta^{-\frac{1}{20}}\log(\xi)\mathbf{1}_{B_1},&\beta=\frac{1}{5},\\
				|\widetilde{K}_0(\eta,\xi)\mathbf{1}_{B_1}|&\lesssim \xi^{-s}\xi^{-1}\xi^{\frac{\beta}{4}}\eta^{-\frac{\beta}{4}}\xi^{-\frac{\beta+1}{4}}\mathbf{1}_{B_1}\lesssim \xi^{-\frac{5}{4}-s}\eta^{-\frac{\beta}{4}}\mathbf{1}_{B_1},& \frac{1}{5}<\beta<2,\\
				|\widetilde{K}_0(\eta,\xi)\mathbf{1}_{B_1}|&\lesssim \xi^{-s}\xi^{-1}\xi^{\frac{\beta}{4}}\eta^{-\frac{\beta}{4}}|\log(\eta)|^{-1}\xi^{-\frac{\beta+1}{4}}\mathbf{1}_{B_1}\lesssim \xi^{-\frac{5}{4}-s}\eta^{-\frac{1}{2}}|\log(\eta)|^{-1}\mathbf{1}_{B_1},&\beta=2,\\
				|\widetilde{K}_0(\eta,\xi)\mathbf{1}_{B_1}|&\lesssim \xi^{-s}\xi^{-1}\xi^{\frac{\beta}{4}}\eta^{\frac{\beta}{4}-1}\xi^{-\frac{\beta+1}{4}}\mathbf{1}_{B_1}\lesssim \xi^{-\frac{5}{4}-s}\eta^{\frac{\beta}{4}-1}\mathbf{1}_{B_1},& 2<\beta<4,\\
			\end{aligned}
		\end{equation*}
		which in all cases turns out to be Hilbert-Schmidt. Similarly, the estimates for $\widetilde{K}_0\mathbf{1}_{B_2}$ are
		\begin{equation*}
			\begin{aligned}
				|\widetilde{K}_0(\eta,\xi)\mathbf{1}_{B_2}|&\lesssim \eta^{s}\eta^{-1}\eta^{\frac{\beta}{4}}\xi^{-\frac{\beta}{4}}\eta^{-\frac{3\beta}{2}}\mathbf{1}_{B_2}\lesssim \eta^{-\frac{5\beta}{4}-1+s}\xi^{-\frac{\beta}{4}}\mathbf{1}_{B_2},& 0<\beta<\frac{1}{5},\\
				|\widetilde{K}_0(\eta,\xi)\mathbf{1}_{B_2}|&\lesssim \eta^{s}\eta^{-1}\eta^{\frac{1}{20}}\xi^{-\frac{1}{20}}\eta^{-\frac{3}{10}}\log(\eta)\mathbf{1}_{B_2}\lesssim \eta^{-\frac{5}{4}+s}\xi^{-\frac{1}{20}}\log(\eta)\mathbf{1}_{B_2},& \beta=\frac{1}{5},\\
				|\widetilde{K}_0(\eta,\xi)\mathbf{1}_{B_2}|&\lesssim \eta^{s}\eta^{-1}\eta^{\frac{\beta}{4}}\xi^{-\frac{\beta}{4}}\eta^{-\frac{\beta+1}{4}}\mathbf{1}_{B_2}\lesssim \eta^{-\frac{5}{4}+s}\xi^{-\frac{\beta}{4}}\mathbf{1}_{B_2},& \frac{1}{5}<\beta<2,\\
				|\widetilde{K}_0(\eta,\xi)\mathbf{1}_{B_2}|&\lesssim \eta^{s}\eta^{-1}\eta^{\frac{1}{2}}\xi^{-\frac{1}{2}}|\log(\xi)|^{-1}\eta^{-\frac{\beta+1}{4}}\mathbf{1}_{B_2}\lesssim \eta^{-\frac{5}{4}+s}\xi^{-\frac{1}{2}}|\log(\xi)|^{-1}\mathbf{1}_{B_2},& \beta=2,\\
				|\widetilde{K}_0(\eta,\xi)\mathbf{1}_{B_2}|&\lesssim \eta^{s}\eta^{-1}\eta^{\frac{1}{2}}\xi^{\frac{\beta}{4}-1}\eta^{-\frac{\beta+1}{4}}\mathbf{1}_{B_2}\lesssim \eta^{-\frac{5}{4}+s}\xi^{\frac{\beta}{4}-1}\mathbf{1}_{B_2},& 2<\beta<4,\\
			\end{aligned}
		\end{equation*}
		which in all cases turns out also to be Hilbert-Schmidt if $0\leq s<\min\{\frac{5\beta}{4}+\frac{1}{2},\frac{3}{4}\}$.

		\par \underline{\textit{(1.4) Boundness of $\widetilde{K}_0\mathbf{1}_{D\cap A}$}.} We have
		\begin{equation*}
			\widetilde{K}_0(\xi,\eta)\mathbf{1}_{D\cap A}=\frac{\widetilde{F}(\eta,\eta)}{\xi-\eta}\mathbf{1}_{D\cap A}+\frac{\widetilde{F}(\xi,\eta)-\widetilde{F}(\eta,\eta)}{\xi-\eta}\mathbf{1}_{D\cap A},
		\end{equation*}
		where 
		\begin{equation*}
			\begin{aligned}
				|\widetilde{F}(\eta,\eta)\mathbf{1}_{D\cap A}|&\lesssim \eta^{-\frac{\beta}{2}}\eta^{\frac{3\beta}{2}}\mathbf{1}_{D\cap A}\lesssim\eta^{\beta}\mathbf{1}_{D\cap A},& 0<\beta<\frac{2}{3},\\
				|\widetilde{F}(\eta,\eta)\mathbf{1}_{D\cap A}|&\lesssim \eta^{-\frac{1}{3}}\eta|\log(\eta/2)|\mathbf{1}_{D\cap A}\lesssim \eta^{\frac{2}{3}}|\log(\eta/2)|\mathbf{1}_{D\cap A},& \beta=\frac{2}{3},\\
				|\widetilde{F}(\eta,\eta)\mathbf{1}_{D\cap A}|&\lesssim \eta^{-\frac{\beta}{2}}\eta\mathbf{1}_{D\cap A}\lesssim\eta^{1-\frac{\beta}{2}}\mathbf{1}_{D\cap A},& \frac{2}{3}<\beta<2,\\
				|\widetilde{F}(\eta,\eta)\mathbf{1}_{D\cap A}|&\lesssim \eta^{-1}|\log(\eta)|^{-2}\eta\mathbf{1}_{D\cap A}\lesssim|\log(\eta)|^{-2}\mathbf{1}_{D\cap A},&\beta=2,\\
				|\widetilde{F}(\eta,\eta)\mathbf{1}_{D\cap A}|&\lesssim \eta^{\frac{\beta}{2}-2}\eta\mathbf{1}_{D\cap A}\lesssim\eta^{\frac{\beta}{2}-1}\mathbf{1}_{D\cap A},& \beta>2,\\
			\end{aligned}
		\end{equation*}
		so the first term $\frac{\widetilde{F}(\eta,\eta)}{\xi-\eta}\mathbf{1}_{D\cap A}$ is bounded from Hilbert transform. For the second term $\frac{\widetilde{F}(\xi,\eta)-\widetilde{F}(\eta,\eta)}{\xi-\eta}\mathbf{1}_{D\cap A}$, we splits it further into two parts
		\begin{equation*}
			\frac{\widetilde{F}(\xi,\eta)-\widetilde{F}(\eta,\eta)}{\xi-\eta}\mathbf{1}_{D\cap A}=\frac{\left<\eta\right>^s\left<\xi\right>^{-s}\sqrt{\rho(\xi)\rho(\eta)}-\rho(\eta)}{\xi-\eta}F(\xi,\eta)\mathbf{1}_{D\cap A}+\rho(\eta)\frac{F(\xi,\eta)-F(\eta,\eta)}{\xi-\eta}\mathbf{1}_{D\cap A},
		\end{equation*} 
		with the first term behaving like
		\begin{equation*}
			\begin{aligned}
				\left|\frac{\left<\eta\right>^s\left<\xi\right>^{-s}\sqrt{\rho(\xi)\rho(\eta)}-\rho(\eta)}{\xi-\eta}F(\xi,\eta)\right|\mathbf{1}_{D\cap A}&\lesssim\eta^{-\frac{\beta}{2}-1}\eta^{\frac{3\beta}{2}}\mathbf{1}_{D\cap A}\lesssim\eta^{\beta-1}\mathbf{1}_{D\cap A},& 0<\beta<2,\\
				\left|\frac{\left<\eta\right>^s\left<\xi\right>^{-s}\sqrt{\rho(\xi)\rho(\eta)}-\rho(\eta)}{\xi-\eta}F(\xi,\eta)\right|\mathbf{1}_{D\cap A}&\lesssim\eta^{-2}|\log(\eta)|^{-2}\eta|\log(\eta/2)|\mathbf{1}_{D\cap A}\\&\lesssim \eta^{-1}|\log(\eta)|^{-1}\mathbf{1}_{D\cap A},& \beta=2,\\
				\left|\frac{\left<\eta\right>^s\left<\xi\right>^{-s}\sqrt{\rho(\xi)\rho(\eta)}-\rho(\eta)}{\xi-\eta}F(\xi,\eta)\right|\mathbf{1}_{D\cap A}&\lesssim\eta^{\frac{\beta}{2}-3}\eta\mathbf{1}_{D\cap A}\lesssim \eta^{\frac{\beta}{2}-2}\mathbf{1}_{D\cap A},& \beta>2.\\
			\end{aligned}
		\end{equation*}
		while the second term behaves like
		\begin{equation*}
			\begin{aligned}
				\rho(\eta)\left|\frac{F(\xi,\eta)-F(\eta,\eta)}{\xi-\eta}\right|\mathbf{1}_{D\cap A}&\lesssim \eta^{-\frac{\beta}{2}}\eta^{\frac{3\beta}{2}-1}\mathbf{1}_{D\cap A}\lesssim\eta^{\beta-1}\mathbf{1}_{D\cap A},& 0<\beta<\frac{2}{3},\\
				\rho(\eta)\left|\frac{F(\xi,\eta)-F(\eta,\eta)}{\xi-\eta}\right|\mathbf{1}_{D\cap A}&\lesssim \eta^{-\frac{1}{3}}|\log(\eta/2)|\mathbf{1}_{D\cap A}\lesssim \eta^{-\frac{1}{3}}|\log(\eta/2)|\mathbf{1}_{D\cap A},& \beta=\frac{2}{3},\\
				\rho(\eta)\left|\frac{F(\xi,\eta)-F(\eta,\eta)}{\xi-\eta}\right|\mathbf{1}_{D\cap A}&\lesssim \eta^{-\frac{\beta}{2}}\mathbf{1}_{D\cap A},& \frac{2}{3}<\beta<2,\\
				\rho(\eta)\left|\frac{F(\xi,\eta)-F(\eta,\eta)}{\xi-\eta}\right|\mathbf{1}_{D\cap A}&\lesssim \eta^{-1}(\log(\eta))^{-2}\mathbf{1}_{D\cap A},& \beta=2,\\
				\rho(\eta)\left|\frac{F(\xi,\eta)-F(\eta,\eta)}{\xi-\eta}\right|\mathbf{1}_{D\cap A}&\lesssim \eta^{\frac{\beta}{2}-2}\mathbf{1}_{D\cap A},& \beta>2,\\
			\end{aligned}
		\end{equation*}
		which in all cases shows that $\frac{\widetilde{F}(\xi,\eta)-\widetilde{F}(\eta,\eta)}{\xi-\eta}\mathbf{1}_{D\cap A}$ is bounded by Schur's lemma.
		\par \underline{\textit{(1.5) Boundness of $\widetilde{K}_0\mathbf{1}_{D^c\cap A}$}.} It's enough to prove the boundedness for $\widetilde{K}_0\mathbf{1}_{E}$, for which we have the following estimate
		\begin{equation*}
			\begin{aligned}
				|\widetilde{K}_0(\eta,\xi)\mathbf{1}_{E}|&\lesssim \xi^{-1}\xi^{-\frac{\beta}{4}}\eta^{-\frac{\beta}{4}}\xi^{\frac{\beta+1}{4}}\eta^{\frac{5\beta-1}{4}}\mathbf{1}_{E}\lesssim \xi^{-\frac{3}{4}}\eta^{\beta-\frac{1}{4}}\mathbf{1}_{E},& 0<\beta<\frac{1}{5},\\
				|\widetilde{K}_0(\eta,\xi)\mathbf{1}_{E}|&\lesssim \xi^{-1}\xi^{-\frac{1}{20}}\xi^{\frac{3}{10}}\eta^{-\frac{1}{20}}\log(\xi/\eta)\mathbf{1}_{E}\lesssim \xi^{-\frac{3}{4}}\eta^{-\frac{1}{20}}\log(\xi/\eta)\mathbf{1}_{E},& \beta=\frac{1}{5},\\
				|\widetilde{K}_0(\eta,\xi)\mathbf{1}_{E}|&\lesssim \xi^{-1}\xi^{-\frac{\beta}{4}}\eta^{-\frac{\beta}{4}}\xi^{\frac{3\beta}{2}}\mathbf{1}_{E}\lesssim \xi^{\frac{5\beta}{4}-1}\eta^{-\frac{\beta}{4}}\mathbf{1}_{E},& \frac{1}{5}<\beta<\frac{2}{3},\\
				|\widetilde{K}_0(\eta,\xi)\mathbf{1}_{E}|&\lesssim \xi^{-1}\xi^{-\frac{1}{6}}\eta^{-\frac{1}{6}}\xi\log(\xi/\eta)\mathbf{1}_{E}\lesssim \xi^{-\frac{1}{6}}\eta^{-\frac{1}{6}}\log(\xi/\eta)\mathbf{1}_{E},& \beta=\frac{2}{3},\\
				|\widetilde{K}_0(\eta,\xi)\mathbf{1}_{E}|&\lesssim \xi^{-1}\xi^{-\frac{\beta}{4}}\eta^{-\frac{\beta}{4}}\xi\mathbf{1}_{E}\lesssim \xi^{-\frac{\beta}{4}}\eta^{-\frac{\beta}{4}}\mathbf{1}_{E},& \frac{2}{3}<\beta<2,\\
				|\widetilde{K}_0(\eta,\xi)\mathbf{1}_{E}|&\lesssim \xi^{-1}\xi^{-\frac{1}{2}}|\log(\xi)|^{-1}\eta^{-\frac{1}{2}}|\log(\eta)|^{-1}\xi\mathbf{1}_{E}\\&\lesssim \xi^{-\frac{1}{2}}|\log(\xi)|^{-1}\eta^{-\frac{1}{2}}|\log(\eta)|^{-1}\mathbf{1}_{E},&\beta=2,\\
				|\widetilde{K}_0(\eta,\xi)\mathbf{1}_{E}|&\lesssim \xi^{-1}\xi^{\frac{\beta}{4}-1}\eta^{\frac{\beta}{4}-1}\xi\mathbf{1}_{E}\lesssim \xi^{\frac{\beta}{4}-1}\eta^{\frac{\beta}{4}-1}\mathbf{1}_{E},& 2<\beta<4,\\
			\end{aligned}
		\end{equation*}
		which proves the boundedness of $\widetilde{K}_0\mathbf{1}_{E}$ since it's a Hilbert-Schmidt kernel. Untill now, we have established the boundedness of $\mathcal{K}_{cc}$ as an operator on $L^{2,\alpha}_\rho$.

		\par \underline{\textbf{(2) Boundedness of $[\mathcal{K}_{cc},\xi\partial_\xi]\colon L^{2,\alpha}_{\rho}((0,+\infty))\to L^{2,\alpha}_{\rho}((0,+\infty))$}.} A direct computation shows that the Schwartz kernel of this operator is of the form $\mathrm{p.v.}\frac{\rho(\xi)}{\xi-\eta}F^{\mathrm{com}}(\xi,\eta)$ modulo some delta distributions sitting on the diagonal $\{\xi=\eta\}$, where
		\begin{equation*}
			F^{\mathrm{com}}(\xi,\eta)=\frac{\xi\rho'(\xi)}{\rho(\xi)}F(\xi,\eta)+\xi\partial_\xi F(\xi,\eta)+\eta \partial_\eta F(\xi,\eta).
		\end{equation*}
		To prove the boundedness, it's equivalent to prove that the integral kernel (in the principal value sense)
		\begin{equation*}
			\widetilde{K}_{c}^{\mathrm{com}}(\eta,\xi):=\left<\eta\right>^s\left<\xi\right>^{-s}\frac{\sqrt{\rho(\xi)\rho(\eta)}}{\xi-\eta}F^{\mathrm{com}}(\xi,\eta)=:\frac{\widetilde{F}^{\mathrm{com}}(\xi,\eta)}{\xi-\eta}
		\end{equation*}
		induces a bounded operator on $L^2((0,+\infty))$. Moreover, since $\rho$ has symbol-like behavior near $0$ and $\infty$, we may omit the term $\frac{\xi\rho'(\xi)}{\rho(\xi)}F(\xi,\eta)$ in $F^{\mathrm{com}}(\xi,\eta)$.

		\par \underline{\textit{(2.1) Boundedness of $\widetilde{K}_c^{\mathrm{com}}\mathbf{1}_{D\cap A^c}$}.} 
		
		As in \textit{part (2.1)}, after extracting the Hilbert transfrom type operator $\frac{\widetilde{F}^{\mathrm{com}}(\eta,\eta)}{\xi-\eta}$ with
		\begin{equation*}
			|\widetilde{F}^{\mathrm{com}}(\eta,\eta)|\lesssim \eta^{\frac{\beta}{2}}\eta\eta^{-\frac{\beta}{2}-1}\lesssim 1,
		\end{equation*}
		we only have to deal with the difference qoutient $\frac{\widetilde{F}^{\mathrm{com}}(\xi,\eta)-\widetilde{F}^{\mathrm{com}}(\eta,\eta)}{\xi-\eta}\mathbf{1}_{D\cap A^c}$. To use derivative bound, it suffices to consider the terms where $\partial_\xi$ falls on $\partial_\xi F(\xi,\eta)$ and $\partial_\eta F(\xi,\eta)$, since other terms are better. We have, for example,
		\begin{equation*}
			\rho(\xi)\xi\partial_\xi^2F(\xi,\eta)\mathbf{1}_{D\cap A^c}\lesssim \frac{\xi^{-\frac{1}{2}}}{1+|\xi^{\frac{1}{2}}-\eta^{\frac{1}{2}}|^2}\mathbf{1}_{D\cap A^c}.
		\end{equation*}
		Since
		\begin{equation*}
			\begin{aligned}
			\int_{0}^\infty \frac{\xi^{-\frac{1}{2}}}{1+|\xi^{\frac{1}{2}}-\eta^{\frac{1}{2}}|^2}\mathbf{1}_{D\cap A^c}d\eta&\lesssim \int_{\xi/4}^{(\xi^{\frac{1}{2}}-1)^2}\frac{\xi^{-\frac{1}{2}}}{(\xi^{\frac{1}{2}}-\eta^{\frac{1}{2}})^2}d\eta +\int_{(\xi^{\frac{1}{2}}-1)^2}^{\xi}\xi^{-\frac{1}{2}}d\eta\\
			&\lesssim \int_{1/4}^{(1-\xi^{-\frac{1}{2}})^2}\frac{\xi^{-\frac{1}{2}}}{(1-u^{\frac{1}{2}})^2}du+1\\
			&\lesssim\xi^{-\frac{1}{2}}(2\xi^{-\frac{1}{2}}-\xi^{-1})^{-1}+1\lesssim 1.\\
			\end{aligned}
		\end{equation*}
		we get the desired boundedness by Schur's lemma.

		\par \underline{\textit{(2.2) Boundedness of $\widetilde{K}_c^{\mathrm{com}}\mathbf{1}_{D\cap A}$}.} 
		After exacting $\frac{\widetilde{F}^{\mathrm{com}}(\eta,\eta)}{\xi-\eta}$ with $\widetilde{F}^{\mathrm{com}}(\eta,\eta)$ having the same bounds as $\widetilde{F}(\eta,\eta)$, it suffices to check
		\begin{equation*}
			\begin{aligned}
				\rho(\xi)\xi\partial_\xi^2F(\xi,\eta)\mathbf{1}_{D\cap A}&\lesssim \xi^{-\frac{\beta}{2}}\xi\xi^{\frac{3\beta}{2}-2}\lesssim \xi^{\beta-1}, &\frac{1}{2}<\beta<\frac{4}{3},\\
				\rho(\xi)\xi\partial_\xi^2F(\xi,\eta)\mathbf{1}_{D\cap A}&\lesssim \xi^{-\frac{2}{3}}\xi|\log(\xi)|\lesssim \xi^{\frac{1}{3}}|\log(\xi)|, &\beta=\frac{4}{3},\\
				\rho(\xi)\xi\partial_\xi^2F(\xi,\eta)\mathbf{1}_{D\cap A}&\lesssim \xi^{-\frac{\beta}{2}}\xi\lesssim\xi^{1-\frac{\beta}{2}} , &\frac{4}{3}<\beta<2,\\
				\rho(\xi)\xi\partial_\xi^2F(\xi,\eta)\mathbf{1}_{D\cap A}&\lesssim \xi^{-1}|\log(\xi)|^{-2}\xi\lesssim |\log(\xi)|^{-2}, &\beta=2,\\
				\rho(\xi)\xi\partial_\xi^2F(\xi,\eta)\mathbf{1}_{D\cap A}&\lesssim \xi^{\frac{\beta}{2}-2}\xi\lesssim \xi^{\frac{\beta}{2}-1}, &\beta>2.\\
			\end{aligned}
		\end{equation*}
		which in all cased turns out to be bounded by Schur's lemma.

		\par \underline{\textit{(2.3) Boundedness of $\widetilde{K}_c^{\mathrm{com}}\mathbf{1}_{D^c}$}.} Outside $D$, $\xi\partial_\xi F(\xi,\eta)$ and $\eta\partial_\eta F(\xi,\eta)$ have the same bounds as $F(\xi,\eta)$, hence the boundedness follows.

		\par This concludes the proof of the boundedness of $[\mathcal{K}_{cc},\xi\partial_\xi]$.

		\par \underline{\textbf{(3) Boundedness of $\mathcal{K}_{cd}$, $\mathcal{K}_{c0}$, $\mathcal{K}_{dc}$ and $\mathcal{K}_{0c}$.}} 
		To prove the boundedness of $\mathcal{K}_{dc}$ and $\mathcal{K}_{0c}$, we have to show that $\widetilde{K_{dc}}(\xi)$ and $\widetilde{K_{0c}}(\xi)$ belong to $L^2((0,+\infty))$, where
		\begin{equation*}
			\widetilde{K_{dc}}(\xi):=-\frac{2}{3}\left<\xi\right>^{-s}\frac{\sqrt{\rho(\xi)}}{\xi-\xi_d}\left<U(R)\phi(R,\xi),\phi_d(R)\right>_{L^2_R}=-\frac{2}{3}\left<\xi\right>^{-s}\frac{\sqrt{\rho(\xi)}}{\xi-\xi_d}F(\xi,\xi_d),
		\end{equation*}
		and
		\begin{equation*}
			\widetilde{K_{0c}}(\xi):=-\frac{2}{3}\left<\xi\right>^{-s}\frac{\sqrt{\rho(\xi)}}{\xi}\left<U(R)\phi(R,\xi),\phi_0(R))\right>_{L^2_R}=-\frac{2}{3}\left<\xi\right>^{-s}\frac{\sqrt{\rho(\xi)}}{\xi}F(\xi,0).
		\end{equation*}
		These hold for all $s\geq 0$.
		
		To prove the boundedness of $\mathcal{K}_{cd}$ and $\mathcal{K}_{c0}$, we have to show that $\widetilde{\mathcal{K}_{cd}}$ and $\widetilde{\mathcal{K}_{c0}}$ belong to $L^2((0,+\infty))$, where
		\begin{equation*}
			\widetilde{\mathcal{K}_{cd}}(\xi):=\left<\xi\right>^s\frac{\sqrt{\rho(\xi)}}{\xi-\xi_d}\left<U(R)\phi(R,\xi),\phi_d(R)\right>_{L^2_R}=\left<\xi\right>^s\frac{\sqrt{\rho(\xi)}}{\xi-\xi_d}F(\xi,\xi_d),
		\end{equation*}
		and
		\begin{equation*}
			\widetilde{\mathcal{K}_{c0}}(\xi):=\left<\xi\right>^s\frac{\sqrt{\rho(\xi)}}{\xi}\left<U(R)\phi(R,\xi),\phi_0(R)\right>_{L^2_R}=\left<\xi\right>^s\frac{\sqrt{\rho(\xi)}}{\xi}F(\xi,0).
		\end{equation*}
		These hold for $0\leq s<\min\{\frac{5\beta}{4}+\frac{1}{2},\frac{3}{4}\}$.
		
		We leave all the details to the reader.

		\par \underline{\textbf{(4) Boundedness of $\xi\partial_\xi\mathcal{K}_{cd}$, $\xi\partial_\xi\mathcal{K}_{c0}$, $\mathcal{K}_{dc}\xi\partial_\xi$ and $\mathcal{K}_{0c}\xi\partial_\xi$.}} As in \textbf{part (3)}, the corresponding functions are essentially obtained by replacing $F(\xi,\xi_d)$ and $F(\xi,0)$ by $\xi\partial_\xi F(\xi,\xi_d)$ and $\xi\partial_\xi F(\xi,0)$ respectively. Since $\xi\partial_\xi F$ and $F$ have the same bound in this region, the desired result follows from \textbf{part (3)}.	Again, we leave all the details to the reader.	
	\end{proof}

	\section{The final equation}\label{sec:final}

	To solve the equation (\ref{eq:perturb_symp}) in Fourier side, we need to calculate the commutator of $\mathcal{D}$ and $\mathcal{F}$. Recall that from the last two sections, we have
	\begin{equation*}
		\mathcal{F}\mathcal{L}=\xi\mathcal{F},\quad \mathcal{F}R\partial_R = -2\xi\partial_\xi \mathcal{F}+\mathcal{K}\mathcal{F}.
	\end{equation*}
	Hence,
	\begin{equation*}
		\mathcal{F}\mathcal{D}=\left(\partial_\tau-2\frac{\lambda_{\tau}}{\lambda}\xi\partial_\xi-\frac{\lambda_\tau}{\lambda}\right)\mathcal{F}+\frac{\lambda_{\tau}}{\lambda}\mathcal{K}\mathcal{F}=:\widehat{\mathcal{D}}\mathcal{F}+\frac{\lambda_{\tau}}{\lambda}\mathcal{K}\mathcal{F},
	\end{equation*}
	and
	\begin{equation*}
		\begin{aligned}
			\mathcal{F}\mathcal{D}^2&=\left(\widehat{\mathcal{D}}+\frac{\lambda_{\tau}}{\lambda}\mathcal{K}\right)^2\mathcal{F}=\widehat{\mathcal{D}}^2\mathcal{F}+\widehat{\mathcal{D}}\frac{\lambda_{\tau}}{\lambda}\mathcal{K}\mathcal{F}+\frac{\lambda_{\tau}}{\lambda}\mathcal{K}\widehat{\mathcal{D}}\mathcal{F}+(\frac{\lambda_{\tau}}{\lambda})^2\mathcal{K}^2\mathcal{F}\\
			&=\widehat{\mathcal{D}}^2\mathcal{F}+2\frac{\lambda_{\tau}}{\lambda}\mathcal{K}\widehat{\mathcal{D}}\mathcal{F}+\frac{\lambda_{\tau}}{\lambda}[\widehat{\mathcal{D}},\mathcal{K}]\mathcal{F}+(\frac{\lambda_{\tau}}{\lambda})^2(\mathcal{K}^2 +c_\nu\mathcal{K})\mathcal{F}\\
			&=\widehat{\mathcal{D}}^2\mathcal{F}+2\frac{\lambda_{\tau}}{\lambda}\mathcal{K}\widehat{\mathcal{D}}\mathcal{F}+2(\frac{\lambda_{\tau}}{\lambda})^2[\mathcal{K},\xi\partial_\xi]\mathcal{F}+(\frac{\lambda_{\tau}}{\lambda})^2(\mathcal{K}^2+c_\nu\mathcal{K})\mathcal{F},\\
		\end{aligned}
	\end{equation*}
	where $c_\nu=-\frac{2+\nu^{-1}}{1+\nu^{-1}}$.
	\par Denote $(\mathcal{F}\widetilde{\varepsilon})(\tau,\cdot)=\underline{\mathbf{x}}(\tau)=(x_d(\tau),x_0(\tau),x(\tau,\xi)|_{\xi>0})$, then after applying $\mathcal{F}$ on both sides of (\ref{eq:perturb_symp}), we get the final equation as
	\begin{equation*}
		\begin{aligned}
			\left(\widehat{\mathcal{D}}^2+\frac{\lambda_{\tau}}{\lambda}\widehat{\mathcal{D}}+\xi\right)\underline{\mathbf{x}}=&-(\frac{\lambda_{\tau}}{\lambda})^2\left(\mathcal{K}^2+\mathcal{K}+c_\nu\mathcal{K}+2[\mathcal{K},\xi\partial_\xi]\right)\underline{\mathbf{x}}-2\frac{\lambda_{\tau}}{\lambda}\mathcal{K}\widehat{\mathcal{D}}\underline{\mathbf{x}}\\
			&+\lambda^{-2}\mathcal{F}R\left(N_{2k-1}(R^{-1}\mathcal{F}^{-1}\underline{\mathbf{x}})+e_{2k-1}\right).
		\end{aligned}
	\end{equation*}
	We emphasis again that $\partial_\xi$ derivative only acts on the continuous part of $\underline{\mathbf{x}}$, i.e. $\underline{\mathbf{x}}|_{\mathbb{R}_{>0}}$. 
	
	\begin{remark}[on extending $u_{2k-1}$ and $e_{2k-1}$]
		Note that $u_{2k-1}(t,r)$ and $e_{2k-1}(t,r)$ from Section \ref{sec:appro} are only defined inside the light-cone. We have to extend them to $\mathbb{R}^3$ for all small $t$ to apply the distorted Hankel transform. This can be done while preserving the regularity and smallness.
	\end{remark}

	From now on, we work on solving the equation
	\begin{equation}\label{eq:final}
		(\widehat{\mathcal{D}}^2+\frac{\lambda_{\tau}}{\lambda}\widehat{\mathcal{D}}+\xi)\underline{\mathbf{x}}=\underline{\mathbf{y}}.
	\end{equation}
	The function space we are going to work on is defined as 
	\begin{definition}
		For any $N\geq0$ and $s\geq0$, we define the $\tau$-weighted space $L^{\infty,N}L^{2,s}_\rho$ to be the space of functions with finite $L^{\infty,N}L^{2,s}_\rho$ norm:
		\begin{equation*}
			\|\underline{\mathbf{x}}\|_{L^{\infty,N}L^{2,s}_\rho}:=\sup_{\tau\geq \tau_0}\tau^N\|\underline{\mathbf{x}}\|_{L^{2,s}_\rho}.
		\end{equation*}
	\end{definition}

	\par For the continuous part, by imposing the zero data at $\tau=\infty$,  we can solve (\ref{eq:final}) explicitly as follows. Let $Sf(\tau,\xi)=f(\tau,\lambda^{-2}(\tau)\xi)$ and $Mf(\tau,\xi)=\lambda^{-1}(\tau)f(\tau,\xi)$. Then, we have
	\begin{equation*}
		M^{-1}\left(\partial_\tau-2\frac{\lambda_{\tau}}{\lambda}\xi\partial_\xi\right)M\underline{\mathbf{x}}=\widehat{\mathcal{D}}\underline{\mathbf{x}},
	\end{equation*}
	and
	\begin{equation*}
		S^{-1}\partial_\tau S\underline{\mathbf{x}}=\left(\partial_\tau-2\frac{\lambda_{\tau}}{\lambda}\xi\partial_\xi\right)\underline{\mathbf{x}}.
	\end{equation*}
	Hence, we get
	\begin{equation*}
		\widehat{\mathcal{D}}\underline{\mathbf{x}}=(SM)^{-1}\partial_\tau (SM) \underline{\mathbf{x}},
	\end{equation*}
	and
	\begin{equation*}
		\left(\widehat{\mathcal{D}}^2+\frac{\lambda_{\tau}}{\lambda}\widehat{\mathcal{D}}+\xi\right)\underline{\mathbf{x}}=(SM)^{-1}\left(\partial_\tau^2+\frac{\lambda_{\tau}}{\lambda}\partial_\tau+\lambda^{-2}(\tau)\xi\right)(SM)\underline{\mathbf{x}}.
	\end{equation*}
	One may check that $\sin(\xi^{\frac{1}{2}}\omega(\tau))$ and $\cos(\xi^{\frac{1}{2}}\omega(\tau))$ form a fundamental system of solutions of $\partial_\tau^2+\frac{\lambda_{\tau}}{\lambda}\partial_\tau+\lambda^{-2}(\tau)\xi=0$, where $\omega(\tau)=\int_{1}^{\tau}\lambda^{-1}(\sigma)d\sigma=\nu^{-\nu^{-1}}(1-\tau^{-\nu^{-1}})$. As a result, we obatin an explicit formula:
	\begin{equation}\label{eq:x}
		\begin{aligned}
			x(\tau,\xi)&=\int_{\tau}^{\infty}\xi^{-\frac{1}{2}}\sin\left(\lambda(\tau)\xi^{\frac{1}{2}}(\omega(\sigma)-\omega(\tau))\right)y(\sigma,\frac{\lambda^{2}(\tau)}{\lambda^2(\sigma)}\xi)d\sigma\\
			&=:\int_{\tau}^{\infty}H_c(\tau,\sigma;\xi)y(\sigma,\xi)d\sigma.\\
		\end{aligned}
	\end{equation}
	Here, for each fiexed $\tau$ and $\sigma$, we interpret $H_c(\tau,\sigma;\xi)$ as an operator in $\xi$.
	\begin{proposition}
		Let $\sigma\geq\tau$, we have the following estimates for the kernel $H_c(\tau,\sigma;\cdot)$:
		\begin{equation*}
			\|H_c(\tau,\sigma;\cdot)\|_{L^{2,s}_\rho\to L^{2,s+\frac{1}{2}}_\rho}\lesssim \tau(\frac{\sigma}{\tau})^C,
		\end{equation*}
		and
		\begin{equation*}
		\|\widehat{\mathcal{D}}H_c(\tau,\sigma,\cdot)\|_{L^{2,s}_\rho\to L^{2,s}_{\rho}}\lesssim(\frac{\sigma}{\tau})^C.
		\end{equation*}
		where $C=C(s,\beta,\nu)$ is a (large) constant.
		\par Moreover, let $\underline{\mathbf{x}}$ and $\underline{\mathbf{y}}$ be as in (\ref{eq:final}), then we have for $N$ large enough that
		\begin{equation*}
			\|x\|_{L^{\infty,N-2}L^{2,s+\frac{1}{2}}_\rho}\leq\frac{C'}{N} \|y\|_{L^{\infty,N}L^{2,s}_\rho},
		\end{equation*}
		and
		\begin{equation*}
			\|\widehat{\mathcal{D}}x\|_{L^{\infty,N-1}L^{2,s}_\rho}\leq \frac{C'}{N}\|y\|_{L^{\infty,N}L^{2,s}_{\rho}}.
		\end{equation*}
		where $C'=C'(s,\beta,\nu)$ is a (large) constant independent of $N$.
	\end{proposition}
	
	\begin{proof}
		For any $y\in L^{2,s+\frac{1}{2}}_{\rho}((0,+\infty))$, we have
		\begin{equation*}
			\begin{aligned}
				&\quad \|H_c(\tau,\sigma;\cdot)y(\sigma,\cdot)\|_{L^{2,s+\frac{1}{2}}_\rho}^2\\
				&=\int_{0}^{\infty}\left<\xi\right>^{2s+1}\left|\xi^{-\frac{1}{2}}\sin\left(\lambda(\tau)\xi^{\frac{1}{2}}(\omega(\sigma)-\omega(\tau))\right)\right|^2\left|y(\sigma,\frac{\lambda^2(\tau)}{\lambda^2(\sigma)}\xi)\right|^2\rho(\xi)d\xi\\
				&=\int_{0}^{\infty}\left<\frac{\lambda^2(\sigma)}{\lambda^2(\tau)}\xi\right>^{2s+1}\left<\xi\right>^{-2s}\frac{\rho(\frac{\lambda^2(\sigma)}{\lambda^2(\tau)}\xi)}{\rho(\xi)}\frac{\lambda^2(\sigma)}{\lambda^2(\tau)}\\&\quad \times\left|\frac{\lambda(\tau)}{\lambda(\sigma)}\xi^{-\frac{1}{2}}\sin\left(\lambda(\sigma)\xi^{\frac{1}{2}}(\omega(\sigma)-\omega(\tau))\right)\right|^2\left<\xi\right>^{2s}|y(\sigma,\xi)|^2\rho(\xi)d\xi\\
			\end{aligned}
		\end{equation*}
		Together with the facts that for $\sigma\geq\tau$,
		\begin{equation*}
			\left|\frac{\lambda(\tau)}{\lambda(\sigma)}\xi^{-\frac{1}{2}}\sin\left(\lambda(\sigma)\xi^{\frac{1}{2}}(\omega(\sigma)-\omega(\tau))\right)\right|\leq\min\left\{\lambda(\tau)\int_{\tau}^{\infty}\lambda^{-1}(u)du,\frac{\lambda(\tau)}{\lambda(\sigma)}\xi^{-\frac{1}{2}}\right\}\lesssim\tau\left<\xi\right>^{-\frac{1}{2}},
		\end{equation*}
		\begin{equation*}
			\left<\frac{\lambda^2(\sigma)}{\lambda^2(\tau)}\xi\right>^{2s+1}\left<\xi\right>^{-2s}\lesssim (\frac{\lambda^2(\sigma)}{\lambda^2(\tau)})^{2s+1}\left<\xi\right>,
		\end{equation*}
		and
		\begin{equation*}
			\frac{\rho(\frac{\lambda^2(\sigma)}{\lambda^2(\tau)}\xi)}{\rho(\xi)}\lesssim (\frac{\lambda^2(\sigma)}{\lambda^2(\tau)})^{\frac{\beta}{2}},
		\end{equation*}
		we obtain
		\begin{equation*}
			\|H_c(\tau,\sigma;\cdot)y(\sigma,\cdot)\|_{L^{2,s+\frac{1}{2}}_\rho}^2 \lesssim \int_{0}^{\infty} \tau^2(\frac{\lambda^2(\sigma)}{\lambda^2(\tau)})^{2s+2+\frac{\beta}{2}} \left<\xi\right>^{2s}|y(\sigma,\xi)|^2\rho(\xi)d\xi.
		\end{equation*}
		This gives the desired bound of $H(\tau,\sigma;\cdot)$, where one can choose $C=(2s+2+\frac{\beta}{2})(1+\nu^{-1})$.
		
		\par To get the bound on $\tau$-weighted spaces, we use (\ref{eq:x}) and take $N$ to be large enough:
		\begin{equation*}
			\begin{aligned}
				\tau^{N-2}\|x(\tau)\|_{L^{2,s+\frac{1}{2}}_\rho}&\lesssim \int_{\tau}^{\infty}\tau^{N-1-C}\sigma^{C-N}\left(\sigma^N\|y(\sigma)\|_{L^{2,s}_{\rho}}\right) d\sigma\\
				&\lesssim \frac{1}{N-C-1}\|y\|_{L^{\infty,N}L^{2,s}_{\rho}}\lesssim \frac{1}{N}\|y\|_{L^{\infty,N}L^{2,s}_{\rho}}.\\
			\end{aligned}
		\end{equation*}
		
		\par The statement of $\widehat{\mathcal{D}}H(\tau,\sigma,\cdot)$ can be checked similarly via applying $\partial_\tau-2\frac{\lambda_{\tau}}{\lambda}\xi\partial_\xi-\frac{\lambda_{\tau}}{\lambda}$ onto the expression of $x_d$.
	\end{proof}

	\par For $\xi=0$ part, we can also explicitly solve (\ref{eq:final}) as
	\begin{equation*}
		x_0(\tau)=\int_{\tau}^{\infty}\lambda(\tau)(\omega(\sigma)-\omega(\tau))y_0(\sigma)d\sigma,
	\end{equation*}
	with the property that for $N$ large enough,
	\begin{equation*}
		\tau^{N-2}|x_0(\tau)|\lesssim \int_{0}^{\infty}\tau^{N-1+\nu^{-1}}\tau^{-\nu^{-1}}\sigma^{-N}\|y_0\|_{L^{\infty,N}}d\sigma\lesssim \frac{1}{N}\|y_0\|_{L^{\infty,N}}.
	\end{equation*}
	
	\begin{proposition}
		Let $\underline{\mathbf{x}}$ and $\underline{\mathbf{y}}$ be as in (\ref{eq:final}), then we have for $N$ large enough that
		\begin{equation*}
			\|x_0\|_{L^\infty,N-2}\leq \frac{C}{N}\|y_0\|_{L^{\infty,N}},
		\end{equation*}
		and
		\begin{equation*}
			\|\widehat{\mathcal{D}}x_0\|_{L^{\infty,N-1}}\leq \frac{C}{N}\|y_0\|_{L^{\infty,N}}.
		\end{equation*}
		where $C$ is a constant independent of $N$.
	\end{proposition}
	
	\begin{proof}
		The bounds of $x_0$ has been proved. The bound of $\widehat{\mathcal{D}}x_0$ can be checked directly via applying $\partial_\tau-\frac{\lambda_{\tau}}{\lambda}$ onto the expression of $x_0$.
	\end{proof}
	
	\par For $\xi=\xi_d$ part, it seems that we can't explicitly solve (\ref{eq:final}), but we may put some terms from LHS of (\ref{eq:final}) into the RHS to get
	\begin{equation}\label{eq:x_d}
		x_d(\tau) = \int_{\tau}^{\infty}\frac{|\xi_d|^{-\frac{1}{2}}}{2}e^{-|\xi_d|^{\frac{1}{2}}(\sigma-\tau)} \widetilde{y_d}(\sigma)d\sigma + \int_{1}^{\tau}\frac{|\xi_d|^{-\frac{1}{2}}}{2}e^{-|\xi_d|^{\frac{1}{2}}(\tau-\sigma)} \widetilde{y_d}(\sigma)d\sigma,
	\end{equation}
	where
	\begin{equation*}
		\widetilde{y_d}(\tau)=y_d(\tau)-\partial_\tau(\frac{\lambda_{\tau}}{\lambda})x_d(\tau)-\frac{\lambda_{\tau}}{\lambda}\partial_\tau x_d(\tau).
	\end{equation*}
	For the first term in the RHS of (\ref{eq:x_d}),
	\begin{equation*}
		\tau^N\left|\int_{\tau}^{\infty}\frac{|\xi_d|^{-\frac{1}{2}}}{2}e^{-|\xi_d|^{\frac{1}{2}}(\sigma-\tau)} \widetilde{y_d}(\sigma)d\sigma\right|\lesssim \int_{\tau}^{\infty}\tau^N\sigma^{-N}e^{-|\xi_d|^{\frac{1}{2}}(\sigma-\tau)}\|\widetilde{y_d}\|_{L^{\infty,N}}d\sigma\lesssim \|\widetilde{y_d}\|_{L^{\infty,N}}.
	\end{equation*}
	For the second term in the RHS of (\ref{eq:x_d}), we split it further into two parts:
	\begin{equation*}
		\tau^N\left|\int_{1}^{\frac{\tau}{2}}\frac{|\xi_d|^{-\frac{1}{2}}}{2}e^{-|\xi_d|^{\frac{1}{2}}(\sigma-\tau)} \widetilde{y_d}(\sigma)d\sigma\right|\lesssim \int_{1}^{\frac{\tau}{2}}\tau^{N}e^{-|\xi_d|^{\frac{1}{2}}\tau/2}\|\widetilde{y_d}\|_{L^{\infty,N}}d\sigma\lesssim \|\widetilde{y_d}\|_{L^{\infty,N}},
	\end{equation*}
	and
	\begin{equation*}
		\tau^N\left|\int_{\frac{\tau}{2}}^{\tau}\frac{|\xi_d|^{-\frac{1}{2}}}{2}e^{-|\xi_d|^{\frac{1}{2}}(\sigma-\tau)} \widetilde{y_d}(\sigma)d\sigma\right|\lesssim\int_{\frac{\tau}{2}}^{\tau}2^N e^{-|\xi_d|^{\frac{1}{2}}(\tau-\sigma)}\|\widetilde{y_d}\|_{L^{\infty,N}}d\sigma\lesssim 2^N\|\widetilde{y_d}\|_{L^{\infty,N}}.
	\end{equation*}
	As a conclusion, we obtain
	\begin{equation*}
		\tau^N|x_d(\tau)|\leq C_N\|\widetilde{y_d}\|_{L^{\infty,N}},
	\end{equation*}
	where $C_N$ is a constant depending on $N$ (increases as $N$ increases).
	
	\begin{proposition}
		Let $\underline{\mathbf{x}}$ and $\underline{\mathbf{y}}$ be as in (\ref{eq:final}), then we have
		\begin{equation*}
			\|x_d\|_{L^{\infty,N}}\leq C_N'\left(\|y_d\|_{L^{\infty,N}}+\|x_d\|_{L^{\infty,N-2}}+\|\widehat{\mathcal{D}}x_d\|_{L^{\infty,N-1}}\right),
		\end{equation*}
		and
		\begin{equation*}
			\|\widehat{\mathcal{D}}x_d\|_{L^{\infty,N}}\leq C_N'\left(\|y_d\|_{L^{\infty,N}}+\|x_d\|_{L^{\infty,N-2}}+\|\widehat{\mathcal{D}}x_d\|_{L^{\infty,N-1}}\right),
		\end{equation*}
		where $C_N'$ is a constant depending on $N$ (increases as $N$ increases).
		\par Hence, by taking $\tau_0$ in the definition of space $L^{\infty,N}L^{2,s}$ large enough (increases as $N$ increases), we have
		\begin{equation*}
			\|x_d\|_{L^{\infty,N-2}}+\|\widehat{\mathcal{D}}x_d\|_{L^{\infty,N-1}}\leq \frac{C}{N}\|y_d\|_{L^{\infty,N}},
		\end{equation*}
		where $C$ is a constant independent of $N$.
	\end{proposition}
	
	\begin{proof}
		Use the facts that
		\begin{equation*}
			\|\widetilde{y_d}\|_{L^{\infty,N}}\lesssim \|y_d\|_{L^{\infty,N}}+\|x_d\|_{L^{\infty,N-2}}+\|\widehat{\mathcal{D}}x_d\|_{L^{\infty,N-1}}
		\end{equation*}
		and
		\begin{equation*}
			\|x_d\|_{L^{\infty,N-1}}\leq \tau_0^{-1}\|x_d\|_{L^{\infty,N}}.
		\end{equation*}
		\par The bound of $\widehat{\mathcal{D}}x_d$ can be checked directly via applying $\partial_\tau-\frac{\lambda_{\tau}}{\lambda}$ onto the expression of $x_d$.
	\end{proof}

	\begin{remark}
		It's an interesting open problem to solve $\underline{\mathbf{x}}(\tau)$ with prescribed scattering data at $\tau=+\infty$, which is impossible in our framework since any non-trivial free wave won't live in the $\tau$-weighted space we are going to introduce for $N$ very large.
	\end{remark}

	\par Now we denote by $\mathcal{H}$ the solution operator (parametrix) for the system (\ref{eq:final}). As a conclusion, we obtain the following.
	
	\begin{proposition}\label{prop:sol_op}
		Let $s\geq0$ and $N$ be large enough. Then there exists $\tau_0=\tau_0(N)\geq1$ (increases as $N$ increases) depending on $N$, such that
		\begin{equation*}
			\|\mathcal{H}\underline{\mathbf{y}}\|_{L^{\infty,N-2}L^{2,s+1/2}_\rho}+\left\|\widehat{\mathcal{D}}\mathcal{H}\underline{\mathbf{y}}\right\|_{L^{\infty,N-1}L^{2,s}_\rho}\leq \frac{C}{N}\|\underline{\mathbf{y}}\|_{L^{\infty,N}L^{2,s}_\rho},
		\end{equation*}
			where the constant $C$ is independent of $N$ (but may depend on $s,\nu$ and $\beta$).
	\end{proposition}

	\section{The nonlinear terms}\label{sec:nonlinear}
	
	To proceed the contraction mapping argument, we also need to study the mapping property of the nonlinear terms $\underline{\mathbf{x}}\mapsto\lambda^{-2}\mathcal{F}R\left(N_{2k-1}(R^{-1}\mathcal{F}^{-1}\underline{\mathbf{x}})\right)$. In fact, there is a linear term $(u_{2k-1}^4-u_0^4)\widetilde{\varepsilon}$ also being included in $N_{2k-1}(\widetilde{\varepsilon})$ due to our notation. However, as $(u_{2k-1}-u_0)\in\frac{\lambda^{\frac{1}{2}}(t)}{(t\lambda(t))^2}S^{\frac{\beta+3}{2}}(\max\{R^{\frac{-\beta+3}{2}}(\log(R))^{p},R^{\frac{\beta-1}{2}}(\log(R))^{p'}\},\mathcal{Q}_\beta)$, this linear term gains a twice order decay $\tau^{-2}$ as $\tau\to\infty$, which is safe. For other terms, there are \textit{nonlinear}, which also guarantees the desired decay in $\tau$.
	
	\par Now we introduce the notion of $\alpha$-Sobolev space, which is the Sobolev space associated with $\mathcal{L}_\alpha=-\Delta+\frac{\alpha}{r^2}$ instead of $-\Delta$. For $\alpha>-\frac{1}{4}$, the operator $\mathcal{L}_\alpha$ is positive, so we define the homogeneous $s$-th order $\alpha$-Sobolev space $\dot{H}_\alpha^s(\mathbb{R}^3)$ for $s\geq 0$ to have the norm
	\begin{equation*}
		\|u\|_{\dot{H}_\alpha^s(\mathbb{R}^3)} :=\|\mathcal{L}_\alpha^{\frac{s}{2}} u\|_{L^2(\mathbb{R}^3)}.
	\end{equation*}
	Similarly, the $s$-th order inhomogeneous $\alpha$-Sobolev space is defined to have the norm
	\begin{equation*}
		\|u\|_{H^s_\alpha(\mathbb{R}^3)}:=\|(1+\mathcal{L}_\alpha)^{\frac{s}{2}}u\|_{L^2(\mathbb{R}^3)}.
	\end{equation*}

	\begin{remark}
		Using the sharp hardy inequality
		\begin{equation*}
			\frac{1}{4}\int_{\mathbb{R}^3} \frac{u^2}{r^2} dx \leq \int_{\mathbb{R}^3} |\nabla u|^2 dx,\quad u\in C^\infty_0(\mathbb{R}^3),
		\end{equation*}
		  one immediately see that $\dot{H}_\alpha^1(\mathbb{R}^3)=\dot{H}^1(\mathbb{R}^3)$ and the two norms are equivalent. More generally, we have the following
	\end{remark}

	\begin{lemma}[equivalence of Sobolev spaces]\label{lem:equiv2}
		For $\alpha>-\frac{1}{4}$ and $0\leq s<\min\{1+\frac{\beta}{2},\frac{3}{2}\}$, we have 
		\begin{equation*}
		\|u\|_{\dot{H}_\alpha^s(\mathbb{R}^3)} \approx \|u\|_{\dot{H}^s(\mathbb{R}^3)},
		\end{equation*}
		and 
		\begin{equation*}
		\|u\|_{H_\alpha^s(\mathbb{R}^3)} \approx \|u\|_{H^s(\mathbb{R}^3)},
		\end{equation*}
		for $u\in C^\infty_0(\mathbb{R}^3)$. As a result, we have $\dot{H}_\alpha^s(\mathbb{R}^3)=\dot{H}^s(\mathbb{R}^3)$ and $H_\alpha^s(\mathbb{R}^3)=H^s(\mathbb{R}^3)$ for $0\leq s<\min\{1+\frac{\beta}{2},\frac{3}{2}\}$.
	\end{lemma}
	
	\begin{proof}
		The statement of homogeneous Sobolev spaces is exactly \cite[Theorem 1.2]{KMVZZ18}. The statement of inhomogenous Sobolev spaces follows from the homogeneous one and the fact that
		\begin{equation*}
			H_\alpha^s(\mathbb{R}^3)=\dot{H}_\alpha^s(\mathbb{R}^3)\cap L^2(\mathbb{R}^3),\quad H^s(\mathbb{R}^3)=\dot{H}^s(\mathbb{R}^3)\cap L^2(\mathbb{R}^3)
		\end{equation*}
		with equivalent norms.
	\end{proof}

	\par As being said, we first relate the $L^{2,s}_{d\rho}$ spaces in frequency space to the $\alpha$-Sobolev spaces $H^{2s}_{\alpha}(\mathbb{R}^3)$ in physics space.
	
	\begin{lemma}\label{lem:equiv1}
		For any $\alpha>-\frac{1}{4}+\frac{1}{25}$ and $0\leq s\leq 1$, we have
		\begin{equation*}
			\|\underline{\mathbf{x}}\|_{L^{2,s}_{d\rho}}\approx \|R^{-1}\mathcal{F}^{-1}\underline{\mathbf{x}}\|_{H^{2s}_{\alpha}(\mathbb{R}^3)}.
		\end{equation*}
	\end{lemma}

	\begin{proof}
		If $\alpha=0$, this has been proved in \cite{KST09a} (without restriction on $s$), so in the following we assume $\alpha\neq0$. Also, $s=0$ case is obvious.
		\par Since $\mathcal{F}\colon L^2((0,+\infty))\to L^2(\mathbb{R},d\rho)$ is an unitary operator, we have for $s=1$ that
		\begin{equation*}
			\begin{aligned}
				\|\underline{\mathbf{x}}\|_{L^{2,1}_{d\rho}}\approx \|\underline{\mathbf{x}}\|_{L^2_{d\rho}}+\|\xi\underline{\mathbf{x}}\|_{L^2_{d\rho}}&=\|\mathcal{F}^{-1}\underline{\mathbf{x}}\|_{L^2_{R}}+\|\mathcal{L}\mathcal{F}^{-1}\underline{\mathbf{x}}\|_{L^2_{R}}\\&\approx\|R^{-1}\mathcal{F}^{-1}\underline{\mathbf{x}}\|_{L^2(\mathbb{R}^3)}+\|R^{-1}\mathcal{L}\mathcal{F}^{-1}\underline{\mathbf{x}}\|_{L^2(\mathbb{R}^3)}.
			\end{aligned}
		\end{equation*}
		Let
		\begin{equation*}
			\widetilde{\mathcal{L}}=R^{-1}\mathcal{L}R=-\partial_R^2-\frac{2}{R}\partial_R+\frac{\alpha}{R^2}-5W_\alpha^4(R)=\mathcal{L}_\alpha-\frac{15\beta^2 R^{2\beta-2}}{(1-R^{2\beta})^2},
		\end{equation*}
		where $\mathcal{L}_\alpha=-\Delta_{\mathbb{R}^3}+\frac{\alpha}{R^2}$ (when applied to radial functions in $\mathbb{R}^3$), then we have
		\begin{equation*}
			\|\underline{\mathbf{x}}\|_{L^{2,1}_{d\rho}}\approx \|R^{-1}\mathcal{F}^{-1}\underline{\mathbf{x}}\|_{L^2(\mathbb{R}^3)}+\|\widetilde{\mathcal{L}}R^{-1}\mathcal{F}^{-1}\underline{\mathbf{x}}\|_{L^2(\mathbb{R}^3)}.
		\end{equation*}
		\par If $\beta>1$, then $\left|\frac{15\beta^2 R^{2\beta-2}}{(1-R^{2\beta})^2}\right|\lesssim 1$, hence
		\begin{equation*}
			\|(\widetilde{\mathcal{L}}-\mathcal{L}_\alpha)R^{-1}\mathcal{F}^{-1}\underline{\mathbf{x}}\|_{L^2(\mathbb{R}^3)}\lesssim \|R^{-1}\mathcal{F}^{-1}\underline{\mathbf{x}}\|_{L^2(\mathbb{R}^3)},
		\end{equation*}
		from which the desired equivalence follows.
		\par If $\frac{2}{5}<\beta<1$, $\left|\frac{15\beta^2 R^{2\beta-2}}{(1-R^{2\beta})^2}\right|\lesssim |U(R)|$ (recall that $U(R)\in S^{2\beta-2}(R^{-2\beta-2})$), then by Hardy's inequality for $\mathcal{L}_\alpha$ (see \cite[Proposition 3.2]{KMVZZ18}, here we need $\beta>\frac{2}{5}$), we have
		\begin{equation*}
			\|(\widetilde{\mathcal{L}}-\mathcal{L}_\alpha)R^{-1}\mathcal{F}^{-1}\underline{\mathbf{x}}\|_{L^2(\mathbb{R}^3)}\lesssim \|R^{2\beta-2}R^{-1}\mathcal{F}^{-1}\underline{\mathbf{x}}\|_{L^2(\mathbb{R}^3)}\lesssim \|\mathcal{L}_\alpha^{1-\beta}R^{-1}\mathcal{F}^{-1}\underline{\mathbf{x}}\|_{L^2(\mathbb{R}^3)}.
		\end{equation*}
		This proves
		\begin{equation*}
			\|\underline{\mathbf{x}}\|_{L^{2,1}_{d\rho}}\lesssim \|R^{-1}\mathcal{F}^{-1}\underline{\mathbf{x}}\|_{H^{2}_\alpha(\mathbb{R}^3)}.
		\end{equation*}
		On the otherhand, 
		\begin{equation*}
			\|(\widetilde{\mathcal{L}}-\mathcal{L}_\alpha)R^{-1}\mathcal{F}^{-1}\underline{\mathbf{x}}\|_{L^2(\mathbb{R}^3)}\lesssim \|U(R)\phi_d(R)\|_{L^2_R}|x_d|+\int_{0}^{\infty}\|U(R)\phi(R,\xi)\|_{L^2_R}|x(\xi)|\rho(\xi)d\xi.
		\end{equation*}
		Since $\phi_d(R)\sim R^{\frac{\beta+1}{2}}$ as $R\to 0$ and has exponential decay as $R\to\infty$, $U(R)\phi_d(R)$ is square-integrable when $\beta>\frac{2}{5}$. For $\xi\geq 1$, we estimate $\|U(R)\phi(R,\xi)\|_{L^2_R}$ as follows
		\begin{equation*}
			\begin{aligned}
				\|U(R)\phi(R,\xi)\|_{L^2_R}^2&\lesssim \int_{0}^{\xi^{-\frac{1}{2}}}(R^{2\beta-2}R^{\frac{\beta+1}{2}})^2dR + \int_{\xi^{-\frac{1}{2}}}^{1}(R^{2\beta-2}\xi^{-\frac{\beta+1}{4}})^2dR + \int_{1}^{\infty}(R^{-2\beta-2}\xi^{-\frac{\beta+1}{4}})^2dR\\
				&\lesssim \xi^{\frac{2-5\beta}{2}}+\xi^{-\frac{\beta+1}{2}}+\xi^{-\frac{7}{8}}\log(\xi).\\
			\end{aligned}
		\end{equation*}
		Hence, by Cauchy-Schwarz inequality we get
		\begin{equation*}
			\int_{1}^{\infty}\|U(R)\phi(R,\xi)\|_{L^2_R}|x(\xi)|\rho(\xi)d\xi\lesssim \left(\int_{1}^{\infty}|x(\xi)|^2\xi^{2}\rho(\xi)d\xi\right)^{\frac{1}{2}}.
		\end{equation*}
		For $\xi<1$, we have
		\begin{equation*}
			\begin{aligned}
				\|U(R)\phi(R,\xi)\|^2_{L^2_R}&\lesssim \int_{0}^{1}(R^{2\beta-2}R^{\frac{\beta+1}{2}})^2dR + \int_{1}^{\xi^{-\frac{1}{2}}}(R^{-2\beta-2}R^{\frac{-\beta+1}{2}})^2dR + \int_{1}^{\infty}(R^{-2\beta-2}\xi^{\frac{\beta-1}{4}})^2dR\\
				&\lesssim \xi^{\frac{\beta-1}{2}}.\\
			\end{aligned}
		\end{equation*}
		Hence, by Cauchy-Schwarz inequality we get
		\begin{equation*}
			\int_{0}^{1}\|U(R)\phi(R,\xi)\|_{L^2_R}|x(\xi)|\rho(\xi)d\xi\lesssim \left(\int_{0}^{1}|x(\xi)|^2\rho(\xi)d\xi\right)^{\frac{1}{2}}.
		\end{equation*}
		In conclusion, we have
		\begin{equation*}
			\|(\widetilde{\mathcal{L}}-\mathcal{L}_\alpha)R^{-1}\mathcal{F}^{-1}\underline{\mathbf{x}}\|_{L^2(\mathbb{R}^3)}\lesssim \|\underline{\mathbf{x}}\|_{L^{2,1}_{d\rho}},
		\end{equation*}
		which implies
		\begin{equation*}
			\|R^{-1}\mathcal{F}^{-1}\underline{\mathbf{x}}\|_{H^2_\alpha(\mathbb{R}^3)}\lesssim \|\underline{\mathbf{x}}\|_{L^{2,1}_{d\rho}},
		\end{equation*}
		and concludes the proof for $s=1$.
		\par For non-integer $0<s<1$, we use interpolation. First, we consider the map $\underline{\mathbf{x}}\mapsto R^{-1}\mathcal{F}^{-1}\underline{\mathbf{x}}$ and obtain the bound
		\begin{equation*}
			\|R^{-1}\mathcal{F}^{-1}\underline{\mathbf{x}}\|_{H^{2s}_\alpha(\mathbb{R}^3)}\lesssim \|\underline{\mathbf{x}}\|_{L^{2,s}_{d\rho}}.
		\end{equation*}
		The converse bound is obtained by considering the map $u\mapsto \mathcal{F}RS(u)$, where $S(u)$ denotes the spherical average of $u$ in $\mathbb{R}^3$.
	\end{proof}

	\begin{remark}
		The restrictions $\alpha>-\frac{1}{4}+\frac{1}{25}$ and $0\leq s\leq 1$ originate in the avalibility of Hardy's inequality for $\mathcal{L}_\alpha$. By taking $\beta$ large enough, there will be more $s$ admissble. However, it's enough for us to only consider $0\leq s\leq 1$.
	\end{remark}

	\begin{proposition}\label{prop:nonlinear}
		Let $\alpha>-\frac{1}{4}+\frac{1}{16}$ and $N\geq0$ be large enough. Then the map
		\begin{equation*}
			\underline{\mathbf{x}}\mapsto \lambda^{-2}\mathcal{F}R\left(N_{2k-1}(R^{-1}\mathcal{F}^{-1}\underline{\mathbf{x}})\right)
		\end{equation*}
		is locally Lipschitz from $L^{\infty,N-2}L_\rho^{2,s+1/2}$ to $L^{\infty,N}L^{2,s}_\rho$ for $\frac{1}{8}<s<\min\{\frac{\beta}{4},\frac{1}{4}\}$. 
	\end{proposition}

	\begin{proof}
		Accroding to Lemma \ref{lem:equiv1} and Lemma \ref{lem:equiv2}, it suffices to prove that $v\mapsto \lambda^{-2}N_{2k-1}(v)$ is locally Lipschitz from $L^{\infty,N-2}H^{2s+1}$ to $L^{\infty,N}H^{2s}$, where
		\begin{equation*}
			N_{2k-1}(v) = 5(u_{2k-1}^4-u_0^4)v + 10u_{2k-1}^3v^2 + 10u_{2k-1}^2v^3 + 5u_{2k-1}v^4 + v^5.
		\end{equation*}
		\par First, we show that from $N_{2k-1}(\cdot)$ maps from $L^{\infty,N-2}H^{2s+1}$ to $L^{\infty,N}H^{2s}$. Since $u_{2k-1}(\tau,\cdot)$ may not belong to $H^{2s+1}(\mathbb{R}^3)$ for any $s\geq 0$, we localize our estimate into a single unit cube $Q$ in $\mathbb{R}^3$. Once we get the localized estimate
		\begin{equation*}
			\|N_{2k-1}(v)\|_{L^{\infty,N}H^{2s}(Q)} \leq C_Q \sum_{k=1}^{5} \|v\|_{L^{\infty,N-2}H^{2s+1}(Q)}^k
		\end{equation*}
		for each $Q$, where $C_Q\lesssim 1$ uniformly in $Q$, then taking the squared sum yields the desired result immediately.
		\par For the localized estimates, we shall take benefit of the following quintilinear estimate (it follows from the fractional Lebiniz rule \cite[page 105, Proposition 1.1]{Tay00} and Sobolev embedding): 
		\begin{equation*}
			H^{1+2s}\cdot H^{1+2s}\cdot H^{1+2s}\cdot H^{1+2s}\cdot H^{1+2s}\subset H^{2s},
		\end{equation*}
		which holds for $s>\frac{1}{8}$. 

		\par Now we have $u_{2k-1}|_Q\in H^{1+2s}(Q)$ for $s<\min\{\frac{\beta}{4},\frac{\beta\nu}{4}\}$ for all cubes $Q$, and
		\begin{equation*}
			\|u_{2k-1}\|_{H^{1+2s}(Q)}\lesssim \lambda^{\frac{1}{2}}(\tau)\max\{\mathrm{dist}(Q,0)^{\frac{-\beta-1}{2}},\mathrm{dist}(Q,0)^{\frac{\beta-5}{2}}\}|\log(\mathrm{dist}(Q,0))|^{p}
		\end{equation*} Similarly, we have $(u_{2k-1}-u_0)|_Q\in H^{1+2s}(Q)$ for $s<\min\{\frac{\beta}{4},\frac{\beta\nu}{4}\}$ for all cubes $Q$, and
		\begin{equation*}
			\|u_{2k-1}-u_0\|_{H^{1+2s}(Q)}\lesssim \frac{\lambda^{\frac{1}{2}}(\tau)}{\tau^2} \max\{\mathrm{dist}(Q,0)^{\frac{-\beta+3}{2}},\mathrm{dist}(Q,0)^{\frac{\beta-1}{2}}\}|\log(\mathrm{dist}(Q,0))|^{q}.
		\end{equation*}
		Therefore we need $\frac{\beta\nu}{4}>\frac{1}{8}$, i.e. $\nu>\frac{1}{2\beta}$ to invoke the quintilinear estimate.
		
		\par Then we show that $N_{2k-1}(\cdot)$ is locally Lipschitz, whose proof is almost the same as above.
	\end{proof}

	\begin{remark}
		Here we make the assumption that $\alpha>-\frac{1}{4}+\frac{1}{16}$ in order to ensure $\frac{1}{8}<\frac{\beta}{4}$. This restriction on $\alpha$ comes from the equivalence of Sobolev spaces $H_\alpha^s(\mathbb{R}^3)=H^s(\mathbb{R}^3)$. However, one may eliminate or improve this extra $\frac{1}{16}$ in the lower bound by analyzing directly with $H^s_\alpha(\mathbb{R}^3)$.
	\end{remark}

	\section{Proof of the main theorem}\label{sec:proof}

	Let $I\subset\mathbb{R}$ be a nonempty interval, define
	\begin{equation*}
		S(I):=L_{t,x}^8(I\times \mathbb{R}^3)\cap L_t^5L_x^{10}(I\times\mathbb{R}^3)
	\end{equation*}
	to be the Strichartz space and
	\begin{equation*}
		\|\cdot\|_{S(I)}:=\|\cdot\|_{L_{t,x}^8(I\times \mathbb{R}^3)}+\|\cdot\|_{L_t^5L_x^{10}(I\times\mathbb{R}^3)}
	\end{equation*}
	to be the Strichartz norm. Let $S(t)$ denote the solution operator for the free wave equation with an inverse-square potential, that is 
	\begin{equation*}
		S(t)(u_0,u_1):=\cos(t\sqrt{\mathcal{L}_\alpha})u_0+\frac{\sin(t\sqrt{\mathcal{L}_\alpha})}{\sqrt{\mathcal{L}_\alpha}}u_1.
	\end{equation*}
	
	\begin{theorem}[Strichartz estimates]\label{thm:stri}
		Let $\alpha>-\frac{1}{4}+\frac{1}{25}$. Then for $(u_0,u_1)\in \dot{H}_\alpha^1(\mathbb{R}^3)\times L^2(\mathbb{R}^3)$, we have
		\begin{equation*}
			\|S(t)(u_0,u_1)\|_{S(I)}\leq C\|(u_0,u_1)\|_{\dot{H}_\alpha^1(\mathbb{R}^3)\times L^2(\mathbb{R}^3)},
		\end{equation*}
		for some constant $C$ independent of the time interval $I$.
	\end{theorem}

	\begin{proof}
		See \cite[Proposition 2.5]{MMZ20}.
	\end{proof}

	\begin{theorem}[local well-posedness]\label{thm:well}
		Let $\alpha>-\frac{1}{4}+\frac{1}{25}$. Let $(u_0,u_1)\in \dot{H}_\alpha^1(\mathbb{R}^3)\times L^2(\mathbb{R}^3)$ satisfying $\|(u_0,u_1)\|_{\dot{H}_\alpha^1(\mathbb{R}^3)\times L^2(\mathbb{R}^3)}\leq A$. Then, there exists $\delta_0=\delta_0(A)>0$ such that if $I\subset\mathbb{R}$ and
		\begin{equation*}
			\|S(t)(u_0,u_1)\|_{S(I)}<\delta
		\end{equation*}
		for some $0<\delta<\delta_0$, then there exists a unique solution $u\in C(I,\dot{H}_\alpha^1(\mathbb{R}^3)\times L^2(\mathbb{R}^3))$ of (\ref{eq:NLW}) with initial data $(u_0,u_1)$, which satisfies
		\begin{equation*}
			\|u\|_{S(I)}\lesssim\delta.
		\end{equation*}
	\end{theorem}
	
	\begin{proof}
		See \cite[Proposition 2.7]{MMZ20}.
	\end{proof}

	\begin{remark}[small data global well-posedness]\label{rem:sdgw}
		Combining Theorem \ref{thm:stri} and Theorem \ref{thm:well}, we see that if $\|(u_0,u_1)\|_{\dot{H}_\alpha^1(\mathbb{R}^3)\times L^2(\mathbb{R}^3)}$ is small enough, the solution of (\ref{eq:NLW}) exists globally.
	\end{remark}

	\begin{remark}[persistence of regularity]
		One can show that if the initial data $(u_0,u_1)\in (\dot{H}^{1}_\alpha(\mathbb{R}^3)\cap \dot{H}^{1+\mu}_\alpha(\mathbb{R}^3))\times H^{\mu}_\alpha(\mathbb{R}^3)$ for some $0\leq\mu\leq 1$, then the solution remains in this space: $(u(t),\partial_t u(t))\in C(I,(\dot{H}^{1}_\alpha(\mathbb{R}^3)\cap \dot{H}^{1+\mu}_\alpha(\mathbb{R}^3))\times H^{\mu}_\alpha(\mathbb{R}^3))$.
	\end{remark}

	To end this section and the whole paper, we put all the ingredients established in the above sections together to give a proof of Theorem \ref{thm:main}.
	\begin{proof}[Proof of Theorem \ref{thm:main}]
		Combining Proposition \ref{prop:sol_op} and Proposition \ref{prop:nonlinear}, the contraction mapping principle gives $\underline{\mathbf{x}}\in L^{\infty,N-2}L^{2,s+\frac{1}{2}}_\rho$ satisfying the desired equation. Let $v=R^{-1}\mathcal{F}^{-1}\underline{\mathbf{x}}\in L^{\infty,N-2}H^{2s+1}_\alpha$ and $u=u_{2k-1}+v$, then $u$ satisfies the nonlinear wave equation inside the light-cone. Given any $\varepsilon>0$, by taking $t_0$ small enough, we can assume
		\begin{equation*}
			\int_{K_t} \left((\partial_t u-\partial_t u_0)^2 + |\nabla u-\nabla u_0|^2 + \frac{\alpha}{r^2} (u-u_0)^2 + (u-u_0)^6\right) dx\leq \delta,
		\end{equation*}
		and
		\begin{equation*}
			\int_{K_t^c} \left((\partial_t u)^2 + |\nabla u|^2 + \frac{\alpha}{r^2} u^2 + u^6\right) dx \leq\delta
		\end{equation*}
		for any $0<t\leq t_0$.
		\par Accroding to Theorem \ref{thm:well}, the initial value problem (\ref{eq:NLW}) with $w(t_0)=u(t_0)$, $\partial_t w(t_0)=\partial_t u(t_0)$ admits a (backward) solution $w$ on $(T_-,t_0]$ for some $T_-<t_0$ and $w=u$ inside the light-cone due to the finite speed of propagation. 
		
		\par To show that $w$ blows up exactly at $(0,0)$, we first prove that the $\dot{H}^1_\alpha(\mathbb{R}^3)\times L^2(\mathbb{R}^3)$ norm of $w$ stays small outside the forward light-cone as $t\to 0^+$:
		\begin{equation*}
			\int_{K_t^c} \left((\partial_t w)^2 + |\nabla w|^2 +\frac{\alpha}{r^2} w^2 + w^6\right) dx \lesssim \delta
		\end{equation*}
		for all $0<t\leq t_0$. For this, we use the energy conservation and a bootstrap argument. The conserved energy of $w$ is given by
		\begin{equation*}
			\mathcal{E}[w](t):=\int_{\mathbb{R}^3} \left(\frac{1}{2}(\partial_t w)^2 + \frac{1}{2}|\nabla w|^2 +\frac{\alpha}{2r^2}w^2 -\frac{w^6}{6}\right) dx= \mathcal{E}[w](t_0)=\mathcal{E}[W_\alpha]+O(\delta).
		\end{equation*}
		From above, we have the smallness of energy outside the light-cone:
		\begin{equation*}
			\left|\int_{K_t^c} \left(\frac{1}{2}(\partial_t w)^2 + \frac{1}{2}|\nabla w|^2 +\frac{\alpha}{2r^2} w^2 -\frac{w^6}{6}\right) dx \right|\lesssim \delta
		\end{equation*}
		for all $0<t\leq t_0$. Together with the Sobolev embedding $\dot{H}^1_\alpha=\dot{H}^1 \hookrightarrow L^6$ (with an absolute constant independent of $t$)
		\begin{equation*}
			\int_{K_t^c} w^6 dx\lesssim \left(\int_{K_t^c} \left(|\nabla w|^2 + \frac{\alpha}{r^2} w^2\right) dx\right)^3,
		\end{equation*}
		we get there is a (small) constant $C_0$ such that if 
		\begin{equation*}
			\int_{K_t^c} \left(\frac{1}{2}(\partial_t w)^2 + \frac{1}{2}|\nabla w|^2 + \frac{\alpha}{2r^2} w^2\right) dx \leq C_0,
		\end{equation*}
		then it actually $\lesssim \delta$ (we may assume $\delta\ll C_0$). This is true at time $t_0$, hence a continuity argument applied to the continuous function
		$$t\mapsto \int_{K_t^c} \left(\frac{1}{2}(\partial_t w)^2 + \frac{1}{2}|\nabla w|^2 + \frac{\alpha}{2r^2} w^2\right) dx$$
		shows that it $\lesssim\delta$ for all $t\in (0,t_0]$.

		Then, we can show that $w$ exists up to $t=0$, i.e. $T_-=0$. Suppose $T_->0$, then let $t_->T_-$ be close enough to $T_-$ and let $\widetilde{w}$ be the solution of (\ref{eq:NLW}) with initial data $(\widetilde{w}(t_-),\partial_t\widetilde{w}(t_-))=((1-\chi(|x|/t_-))w(t_-),(1-\chi(|x|/t_-))\partial_t w(t_-))$, where $\chi(|x|/t_-)$ is a smooth cutoff function being $1$ near $|x|\leq\frac{1}{4}t_-$ and being $0$ near $|x|\geq\frac{1}{2}t_-$. Then the small-data-global-well-posedness theory (see Remark \ref{rem:sdgw}) implies that $\widetilde{w}$ is a global solution. Let $w'(t,x)=w(t,x)$ inside the light-cone and $w'(t,x)=\widetilde{w}(t,x)$ on $\{|x|>\frac{3}{4}t\}$, then $w'$ pastes into a solution in a small neighborhood of $\{t_-\}\times\{|x|>\frac{3}{4}t_-\}$ and $w'=w$ in this neighborhood due to the finite speed of propagation. This gives an  extends of $w$ across $T_-$ (we take $t_--T_-$ to be small enough), a contradiction!
	\end{proof}

	\appendix
	\section{A kind of power series}\label{app:A}
	In Section \ref{sec:appro}, we frequently use the following type of ``power series":
	\begin{equation*}
		f(R)=R^\delta\sum_{i_1,\cdots,i_N=0}^{\infty}a_{i_1,\cdots,i_N}R^{i_1p_1+\cdots+i_Np_N},\quad R < R_0,
	\end{equation*}
	where $N\geq 1$ is a finite number and $p_1,\cdots,p_N$'s are finitely many ``base powers". This type of series works just like power series, for which $N=1$ and $p_1=1$. In fact, define 
	\begin{equation*}
		g(R_1,\cdots,R_N)=\sum_{i_1,\cdots,i_N=0}^{\infty}a_{i_1,\cdots,i_N}R_1^{i_1}\cdots R_N^{i_N},
	\end{equation*}
	which is a multi-variable power series, then
	\begin{equation*}
		f(R)=R^{\delta}g(R^{p_1},\cdots,R^{p_N}).
	\end{equation*}

	\bigskip
	\textbf{Acknowledgement.} The author would like to thank Lifeng Zhao for introducing this interesting problem and so many useful discussions.

\bibliography{ref_paper}
\bibliographystyle{alpha}

\end{document}